\documentclass[11pt,reqno]{amsart}
\usepackage{subfiles}
\usepackage{fcprod}
\title{Products of manifolds with fibered corners}
\author{Chris Kottke}
\address{New College of Florida, Division of Natural Sciences}
\email{ckottke@ncf.edu}
\author{Fr\'{e}d\'{e}ric Rochon}
\address{D\' epartement de Math\'ematiques, Universit\'e du Qu\'ebec \`{a} Montr\'eal}
\email{rochon.frederic@uqam.ca}
\date{\today}
\begin{document}

\begin{abstract}
Manifolds with fibered corners arise as resolutions of stratified spaces, as `many body' compactifications of vector spaces, and as compactifications of certain moduli spaces including 
those of nonabelian Yang-Mills-Higgs monopoles, 
among other settings.
However, cartesian products of manifolds with fibered corners do not generally have fibered corners themselves,
and thus fail
to reflect the appropriate structure of products of the underlying spaces in the above settings.
Here we determine a resolution of the cartesian product of fibered corners manifolds by blow-up which we call the `ordered product',
which leads to a well-behaved category of fibered corners manifolds in which the ordered product satisfies the appropriate universal property.
In contrast to the usual category of manifolds with corners, this category of fibered corners not only has all finite products, but all finite
transverse fiber products as well, and we show in addition that 
the ordered product is a natural product for wedge (aka incomplete edge) metrics and quasi-fibered boundary metrics, a class which 
includes QAC and QALE metrics.
\end{abstract}

\maketitle
\tableofcontents

\section{Introduction} \label{S:intro}
There is a large body of work in geometric analysis in which non-compact and/or singular spaces are compactified and/or resolved to manifolds with corners, on which
detailed analysis, especially of asymptotic expansions at boundary faces, gives precise results not readily available by other means.
The manifolds with corners associated to \emph{products} of such spaces are naturally of interest, not least because the typical approach 
in such problems is to consider operators via their distributional Schwartz kernels. 
%
%
Especially when there are many boundary faces involved, the identification of the `correct' 
version of the product can be a subtle problem which may involve seemingly arbitrary choices, and can often seem somewhat ad hoc.
However, in many (if not most) of these problems, the manifolds with corners additionally admit a natural \emph{fibered corners} structure 
in which each boundary hypersurface is the total space of a fiber bundle comparable in a precise sense to its neighboring bundles at the corners. 
We show here that consideration of this additional structure leads to a remarkably well-behaved theory of products and fiber products.

Fibered corners structures arise in particular in two settings: the first setting is the resolution of \emph{(smoothly) stratified spaces} \cite{AM,ALMP,ALMP2,Albin}, which are often equipped with \emph{wedge (or `iterated incomplete edge') metrics}, Riemannian metrics degenerating conically along the strata in an iterated fashion. 
The second setting is typified in simplest form by \emph{many body spaces}, which are vector spaces which have been radially compactified and subsequently blown up along the boundaries of a family of linear subspaces \cite{Vasy,Kmb,AMN}. 
Euclidean metrics on the original vector spaces become \emph{quasi-asymptotically conic} (QAC) metrics on the associated many body spaces, and this asymptotic geometry may be generalized to the manifold setting in the form of QAC manifolds \cite{DM} and even more generally in the form of  `quasi-fibered boundary' (QFB, aka $\Phi$) manifolds \cite{CDR},
extending the \emph{scattering} and \emph{fibered boundary} structures of \cite{Msc} and \cite{MM} on manifolds with boundary, respectively, as well as the \emph{quasi-asymptotically locally euclidean} (QALE) metrics introduced by Joyce \cite{JoyceQALE, JoyceBook,Carron}.
Important examples of QFB manifolds include the compactifications
of the hyperK\"ahler moduli spaces of $\SU(2)$ monopoles of arbitrary charge \cite{FKS}. 
Roughly speaking, spaces arising in the first setting are modelled (iteratively and in parameterized fashion) on the small ends of cones, while those arising in the second setting are modelled on the large ends.

While a fair amount has been written about manifolds with fibered corners, this article establishes some aspects of an associated \emph{category}, in particular the existence of products and fiber products in this category.
The results presented here will have applications to the construction of various pseudodifferential calculi associated to manifolds with fibered corners.
In particular, these results are used in an essential way in \cite{KR1} and \cite{KR2}, the authors' work on elliptic pseudodifferential operator analysis and $L^2$-cohomology of QFB manifolds.
In fact, it will be shown in a forthcoming work that a wide variety of the known pseudodifferential caculi adapted to various geometric settings
can be constructed within the category of manifolds with fibered corners introduced below, and that doing so simplifies and clarifies many of their features.
This suggests that the fibered corners category may be the appropriate setting in which to answer a question posed by Richard Melrose,
namely to axiomatize and subsequently classify
all `generalized products', meaning the sequences of spaces $X_1, X_2, X_3,\ldots$ and maps between these spaces which satisfy certain properties enjoyed by the projection maps and diagonal inclusions
when $X_1$, $X_2$, and $X_3$ are the respective single, double and triple spaces of a geometric pseudodifferential calculus.

\medskip
In more detail, a \emph{fibered corners structure}\footnote{variously known also as a `resolution structure', `iterated boundary fibration structure', or simply `iterated structure' \cite{DLR, AM, ALMP}}
on a compact (not necessarily connected) manifold with corners $X$ consists of a locally trivial fiber bundle structure $\bfib G : G \to \bfb G$ (hereafter simply called a fibration) on each boundary hypersurface $G$ (again not necessarily connected), the base $\bfb G$ and typical fiber $\bff G$ of which are also manifolds with corners (and in fact fibered corners), satisfying a comparability condition wherever two hypersurfaces intersect (see Definition~\ref{D:fc}).
%
Among other things this
determines a partial order on the set $\M 1(X)$ of boundary hypersurfaces, in which $G < G'$ whenever $G \cap G' \neq
\emptyset$ and $G$ has a strictly coarser fibration (meaning that $G$ is associated with a smaller stratum when $X$ is associated with a stratified space).
As a matter of notation, we write $G \sim G'$ when $G$ and $G'$ are comparable, meaning $G = G'$, $G < G'$, or $G' < G$.

A key observation, seemingly overlooked in the literature, is that this structure and associated partial order extends to the larger set $\Mtot(X) = \M 1(X) \cup \set{X}$ of \emph{principal faces}, meaning the boundary hypersurfaces along with the interior,
provided the interior itself is equipped with a fibration $\bfib X : X \to \bfb X$ similarly comparable to the $\bfib G$.
%
The two most common situations arise by taking one of the two trivial fibration structures on $X$:
\begin{itemize}
\item 
If $\bfb X = X$ and $\bfib X = \id$, making $X \in \Mtot(X)$ maximal in the order, we say $X$ is \emph{interior maximal}, denoted by $X = X_\tmax$, which
turns out to be natural in the setting of stratified spaces and wedge metrics.
\item
If $\bfb X = \pt$ is a single point and $\bfib X$ is the constant map, making $X \in \Mtot(X)$ minimal in the order, we say $X$ is \emph{interior minimal}, denoted by $X = X_\tmin$,
which turns out to be natural in the setting of many body spaces and QFB manifolds.
\end{itemize}
%
We treat the general case of an arbitrary fibration on $X$ below, obtaining results for interior maximal and interior minimal spaces by specialization.
Note that every manifold with fibered corners admits both wedge and QFB metrics.

Even with interior fibrations specified, the
cartesian product $X\times Y$ is not generally a manifold with fibered corners in any natural way, since the product fibrations are not generally comparable
where they meet.
Instead, the product in the fibered corners category is the \emph{ordered product}, defined for interior maximal manifolds with fibered corners by the iterated blow-up
\[
	X_\tmax\ttimes Y_\tmax = [X\times Y; \M1(X) \times \M 1(Y)]
\]
of all codimension 2 corners of the form $G\times H$, where $G \in \M 1(X)$ and $H \in \M 1(Y)$,
taken in any order consistent with the partial order on the product $\M1(X)\times \M1(Y)$; in other words $G' \times H'$ must be blown-up prior to $G\times H$
whenever $G' < G$ and $H' < H$, and any order satisfying this condition leads the same space up to natural diffeomorphism.

On the other hand, for interior minimal manifolds with fibered corners, the ordered product is defined by the opposite order blow-up
\[
	X_\tmin\ttimes Y_\tmin = [X\times Y; \rM 1(X)\times \rM 1(Y)]
\]
where $\rM 1(X)$ denotes the partially ordered set $\M 1(X)$ with order reversed; 
in other words $G' \times H'$ must be blown up prior to $G\times H$ whenever $G' > G$ and $H' > H$.
In general, these are not diffeomorphic, i.e.,
$
	X_\tmax \ttimes Y_\tmax \not\cong X_\tmin \ttimes Y_\tmin,
$
whenever one of the factors has a corner of codimension at least 2.
Both products are in fact special cases of the general ordered product
\begin{equation}
	X\ttimes Y = [X\times Y; \M1^<(X)\times \M 1^<(Y), \rM1^>(X)\times \rM1^>(Y)],
	\label{E:gen_ord_prod}
\end{equation}
where $\M 1^\gtrless(X) = \set{G \in \M 1(X) : G \gtrless X}$.
Our main results about the ordered product are summarized as follows.
\begin{thm*}
[Theorems~\ref{T:fibcorn_product}, \ref{T:weak_category}, \ref{T:phi_metric}, \ref{T:phi_metric}]
For manifolds $X$ and $Y$ with fibered corners:
\begin{enumerate}
\item 
The ordered product $X\ttimes Y$ is naturally a manifold with fibered corners,
%
with principal faces $\tlift G H \in \Mtot(X\ttimes Y)$ identified with those pairs $(G,H) \in \Mtot(X)\times \Mtot(Y)$ which are comparable to $(X,Y)$
in the product order, with $\tlift G H < \tlift {G'}{H'}$ if $(G,H) < (G',H')$.
In particular $\tlift{G}{H}$ and $\tlift{G'}{H'}$ are disjoint if $(G,H) \not\sim (G',H')$.
\item 
For interior maximal spaces, the fibrations of $X_\tmax \ttimes Y_\tmax$ have the form
\begin{equation}
\begin{tikzcd}[column sep=small]
	\bff H \ar[r, -]  &\tlift{X}{ H} \ar[d, "\phi_{X,H}"] \\ 
	& X \ttimes \bfb H,
\end{tikzcd}
\quad 
\begin{tikzcd}[column sep=small]
	\bff G \ar[r, -]  &\tlift{G}{Y} \ar[d, "\phi_{G,Y}"] \\ 
	& \bfb G \ttimes Y,
\end{tikzcd}
\quad \text{or}
\begin{tikzcd}[column sep=small]
	\bff G\jtimes_\tmax \bff H \ar[r, -]  &\tlift{G}{ H} \ar[d, "\phi_{G,H}"] \\ 
	& \bfb G \ttimes \bfb H,
\end{tikzcd}
	\label{E:intro_bdy_fib_prod}
\end{equation}
for $H \in \Mtot(Y)$ and $G \in \Mtot(X)$ in the first and second cases, and 
$(G,H) \in \M 1(X)\times \M 1(Y)$ in the third case, and
where $\jtimes$ denotes the \emph{join product} discussed below (see also Definition~\ref{D:rel_join}). 
\item 
For interior minimal spaces, the fibrations of $X_\tmin \ttimes Y_\tmin$ have the form
\begin{equation}
\begin{tikzcd}[column sep=small]
	X\ttimes \bff H \ar[r, -]  &\tlift{X}{ H} \ar[d, "\phi_{X,H}"] \\ 
	& \bfb H,
\end{tikzcd}
\quad 
\begin{tikzcd}[column sep=small]
	\bff G\ttimes Y \ar[r, -]  &\tlift{G}{Y} \ar[d, "\phi_{G,Y}"] \\ 
	& \bfb G,
\end{tikzcd}
\quad \text{or}
\begin{tikzcd}[column sep=small]
	\bff G\ttimes \bff H \ar[r, -]  &\tlift{G}{ H} \ar[d, "\phi_{G,H}"] \\ 
	& \bfb G \jtimes_\tmin \bfb H.
\end{tikzcd}
	\label{E:intro_bdy_fib_rev}
\end{equation}
For the structure of the fibrations on a general ordered product 
\eqref{E:gen_ord_prod}
specializing to \eqref{E:intro_bdy_fib_prod} and \eqref{E:intro_bdy_fib_rev}, we refer to 
Theorem~\ref{T:fibcorn_product} below.  
\item\label{I:cat_prod_intro}
$X\ttimes Y$ satisfies the universal property of the product in a category of manifolds with fibered corners and appropriate morphisms (see below),
namely, any morphisms $f : W \to X$ and $g : W \to Y$ factor through a unique morphism $W \to X\ttimes Y$ forming a commutative diagram
\[
\begin{tikzcd}
	& W \ar[d, dashed, "\exists !"] \ar[dl, swap, "f"] \ar[dr, "g"] &
\\ 	X & X\ttimes Y \ar[l] \ar[r] & Y
\end{tikzcd}
\]
with the lifted projections $\pi_X : X\ttimes Y \to X$ and $\pi_Y : X\ttimes Y \to Y$.
In particular, as a consequence of this universal property, the ordered product is associative and commutative up to unique isomorphism.
\item 
If $g_X$ and $g_Y$ are wedge metrics (resp.\ QFB metrics) on $X_\tmax$ and $Y_\tmax$ (resp. $X_\tmin$ and $Y_\tmin$), then $g_X+ g_Y$ lifts to a wedge metric (resp.\ QFB metric) on $X_\tmax \ttimes Y_\tmax$ (resp.\ $X_\tmin \ttimes Y_\tmin$).
\end{enumerate}
\end{thm*}
\noindent
Some remarks:
\begin{itemize}
\item The fibered corners structure on $X$ induces a fibered corners structure on each boundary hypersurface $G$ as well as its fiber $\bff G$ and base $\bfb G$,
with respect to which $\bff G$ is interior minimal and $\bfb G$ is interior maximal.
In particular, the base spaces in \eqref{E:intro_bdy_fib_prod} are ordered products of interior maximal manifolds, while the fibers in \eqref{E:intro_bdy_fib_rev}
are ordered products of interior minimal manifolds.
\item
We do not assume that the spaces $X$, $G$, $\bff G$ or $\bfb G$ are connected; in particular, it often arises in practice that $\bfib G : G \to \bfb G$ is a fibration with fiber $\bff G$ having 
disjoint components.
\item The \emph{join product} $\jtimes$ appearing in the fibers in \eqref{E:intro_bdy_fib_prod}, so named 
for the fact that its associated stratified space is the topological join of the stratified spaces of its factors (see Corollary~\ref{C:join}), may be identified as the iterated blow-up
\begin{equation}
\begin{aligned}
	\bff G\jtimes_\tmax \bff H = [\bff G\times \bff H \times I;\; &\M 1(\bff G)\times \bff H \times \set 0, 
		\\&\bff G\times \M 1(\bff H)\times \set 1, \\&\M 1(\bff G)\times \M 1(\bff H)\times I],
	\quad I = [0,1],
\end{aligned}
	\label{E:intro_join_fiber}
\end{equation}
with blow-ups taken in any order consistent with the partial orders on $\M 1(\bff G)$, $\M 1(\bff H)$, and $\M 1(\bff G)\times \M 1(\bff H)$.
The join product appearing the bases in \eqref{E:intro_bdy_fib_rev} is given by a similar formula but with orders reversed:
\begin{equation}
\begin{aligned}
	\bfb G\jtimes_\tmin \bfb H = [\bfb G\times \bfb H \times I;\; &\rM 1(\bfb G)\times \bfb H \times \set 0, 
		\\&\bfb G\times \rM 1(\bfb H)\times \set 1, \\&\rM 1(\bfb G)\times \rM 1(\bfb H)\times I],
\end{aligned}
	\label{E:intro_join_base}
\end{equation}
In contrast to the ordered product, there is in fact a remarkably non-trivial diffeomorphism $X \jtimes_\tmax Y \cong X \jtimes_\tmin Y$ between the minimal
and maximal join products which is proved in \S\ref{S:equiv}, from which it follows that the associated stratified space to \eqref{E:intro_join_base} is
also the topological join of the stratified spaces of the factors.
A general version of the join (see Definition~\ref{D:rel_join}) which appears in the boundary fibrations for a general ordered product \eqref{E:gen_ord_prod} is somewhat more complicated and depends
on the place of $X$ in the order on $\Mtot(\bff G)$ and $\Mtot(\bfb G)$, viewed as ordered subsets of $\Mtot(X)$.
%
%
\item 
The boundary fibrations \eqref{E:intro_bdy_fib_rev} in the interior minimal case (and more generally whenever $X$ and $Y$ are not both maximal) depend 
on equivalence classes of boundary defining functions on $X$ and $Y$, where two defining functions are equivalent if their ratio is basic (i.e., constant on fibers)
over each boundary hypersurface (see Definition~\ref{D:basic_bdf}).
Such a choice of equivalence classes of boundary defining functions is therefore part of the data for $X$ and $Y$ as objects in the general category of manifolds with fibered corners (unless all spaces are taken to be interior maximal).
This equivalence relation also arises naturally in the consideration of QFB metrics (see \S\ref{S:geom}).
\item 
The morphisms referenced in part~\ref{I:cat_prod_intro} are taken to be those interior b-maps $f : X \to Y$ which are
\begin{itemize}
\item \emph{simple}, meaning all boundary exponents are $0$ or $1$,
\item \emph{b-normal}, in particular boundary hypersurfaces map either to boundary hypersurfaces of the target or to the interior, but never to faces of higher codimension, 
so $f$ determines a map $f_\sharp: \Mtot(X) \to \Mtot(Y)$ between sets of principal faces,
\item \emph{ordered}, meaning $f_\sharp : \Mtot(X) \to \Mtot(Y)$  is a map of ordered sets,
\item \emph{fibered}, meaning $f$ restricts over each $G \in \Mtot(X)$ to a map of fiber bundles
\[
\begin{tikzcd}
	G \ar[r, "f"] \ar[d, "\bfib G"] & H = f_\sharp(G) \ar[d, "\bfib H"]
	\\ \bfb G \ar[r, "f_G"] & \bfb H
\end{tikzcd}
\qquad \text{for some $f_G$}
\]
\item \emph{consistent} with the equivalence classes of boundary defining functions on $X$ and $Y$ (see Definition~\ref{D:basic_bdf} for a precise statement),
except in the case that all spaces are interior maximal.
\end{itemize}
\end{itemize}

In contrast to the usual category of manifolds with corners and b-maps\footnote{As shown in \cite{Joycegc}, the suitably transverse fiber product of manifolds with corners exists
as a manifold with `generalized corners', a larger category which has products and transverse fiber products, but not typically a manifold with ordinary corners.}, the fibered corners category has all transverse fiber products:
\begin{thm*}[Theorems~\ref{T:fib_prod}, \ref{T:bhs_of_fib_prod} and Proposition~\ref{P:fib_corn_fib_prod}]
If $f : X \to Z$ and $g : Y \to Z$ are fibered corners morphisms which are \emph{b-transverse},
meaning that 
$\bd f_\ast \bT_x X + \bd g_\ast \bT_y Y = \bT_z Z$ whenever $f(x) = g(y) =z$,
then
\[
	X\ttimes_Z Y = \overline{\set{(x,y) \in X^\circ\times Y^\circ : f(x) = g(y)}} \subset X\ttimes Y
\]
is a manifold with fibered corners, embedded as a p-submanifold in the ordered product and satisfying the universal property of the fiber product in the category
of fibered corners.
The fibrations on $X\ttimes_Z Y$, consistent with the restriction of from those of $X\ttimes Y$, have the form
\[
\begin{tikzcd}
	\bff G \ttimes_{\bff K} \bff H \ar[r, -] & G\ttimes_K H \ar[d, "\bfib{G,H}"] 
	\\ & \bfb G\ttimes_{\bfb K} \bfb H
\end{tikzcd}
\qquad
\text{where $K = f_\sharp(G) = g_\sharp(H) \in \Mtot(Z)$}
\]
except for the case that $G \in \M 1(X)$ and $H \in \M 1(Y)$ are proper boundary hypersurfaces with $f_\sharp(G) = g_\sharp(H) = Z \in \M 0(Z)$, in which case
the fibrations have a more complicated form
\[
\begin{aligned}
\begin{tikzcd}
	\bff G \jtimes_{\tmax,Z} \bff H \ar[r, -] & X\ttimes_Z Y \cap \tlift G H \ar[d, "\bfib{G,H}"] 
	\\ & \bfb G\ttimes_{\bfb Z} \bfb H
\end{tikzcd}
&& \text{if $X = X_\tmax$, $Y = Y_\tmax$, or} 
\\
\begin{tikzcd}
	\bff G \ttimes_{\bff Z} \bff H \ar[r, -] & X\ttimes_Z Y \cap \tlift G H \ar[d, "\bfib{G,H}"] 
	\\ & \bfb G\jtimes_{\tmin,Z} \bfb H
\end{tikzcd}
&& \text{if $X = X_\tmin, Y = Y_\tmin$.} 
\end{aligned}
\]
\end{thm*}
In the notation, $\bff G \jtimes_{\tmax,Z} \bff H$ denotes the closure of $\bff G^\circ \times_Z \bff H^\circ \times I$ in \eqref{E:intro_join_fiber}
and $\bfb G \jtimes_{\tmin,Z} \bfb H$ denotes the closure of $\bfb G^\circ \times_Z \bfb H^\circ \times I$ in \eqref{E:intro_join_base}; the general
case ($X$ and $Y$ not necessarily minimal or maximal) is covered in Theorem~\ref{T:bhs_of_fib_prod}.
In particular, for any morphism $f : X \to Y$, the graph $\Gr(f) = X \ttimes_Y Y$ is realizable as a fiber product; thus $f$ factors canonically within the category
as the inclusion of a p-submanifold $X \cong \Gr(f) \hookrightarrow X\ttimes Y$ and a b-fibration $X\ttimes Y \to Y$.

\medskip
The ordered product realizes and generalizes known products from two different settings mentioned previously.
The first is the product of \emph{smoothly stratified spaces}, which are topological spaces $\strat X = \bigsqcup s_i$ decomposed into disjoint
manifolds $s_i$ of varying dimension called \emph{strata}, with some conditions on how the strata fit together (see \S\ref{S:strat}) amounting
to an iterative, parameterized conic degeneration of each stratum onto the next.
There is a well-known equivalence between a smoothly stratified space $\strat X$ on one hand, and a manifold with fibered corners $X$ on the other, wherein $\strat X$ is obtained by collapsing
the fibers of the boundary fibrations of $X$, while $X$ is recovered from $\strat X$ by iteratively resolving the strata by a kind of blow-up which resolves cones to cylinders
\cite{ALMP,Albin}.
\begin{thm*}[Theorem~\ref{T:prod_resolves_prod}]
If $X$ and $Y$ are manifolds with fibered corners associated to smoothly stratified spaces $\strat X = \bigsqcup_i s_i$ and $\strat Y = \bigsqcup_j s'_j$, then 
the interior maximal ordered product
$X_\tmax \ttimes Y_\tmax$ is the manifold with fibered corners associated to the product $\strat X \times \strat Y = \bigsqcup_{i,j} s_i\times s'_j$ of stratified spaces.
\end{thm*}

The second is the product of \emph{many body spaces}, which are compactifications 
\[
	\mb V = [\ol{V}; \set{\pa \ol S : S \in \cS_V}] 
\]
of finite
dimensional real vector spaces $V$ obtained by iteratively blowing up the radial compactification $\ol V$ along the boundaries of a finite set $\cS_V$ of
linear subspaces.
Such a space is naturally an interior minimal manifold with fibered corners.
\begin{thm*}[Theorem~\ref{T:mb_prod}]
There is a natural isomorphism $\mb V_\tmin \ttimes \mb W_\tmin \cong \mb{V\times W} = [\ol{V\times W}; \set{\pa (\ol{S\times S'}) : S \in \cS_V,\ S' \in \cS_W}]$
between the interior minimal ordered product of $\mb V$ and $\mb W$ and the many body space of the product $V\times W$.
\end{thm*}

\smallskip
An overview of the paper is as follows.
We first recall some notation and background on manifolds with corners in \S\ref{S:bkg}, including some of the theory of \emph{generalized blow-up} from \cite{KMgen} (even though all blow-ups here are ordinary ones), which is used throughout to simplify the proofs of our results.
As it happens, a number of properties of the ordered product do not actually depend on the fibration structures on the principal faces of the factors, but rather only on the induced order 
among these principal faces.
%
For this reason, and since it may have more general applications, we first develop much of the theory in a larger category of \emph{ordered corners} in \S\ref{S:ordcorn} 
in which manifolds are equipped only with a suitable orderings of their principal faces and morphisms are consistent with these.
The properties of products in this ordered corners category are proved in \S\ref{S:ordcorn_product} and the properties of fiber products are proved in \S\ref{S:ord_fib_prod}.
In \S\ref{S:ord_bhs} we give a detailed analysis of the product structure of boundary hypersurfaces of $X\ttimes Y$ in preparation for the fibered corners category; 
a key role is played by so-called \emph{compressed projection maps} which in a certain sense extend the boundary fibrations to fibrations involving a cone
over the fiber or base, depending on context.
We then review fibered corners structures in \S\ref{S:fibcorn} and obtain the results stated as parts (a)--(d) of the main product theorem above as well as the fiber product theorem in \S\ref{S:fibcorn_product}.
The connection to stratified spaces and many body spaces is discussed in sections \S\ref{S:strat} and \S\ref{S:mb}, respectively.
In \S\ref{S:geom} we introduce the geometric structures encoded by the wedge and \phistr\ tangent bundles, proving the product result for 
\phistr\ structures in \S\ref{S:geom_phi}
and for 
wedge structures in \S\ref{S:geom_wedge}.
In fact we give two proofs each of the product results for wedge and QFB metrics, one based on a direct analysis of the respective tangent bundles, and another alternate proof based on examination of the form of the lifted product metrics.
In \S\ref{S:equiv}, we prove the diffeomorphism between the minimal and maximal join products, which apart from its conceptual relevance in this paper, gives an 
unusual example of a diffeomorphism of blow-ups of manifolds with corners which is non-trivial on the interior.
Finally, Appendix~\ref{S:tubes} contains a short exposition of some technical results concerning tubular neighborhood structures with nice properties and an Ehresmann lemma in the fibered corners
category.

\begin{ack}
CK was supported by NSF grant DMS-1811995, and FR was supported by NSERC and a Canada Research chair.
The authors are grateful to Richard Melrose and Pierre Albin for helpful discussions during the preparation of this material.
\end{ack}

\section{Background} \label{S:bkg}

We briefly recall some of the important notions from the theory of manifolds with corners used below, though we assume the reader is already somewhat familiar; 
for more complete references and/or more leisurely introductions, see \cite{Grieser,HMM,MAPSIT,MDAOMWC}. 

A \emph{manifold with corners} is a Hausdorff space $X$ locally diffeomorphic\footnote{i.e.,
locally homeomorphic with transition maps given by diffeomorphisms.} to open
sets in $\bbR_+^n$, where $\bbR_+ = [0,\infty)$. 
With a few notable exceptions, such as the cones and normal models of hypersurfaces considered below, all manifolds with corners in this article are assumed to be compact unless context makes it clear otherwise.
Every point has a well defined codimension, given by the number of vanishing $\bbR_+$ factors in any chart, and the closure of a maximal (connected) set 
of points with a fixed codimension is a \emph{boundary face}, with the set of boundary faces of codimension $k$ denoted by $\M k (X)$; in particular $\M 1(X)$ is
the set of \emph{boundary hypersurfaces}.
We write $\M {}(X) = \bigcup_{k\geq 0} \M k(X)$ for the set of all boundary faces.
The \emph{interior}, $X^\circ$, of $X$ is the set of points with codimension 0, and the \emph{depth} of $X$ is the maximum codimension occurring on $X$, or equivalently the 
maximum number of boundary hypersurfaces that have a nonempty mutual intersection.
We assume boundary hypersurfaces (and therefore all boundary faces) are embedded; in particular each boundary face is again a manifold with corners.
A function $\rho_G : X \to \bbR_+$ is \emph{boundary defining} for $G \in \M 1(X)$ if $\rho_G^\inv(0) = G$ and $d\rho_G \neq 0$ on $G$;
the ratio $\rho_G/\rho'_G$ of two boundary defining functions for $G$ is 
a strictly positive smooth function. 

Note that we do not require that a manifold with corners is connected, nor do we strictly require that boundary faces are connected, despite how they have just been defined!
Indeed, it will be useful below to allow certain unions of disjoint boundary hypersurfaces to be considered as a single hypersurface, or what is called in \cite{AM}
a \emph{collective boundary hypersurface}.
This identification of collective boundary hypersurfaces (i.e., which components are to be considered to belong to `the same' hypersurface) constitutes additional data on the manifold with corners,
and in the event that such data has been specified, we abuse notation by using $\M 1(X)$ to refer to the set of collective boundary hypersurfaces; in the absence of 
such a specification $\M 1(X)$ consists by default of connected components only.
By contrast, we adopt the convention that $X \in \M 0(X)$ is always considered as a single face even when $X$ is not connected; thus $\M 0(X) = \set{X}$ is always a singleton.

%

The b-vector fields $\bV(X) \subset \cV(X) = C^\infty(X; TX)$ are by definition the Lie subalgebra of smooth vector fields tangent to all boundary faces.
These form a locally
free sheaf of constant rank, defining the b-tangent vector bundle 
\[
	\bT X \to X 
\quad \text{by} \quad
	\bV(X) = C^\infty(X; \bT X),
\]
spanned in local coordinates 
$(x_1,\ldots,x_k,y_1,\ldots,y_l) \in \bbR_+^k\times \bbR^l$ by $x_i \npd {x_i}$ and $\npd {y_j}$.
Over $x_i = 0$, the b-tangent vector $x_i \npd {x_i} = \rho_{G_i} \pa_{\rho_{G_i}}$ is actually independent of the choice of boundary defining coordinate $x_i$.
The inclusion $\bV(X) \subset \cV(X)$ induces a natural bundle map $\bT X \to T X$ which is an isomorphism over the interior; at $p \in X$, the kernel
defines the b-normal space $\bN_p E$,
where $E$ is the unique boundary face whose interior contains $p$.
Taken together, these form a vector subbundle, $\bN E$, of $\bT X$ over the interior of $E$ which extends by continuity to the whole of $E$.
In fact $\bN E \to E$ is trivialized globally by the frame $\bigl\{\rho_{G_i} \pa_{\rho_{G_i}}\bigr\}$, where $\rho_{G_i}$ are any boundary defining functions
for the hypersurfaces $G_1,\ldots,G_k$ whose intersection has $E$ as a connected component.

An \emph{interior b-map} (hereafter simply a b-map) $f : X \to Y$ is a smooth map such that, for any sets $\set{\rho_G : G \in \M 1(X)}$ and $\set{\rho_H : H \in \M 1(Y)}$ of boundary defining functions
for $X$ and $Y$,
\begin{equation}
	f^\ast(\rho_H) = a_H \prod_{G \in \M 1(X)} \rho_G^{e(H,G)}, \quad \text{for every $H \in \M 1(Y)$}
	\label{E:bmap}
\end{equation}
where $a_H > 0$ is smooth and strictly positive and the \emph{boundary exponents} $e(H,G) \in \bbZ_+$ are nonnegative integers.
We will say that a b-map $f$ is \emph{rigid} with respect to fixed sets of boundary defining functions on $X$ and $Y$ if each $a_H$ in \eqref{E:bmap}
is identically 1; while this is not a standard notion it turns out to be useful below.

Returning to standard concepts, a b-map $f$ is said to be \emph{simple}
if each $e(H,G) \in \set{0,1}$.
Under an interior b-map $f$, every boundary face $E \in \M k(X)$ is mapped into a unique smallest face $F \in \M l(Y)$ (determined by the condition that $f(E^\circ) \subset F^\circ$), an assignment 
we denote by
\begin{equation}
	f_\sharp : \M {}(X) \to \M {}(Y).
	\label{E:sharpmap}
\end{equation}
In general the differential 
of a b-map
extends by continuity from the interior to a bundle map
\begin{equation}
	\bd f_\ast : \bT X \to \bT Y
	\label{E:bdifferential}
\end{equation}
over $f$, and
this \emph{b-differential}
restricts over $E \in \M {}(X)$ to a bundle map 
\begin{equation}
	\bd f_\ast : \bN E \to \bN F, \quad \text{where $F = f_\sharp(E) \in \M {} (Y)$}.
	\label{E:bnormal_map}
\end{equation}
%
Indeed, \eqref{E:bnormal_map} is determined on the frame $\set{\rho_G \pa_{\rho_G}}$ by 
\begin{equation}
	\rho_G \pa_{\rho_G} \mapsto \sum_H e(H,G)\, \rho_H \pa_{\rho_H},
	\label{E:monoid_gen_map}
\end{equation}
where $e(H, G)\in \bbZ_+$ are the boundary exponents of $f$.

A map $f$ is said to be a \emph{b-submersion} (resp.\ \emph{b-immersion}) if \eqref{E:bdifferential} is a surjective (resp.\ injective) on fibers, and
is said to be \emph{b-normal} provided each map \eqref{E:bnormal_map} is surjective on fibers; note that the latter condition holds if and only if it holds for all hypersurfaces $E \in \M 1(X)$.
In particular a b-normal map cannot map a boundary hypersurface into any face of codimension 2 or more, so \eqref{E:sharpmap} restricts to a map
\[
	f_\sharp : \M 1(X) \to \Mtot(Y) \qquad \text{if $f$ is b-normal}
\]
where $\Mtot(Y) = \M 1(Y) \cup \M 0(Y) = \M 1(Y) \cup \set Y$ is the set of \emph{principal faces}, notation and terminology which is frequently employed below.
An equivalent condition to b-normality is that for each $G \in \M 1(X)$, there is at most one $H \in \M 1(Y)$ for which the boundary exponent $e(H,G)$ is nonzero.

A b-normal map can be \emph{rigidified} in the sense that, for any set of boundary defining functions on $Y$, there is a set
of boundary defining functions on $X$ with respect to which $f$ is rigid; indeed, starting with an arbitrary set $\set{\rho_G : G \in \M 1(X)}$, this is achieved by replacing $\rho_G$ by $b_G \rho_G$, where $b_G = \prod_{H \in \M 1(Y)} (a_H^{\inv})^{g(G,H)}$ and $g(\bullet,\bullet)$ forms a right inverse for $e(\bullet,\bullet)$ viewed as a 
$\abs{f_\sharp(\M 1(X))}\times \abs{\M 1(X)}$ matrix.
In particular, a b-normal map can be put locally into the normal coordinate form: 
\begin{equation}
	f : (x,y) \mapsto \big(x^e, \ol y(x,y)\big) = (\ol x, \ol y)
	\label{E:bnormal_normal}
\end{equation}
where $(\ol x,\ol y) \in \bbR_+^{k'} \times\bbR^{l'}$ are coordinates centered at $f(p) \in Y$ and
$(x, y) \in \bbR_+^k\times \bbR^l$ are coordinates centered at $p \in X$, and where
$\ol x = x^e$ is shorthand for $\ol x_j = \prod_i x_i^{e(j,i)}$, $1 \leq j \leq k'$.

\begin{rmk}
We will occasionally employ the exponential notation $x^e = (\prod_{i=1}^k x_i^{e(1,i)}, \ldots, \prod_{i=1}^k x_i^{e(n,i)})$ where $x = (x_1,\ldots,x_k)$ and $e = [e(j,i)]$
is an $n\times k$ matrix. 
In this convention $(x^e)^g = x^{ge}$, where $ge$ is the usual matrix product of $g$ and $e$. 
\end{rmk}

A b-normal b-submersion is said to be a \emph{b-fibration}, and such a map can always be written in the local normal form
\begin{equation}
	f : (x,y) \mapsto (x^e, \pr_{l'} y), 
	\label{E:bfibn_normal}
\end{equation}
where $\pr_{l'} : (y_1,\ldots,y_l) \to (y_1,\ldots,y_{l'})$ denotes the projection onto the first $l' \leq l$ coordinates.
The restriction of a b-fibration to any boundary face $E \in \M {}(X)$ is again a b-fibration $f : E \to f_\sharp(E)$ onto its image in $Y$.
A simple b-fibration
with the property that each boundary hypersurface $H$ of $Y$ has a unique hypersurface $G \in f_\sharp^\inv(H)$ in $X$ which maps to it is in fact a fiber bundle of manifolds
with corners; indeed in this case it follows from the local normal form \eqref{E:bfibn_normal} (in which $e$ is a projection matrix up to permutation) that $f$ is a surjective submersion of compact manifolds with corners, and the standard argument for manifolds then shows that it is a fiber bundle.

A subset $P\subset X$ is a \emph{p-submanifold} if it is covered by coordinate charts $(x,y)$ in which it is locally defined by the vanishing of $\codim(P)$ of the coordinates.
%
Examples include boundary faces themselves.

The \emph{blow-up} of a closed p-submanifold $P \subset X$ is the space
\[
	[X; P] = (X \setminus P) \cup S_+P
\]
where the \emph{front face} of the blow-up, $S_+P$, consists of the inward pointing spherical normal bundle of $P$, the fiber at $p \in P$
of which consists of normal vectors to $P$ at $p$ which have unit length with respect to any
norm and point inward, with smooth structure generated by polar coordinates normal to $P$.
This is equipped with a canonical b-map $\beta : [X; P] \to X$ (which is not generally b-normal), given by the bundle projection $S_+P \to P$ at the front face and the identity elsewhere.
If $S \subset X$ is another (connected) set, the \emph{lift} of $S$ to $[X; P]$ is defined by $\beta^\inv(S)$ if $S \subset P$, and the closure of $S \setminus P$ in $[X; P]$ 
otherwise, and provided $S$ lifts to a p-submanifold in $[X; P]$ (abusively denoted by the same letter), the iterated blow-up 
\[
	[X; P, S] := [[X; P]; S]
\]
is well-defined.
In particular, every boundary face lifts to a p-submanifold under a blow-up, and the iterated blow-up of any number of boundary faces is always well-defined.

Three additional topics from manifolds with corners will be used below. 
As these are either new or relatively unknown, we devote a subsection to each for ease of reference.

\subsection{Normal Models} \label{S:bkg_normal}

For $G \in \M 1(X)$ we define $\nu : \Np G\to G$ to be the inward pointing normal (not b-normal!) bundle to $G$, which will be referred to as the \emph{normal model}
of $G$ in $X$.
That is to say, $NG = TX\rst_G / TG$, which is a non-canonically trivial line bundle with well-defined inward pointing subspace $\Np G$, taken to include
the zero section, with an 
action by $(0,\infty)$ fixing the zero section $G_0$ which we identify with $G$ itself.
It is convenient to extend this to the case $G = X \in \M 0 (X)$ as well, in which case $\Np X$ is the unique 0-dimensional bundle $\Np X \equiv X_0 \to X$, identifed with $X$ itself, with trivial $(0,\infty)$ action.
If $f : X \to Y$ is a b-normal b-map, then the differential descends to a well-defined $(0,\infty)$-equivariant map
\begin{equation}
	df : \Np G \to \Np H, \qquad H = f_\sharp(G) \in \Mtot(Y)
	\label{E:normal_functor}
\end{equation}
where again $\Np Y = Y$ itself in case $f_\sharp(G) = Y$.
In this latter case $df = \nu^\ast (f\rst_G)$ factors as the composition of the bundle projection $\nu : \Np G \to G$ with $f\rst_G : G \to Y$. 
This consideration applies in particular to smooth functions, with $f \in C^\infty(X; \bbR)$ regarded as a b-map $f : X \to \bbR$ with $f_\sharp(G) = \bbR \in \M 0(\bbR)$, lifting to $df = \nu^\ast f : \Np G \to \bbR$.
In contrast, boundary defining functions on $X$,
regarded as b-maps $\rho_{G'} : X \to [0,\infty)$,
lift to $\Np G$ 
as $d\rho_{G'} = \nu^\ast (\rho_{G'} \rst_G)$
for $G' \neq G$ and $d\rho_G : \Np G \to \Np \set0 = [0,\infty)$ (with nontrivial $(0,\infty)$ action) for $G$ itself. 
It follows from this observation that \eqref{E:normal_functor} is a b-normal b-map of non-compact manifolds with corners serving as an infinitesimal model for $f$ itself near $G \subset X$,
and we frequently use such models below.

\subsection{Generalized blow-up} \label{S:bkg_genblow}
We will make use of some of the theory of \emph{generalized blow-up}
developed in \cite{KMgen}, and we now recall the relevant details here, keeping the treatment as self-contained as possible.
We begin with the observation that within each b-normal space $\bN_p F$, $F
\in \M {}(X)$, there is a freely generated
\emph{monoid} (meaning a commutative semigroup with identity)
\[
	\bM_p F := \bbZ_+\pair{\rho_1\pa_{\rho_1}, \ldots, \rho_k \pa_{\rho_k}},
\]
which is well-defined by the independence of the b-normal vectors $\rho_i \npd{\rho_i}$ of the choices $\rho_i = \rho_{G_i}$ of boundary defining functions,
and which is likewise independent of $p \in F$, so we just denote this monoid by $\bM F$ from now on and 
denote 
the generator 
associated to the boundary hypersurface $G_i = \set{\rho_i = 0}$
simply by 
\[
	g_i := \rho_i \npd{\rho_i}.
\]
Whenever $F \subset E$, the normal monoid $\bM E$ (which is generated by fewer
of the same $g_i$) includes as a \emph{face} of $\bM F$, which, up to
reordering, is just the inclusion
\begin{equation}
	\bM E \cong \bbZ_+\pair{g_1,\ldots,g_l,0,\ldots,0} \subset \bbZ_+\pair{g_1,\ldots, g_l, g_{l+1},\ldots, g_k} \cong \bM F.
	\label{E:monoid_face_incl}
\end{equation}

Since the coefficients of \eqref{E:bnormal_map} are nonnegative integers, the b-differential of a b-map $f : X \to Y$ induces well-defined monoid homomorphisms\footnote{meaning
maps intertwining addition which send $0$ to $0$}
\begin{equation}
	\bd f_\ast : \bM F \to \bM f_\sharp(F), \quad \text{for each $E \in \M {}(X)$},
	\label{E:monoid_hom}
\end{equation}
which are consistent with the inclusions \eqref{E:monoid_face_incl}.
Note that if $f$ is both simple and b-normal then in \eqref{E:monoid_hom} each generator of $\bM F$ 
maps to a generator of $\bM f_\sharp(F)$, and conversely: the condition that generators map to generators in \eqref{E:monoid_hom} implies that $f$ is simple and b-normal.

In the particular case that $f = \beta : \wt X = [X; E_1,\ldots,E_n]\to X$ is the blow-down map of an iterated blow-up of boundary faces $\set{E_i} \subset \M {}(X)$ of $X$, 
the images under \eqref{E:monoid_hom} of the monoids from $\wt X$ form a complete decomposition of each monoid $\bM F$ of the target, in the sense that 
$\bM F$ is a union of such image monoids $\bd f_\ast(\bM \wt F)$, $\wt F \in \M {}(\wt X)$, meeting along boundary faces.
Such decompositions constitute what is called a \emph{refinement} in \cite{KMgen, Kgc}.

For example, if $E = G_1 \cap \cdots \cap G_l$, with $\bM E = \bbZ_+\pair{g_1,\ldots,g_l}$, then 
for any $F \subseteq E$, the refinement of $\bM F$ induced by the blow-up $\beta : [X; E] \to X$
consists of the $l$ submonoids of $\bM F$ in which precisely one of the generators $\set{g_1,\ldots,g_l}$
is replaced by $g_1+\cdots + g_l$; this corresponds to the geometric process of `barycentric subdivision' of $\bM E$, considered as a face of $\bM F$.
Examples are illustrated in Figure~\ref{F:blowup}.

\begin{figure}[htb!]
\[
\begin{tikzcd}
\text{(a)}
\begin{tikzpicture}[baseline=(current bounding box.north),x = (90:2cm), y=(210:2cm), z=(-30:2cm)]
	\node [above] at (1,0,0) {$g_1$};
	\node [below left] at (0,1,0) {$g_2$};
	\node [below right] at (0,0,1) {$g_3$};

	\draw (1,0,0) -- (0,1,0) -- (0,0,1) -- cycle;
	\filldraw (1,0,0) circle (2pt);
	\filldraw (0,1,0) circle (2pt);
	\filldraw (0,0,1) circle (2pt);
\end{tikzpicture}
&
\text{(b)}
\begin{tikzpicture}[baseline=(current bounding box.north),x = (90:2cm), y=(210:2cm), z=(-30:2cm)]
	\node [above] at (1,0,0) {$g_1$};
	\node [below left] at (0,1,0) {$g_2$};
	\node [below right] at (0,0,1) {$g_3$};
	\node at (0.25,0.5,0.5) {$\scriptstyle g_1+g_2+g_3$};

	\draw (1,0,0) -- (0,1,0) -- (0,0,1) -- cycle;
	\filldraw (1,0,0) circle (2pt);
	\filldraw (0,1,0) circle (2pt);
	\filldraw (0,0,1) circle (2pt);
	\filldraw (1,1,1) circle (2pt);
	\draw (1,0,0) -- (1,1,1);
	\draw (0,1,0) -- (1,1,1);
	\draw (0,0,1) -- (1,1,1);
\end{tikzpicture}
\\
\text{(c)}
\begin{tikzpicture}[baseline=(current bounding box.north),x = (90:2cm), y=(210:2cm), z=(-30:2cm)]
	\node [above] at (1,0,0) {$g_1$};
	\node [below left] at (0,1,0) {$g_2$};
	\node [below right] at (0,0,1) {$g_3$};

	\draw (1,0,0) -- (0,1,0) -- (0,0,1) -- cycle;
	\filldraw (1,0,0) circle (2pt);
	\filldraw (0,1,0) circle (2pt);
	\filldraw (0,0,1) circle (2pt);
	\filldraw (0,0.5,0.5) circle (2pt);
	\node [below] at (0,0.5,0.5) {$\scriptstyle h_2 + h_3$};
	\draw (1,0,0) -- (0,0.5,0.5);
\end{tikzpicture}
&
\text{(d)}
\begin{tikzpicture}[baseline=(current bounding box.north),x = (90:2cm), y=(210:2cm), z=(-30:2cm)]
	\node [above] at (1,0,0) {$g_1$};
	\node [below left] at (0,1,0) {$g_2$};
	\node [below right] at (0,0,1) {$g_3$};

	\draw (1,0,0) -- (0,1,0) -- (0,0,1) -- cycle;
	\filldraw (1,0,0) circle (2pt);
	\filldraw (0,1,0) circle (2pt);
	\filldraw (0,0,1) circle (2pt);
	\filldraw (1,1,1) circle (2pt);
	\draw (1,0,0) -- (1,1,1);
	\draw (0,1,0) -- (1,1,1);
	\draw (0,0,1) -- (1,1,1);
	\filldraw (0,0.5,0.5) circle (2pt);
	\draw (1,0,0) -- (0,0.5,0.5);
\end{tikzpicture}
\end{tikzcd}
\]
\caption{A schematic depiction of (a) the monoid $\bbZ_+\pair{g_1,g_2,g_3}$ associated to a boundary face $E_{123} = G_1 \cap G_2 \cap G_3$, 
(b) the refinement associated to the blow-up $[X; E_{123}]$,
(c) the refinement associated to the blow-up
$[X; E_{23}]$, 
and (d) the refinement associated to their common resolution, the iterated blow-up $[X; E_{123}, E_{23}] \cong [X; E_{23}, E_{123}]$.}
\label{F:blowup}
\end{figure}
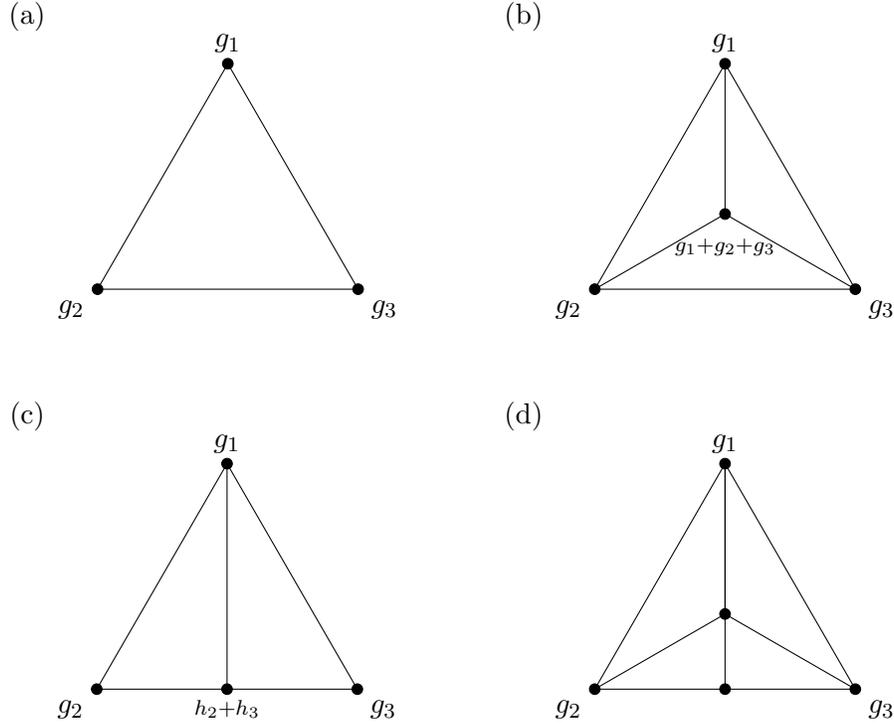

As summarized below, the monoids of the refinement actually characterize the iterated blow-up $[X; E_1,\ldots,E_n]$ completely up to diffeomorphism, and an arbitrary
refinement determines what is called a \emph{generalized blow-up}\footnote{While the refinements considered here 
arise from ordinary iterated boundary blow-up, the general theory in \cite{KMgen} accounts for inhomogeneous blow-ups, as well as 
refinements which do not arise from any classical iterated boundary blow-up, hence the term `generalized blow-up'. 
The paper \cite{Kgc} extends the notion of refinement and generalized blow-up to the category \cite{Joycegc} of manifolds with generalized corners.}, meaning
a manifold with corners $\wt X$
and a suitable blow-down map $\beta : \wt X \to X$, the images of the induced monoid homomorphisms \eqref{E:monoid_hom} of which constitute the given refinement.
In particular, the boundary hypersurfaces of $\wt X$ and their incidence relations (meaning their intersection properties) can be simply read off from the monoids of the refinement: hypersurfaces are in bijection with generators (i.e., 1 dimensional monoids) and intersect if the associated generators together generate a 2 dimensional monoid in the refinement.
The main results we use below are the following, which assert the characterization of an iterated boundary blow-up by its monoid refinement, as well as 
a criterion for lifting b-maps under such a blow-up.

\begin{thm}[{\cite[Theorem 6.3]{KMgen}, also \cite[Theorem 3.7]{Kgc}}]
Let $[Y;F_1,\ldots, F_n]$ be an iterated blow-up of boundary faces $F_1,\ldots,F_n \in \M{}(Y)$ of a manifold with corners $Y$, with blow down map $\beta : [Y; F_1,\ldots,F_n] \to Y$, 
and let $\cR$ denote the associated collection of monoid refinements (given by $\bd\beta_\ast(\bM \wt F) \subset \bM F$ for $\wt F \in \cM {} ([Y; F_1,\ldots,F_n])$ and
 $F = \beta_\sharp(\wt F) \in \cM {} (Y)$). Then
\mbox{}
\begin{enumerate}
\item $\cR$ determines a manifold with corners $\wt Y$ and along with an interior b-map $\beta' : \wt Y \to Y$, restricting to a diffeomorphism $\beta' : \wt Y^\circ \cong Y^\circ$ of interiors, such that
\item
if $f : X \to Y$ is an interior b-map with the property that 
every
monoid homomorphism $\bd f_\ast : \bM E \to \bM f_\sharp(E)$, $E \in \M {}(X)$, factors through a monoid in $\cR$,
then $f$ lifts to a unique interior b-map $\wt f : X \to \wt Y$ such that $f = \beta' \circ \wt f$:
\[
\begin{tikzcd}
	& \wt Y \ar[d,"\beta"]	
	\\ X \ar[r, "f"] \ar[ur, "\wt f"] & Y
\end{tikzcd}
\]
\item
In particular, the blow-down map $\beta : [Y; F_1,\ldots,F_n] \to Y$ factors through a unique diffeomorphism $\wt \beta : [Y;F_1,\ldots,F_n] \cong \wt Y$ identifying $\wt Y$ with the iterated blow-up itself,
and if $[Y; F_{\sigma(1)},\ldots, F_{\sigma(n)}]$ is an iterated blow-up of the same boundary faces in a different order (i.e., for a permutation $\sigma$ of $\set{1,\ldots,n}$), then as long as the two associated refinements of the monoids of $Y$ agree, there is a canonical diffeomorphism $[Y; F_1,\ldots,F_n] \cong [Y; F_{\sigma(1)}, \ldots, F_{\sigma(n)}]$ extending the identity in the interior.
\end{enumerate}
\label{T:lifting_b-maps}
\end{thm}

Though we refer to \cite{KMgen} for the full proof of these results, we give a
sketch here for the convenience of the unfamiliar
reader.
First, the space $\wt Y$ may be constructed directly from the monoid refinement of $Y$ as follows.
Locally, if $(x,y) \in \bbR_+^l \times \bbR^{n-l}$ are coordinates on $Y$ centered at an interior point of a face $F \in \M l(Y)$, then for each monoid of
dimension $l$ in the refinement of $\bM F$, let $\nu = [\nu\pns{i,j}] \in \Mat(l\times l,\bbZ_+)$ denote the matrix of the inclusion map with respect to the generators
(which is unimodular and in particular invertible) and construct a coordinate patch $(t,y) \in \bbR_+^l \times \bbR^{n-l}$ for $\wt Y$ with local blow-down map 
\begin{equation}
	\beta : (t,y) \mapsto (x,y) = (t^\nu, y).  
\label{E:genblow_down_coords}
\end{equation}
In particular $t = x^{\nu^\inv}$ gives `blow-up coordinates' on $\wt Y$ as rational combinations of the $x$.
For adjacent monoids in the refinement, associated to matrices $\nu$ and $\nu'$ respectively, the transition maps
\[
	(t',y) \mapsto (t,y) = (t^{\nu^{\inv} \nu'}, y)
\]
are seen to be diffeomorphisms on open subsets where some of the variables (associated to those generators which are not common to both monoids)
are strictly positive, and glue together naturally with respect to coordinate transitions on $Y$ itself to determine $\wt Y$ and the blow down map
$\beta : \wt Y \to Y$ whose boundary exponents are the $\nu\pns{i,j}$. 

For the lifting result, suppose $f$ is given locally by 
\[
	f : (x',y') \to 
	(x,y) = 
	\big(a{x'}^\mu, y\big) 
\]
where $y = y(x',y')$, and $a{x'}^\mu = (a_1(x',y')\prod_{j} {x'}_j^{\mu\pns{1,j}},\ldots, a_l(x',y') \prod_{j} {x'}_j^{\mu\pns{l,j}})$ with $a_i > 0$,
so that the matrix $\mu = [\mu\pns{i,j}] \in \Mat(l\times k, \bbZ_+)$ of boundary exponents represents the associated monoid homomorphism $\bd f_\ast : \bM E \to \bM F$ (with $E \in \M k(X)$ given locally by $\set{x' = 0}$).
The condition that the latter factors through some monoid in the refinement means that $\mu = \nu \wt \mu$ for some $\nu$ as above, and then
\begin{equation}
	\wt f : (x',y') \to 
	(t,y) = 
	\big(a^{\nu^{\inv}}{x'}^{\wt \mu},y)
\label{E:genblow_lift}
\end{equation}
gives the local coordinate expression for the unique lift $\wt f : X \to \wt Y$; it is straightforward to check that this is functorial and in particular 
behaves well with respect to changes of coordinates.

\subsection{Rational combinations of boundary defining functions} \label{S:bkg_proj}
A \emph{rational combination} of boundary defining functions $\set{\rho_i}$ on a space $X$ associated to boundary hypersurfaces $G_1,\ldots,G_n$ is defined to be a product
of the form
\begin{equation}
	\sigma = \prod_j \rho_j^{a_j}, \qquad a_j \in \bbZ.
	\label{E:proj_comb}
\end{equation}
It will be of interest to know where in $X$ such a combination determines a smooth function locally, and more importantly, when $\sigma$ lifts
globally to a smooth function under a blow-up $\beta : \wt X \to X$.
These 
properties may be determined by considering the logarithmic differential 
\[
	\bd \sigma = \mathrm{d} \log \sigma = \sum_j a_j \frac{d\rho_j}{\rho_j}
\]
of $\sigma$, which induces a $\bbZ$-linear map 
\[
	\bd \sigma : \bM F \to \bbZ 
\]
on the monoid of each boundary face $F \in \M {}(X)$, determined on generators 
by $\bd \sigma (h_i) = a_i$.

Since $F$ is covered by coordinate charts with boundary defining coordinates $x_i$ associated to the generators $g_i \in \bM F$ in which
$\sigma$ has the local coordinate expression $\sigma(x,y) = b(x,y)\prod_i x_i^{a_i}$ where $b > 0$, it follows that $\sigma$ determines a smooth funtion 
$\sigma : U \supset F^\circ \to [0,\infty)$ for some neighborhood $U$ of $F^\circ$ if and only if each $a_i = \bd \sigma(g_i) \geq 0$, or equivalently if $\bd \sigma(\bM F) \subset \bbZ_+$.
Likewise, $\sigma^\inv$ determines a smooth function $\sigma^\inv : U \subset F^\circ \to [0,\infty)$ if and only if $\bd \sigma(\bM F) \subset \bbZ_-$.
Moreover, $\sigma$ (resp.\ $\sigma^\inv$) vanishes on a hypersurface $G_j \cap U$ if and only if $\bd \sigma(g_j) > 0$ (resp.\ $\bd \sigma(g_j) < 0$), since this corresponds 
to $\sigma$ having a local expression involving a term $x_j^{a_j}$ with $a_j > 0$ (resp.\ $a_j < 0$).
Note that in either case (in which $\bd \sigma$ has a constant sign on $\bM F$), $\sigma$ can be regarded as a b-map $\sigma : U \to [0,\infty]$, with $[0,\infty]$ considered as a one-dimensional compact manifold with two boundary hypersurfaces $\set{0}$ and $\set{\infty}$.

Now if $\beta : \wt X = [X; F_1,\ldots,F_n] \to X$ is an iterated bounday blow-up, the behavior of the pullback $\beta^\ast \sigma$ of $\sigma$ to $\wt X$ can be analyzed using the fact that $\bd (\beta^\ast \sigma)(m) = \bd \sigma(\beta_\ast m)$
for $m \in \bM \wt F$, $\wt F \in \M {}(\wt X)$, from which we obtain the following result.

%

\begin{prop}
A rational combination $\sigma$ lifts to a smooth b-map $\beta^\ast \sigma : \wt X = [X; F_1,\ldots,F_n]\to [0,\infty]$ if and only if $\bd \sigma$ takes a constant sign on each monoid in the refinement associated to $\wt X$, and in this case $\beta^\ast \sigma$ (resp.\ $\beta^\ast \sigma^\inv$) vanishes on $\wt G \in \M {}(\wt X)$ provided $\bd \sigma(\wt g) > 0$ (resp.\ $\bd \sigma(\wt g) < 0$), where $\wt g$ is the generator in the refinement associated to the hypersurface $\wt G \in \M 1(\wt X)$.
\label{P:lift_proj_comb}
\end{prop}

Note that boundary defining functions on a blow-up $\wt Y$ of $Y$ are given locally by rational combinations of boundary defining functions from $Y$, as follows
from writing $t = x^{\nu^\inv}$ in \eqref{E:genblow_down_coords}.
Moreover, if a map $f : X \to Y$ is locally rigid and satisfies the hypotheses of Theorem~\ref{T:lifting_b-maps}, then its lift $f : X \to \wt Y$ is rigid with respect to the rational boundary defining functions on $\wt Y$ as follows from \eqref{E:genblow_lift}.

It is evident from Proposition~\ref{P:lift_proj_comb} that a set of the form
$\set{\sigma = 1} \subset X$ (or more generally $\set{\sigma = c}$ for $c \in (0,\infty)$) lifts to an interior p-submanifold of a blow-up
$\wt X$ whenever $\sigma$ lifts to be smooth on $\wt X$,
and by iteration, a set of the form $\set{\sigma_1 = \cdots = \sigma_r = 1}$ 
lifts to an interior p-submanifold of $\wt X$ if each of the $\sigma_i$ lifts to be smooth; equivalently each $\bd \sigma_i$ has a constant sign on each monoid associated to the blow-up $\wt X$.

On the other hand, $\set{\sigma_1= \cdots = \sigma_r = 1}$ may lift to a
p-submanifold of $\wt X$ even when none of the individual rational
combinations $\sigma_i$ lift to be smooth, as in the following result (a specialization of \cite[Prop.~10.3]{KMgen}) that will be used below.
Note that the property that $\bd \sigma$ has a constant sign on a monoid $M
= \bbZ_+\pair{g_1,\ldots,g_l}$ is equivalent to the property that the
intersection of $M$ with the subspace $\ker \bd \sigma$ is a face (possibly $\set{0}$) of $M$, meaning 
a submonoid generated by a (possibly empty) subset of $\set{g_1,\ldots,g_l}$.
Indeed, this is the case if and only if the remaining generators of $M$ all lie entirely within one or the other of the half-spaces 
determined by $\bd \sigma > 0$ or $\bd \sigma < 0$.

\begin{prop}[{\cite[Prop.~10.3]{KMgen}}]
Let $\sigma_i$
be rational combinations of boundary defining functions on $X$ and $c_i \in (0,\infty)$
for
$1 \leq i \leq r$.
Then the set $\set{\sigma_1 = c_1,\ldots,\sigma_r = c_r} \subset X$ lifts to an interior p-submanifold of the iterated blow-up $\wt X = [X; F_1,\ldots, F_n]$ provided
that each monoid $M$ in the refinement associated to $\wt X$ intersects the subspace $\bigcap_{i=1}^r \ker \bd \sigma_i$ along a face of $M$.
\label{P:p-sub_lift}
\end{prop}
\begin{proof}
The key observation is that the equations $\set{\log \sigma_i = \log c_i : 1 \leq i \leq r}$ may be replaced by equivalent $\bbZ$-linear combinations. 
Thus for any given monoid $M$ 
in the refinement associated to $\wt X$, 
the hypothesis that $M$ meets $\bigcap_{i=1}^r \ker \bd \sigma_i$ along a face of $M$ means that the $\log \sigma_i$ may be replaced by $\bbZ$-linear combinations $\log \wt \sigma_i = \sum_j {b_{ij}} \log \sigma_j$ with the property that $\bd \wt \sigma_i = \sum_j b_{ij} \bd \sigma_j$ has constant sign on $M$,
and then it follows from Proposition~\ref{P:lift_proj_comb} that the set lifts to a p-submanifold in any local coordinates on $\wt X$ associated to $M$.
\end{proof}

As an example, in Figure~\ref{F:blowup}, neither of the rational combinations $\rho_1/\rho_2$ or $\rho_1/\rho_3$ lifts to be smooth in any of the depicted blow-ups, yet the set $\set{\rho_1/\rho_2 = \rho_1/\rho_3 = 1}$ lifts to a p-submanifold in the blow-ups depicted in (b) and (d), since $\ker \bd (\rho_1/\rho_2) \cap \ker \bd (\rho_1/\rho_3)$ is the span of $g_1+g_2 +g_3$.

\section{Ordered corners} \label{S:ordcorn}
\begin{defn}
The category of \emph{manifolds with ordered corners} is as follows:
\begin{itemize}
\item An \emph{object} is a manifold with corners $X$ equipped with a partial order on the set
\[
	\Mtot(X) := \M 1(X) \cup \M 0(X) = \M 1 (X) \cup \set{X}
\]
of principal faces,
with the property that every pair of faces in $\Mtot(X)$
which meet are comparable with respect to the order; equivalently, incomparable faces must be disjoint. 
In particular every element of $\Mtot(X)$ is comparable to $X \in \M 0(X)$. 
Recall that we use the notation
\[
	G \sim G' \quad \text{if $G = G'$, $G < G'$, or $G' < G$}
\]
to denote comparable elements.

\item A \emph{morphism} is a simple, b-normal, interior b-map $f : X \to Y$ which is \emph{ordered}, meaning that the induced map
\begin{equation}
	f_\sharp : \Mtot(X) \to \Mtot(Y)
	\label{E:induced_bdy_map}
\end{equation}
is order preserving.
\end{itemize}
\label{D:ordcorn}
\end{defn}
\begin{rmk}
\mbox{}
\begin{itemize}
\item
As noted in \S\ref{S:bkg}, we allow for the possibility that $\M 1(X)$ consists of collective boundary hypersurfaces, that is, unions of disjoint components which are identified as a
single element of $\M 1(X)$.
This always applies to $X \in \M 0(X)$ if $X$ is disconnected, recalling that by convention we always require $\M 0(X) = \set X$ to be a singleton.
\item 
We do not require that a pair of faces be comparable \emph{if and only if} they are non-disjoint;
it is sometimes convenient to consider orders in which possibly disjoint faces are deemed comparable, such as a total order;
nevertheless, for any admissible order on $\Mtot(X)$ there is a minimal suborder with the property that elements are comparable if and only if they are non-disjoint.
\end{itemize}

\end{rmk}

We will often be interested in the case that $X \in \M 0(X)$ is either maximal or minimal, in which case we say $X$ is respectively \emph{interior maximal} or \emph{interior minimal}.
%
Given an order on $\M 1(X)$ satisfying the requirements of Definition~\ref{D:ordcorn},
we denote by $X_\tmax$ or $X_\tmin$ the manifold with ordered corners in which the interior is made maximal or minimal, respectively.
We introduce the notation
\[
\begin{aligned}
	\M 1^<(X) &= \set{G \in \M 1(X) : G < X}, 
	& \M 1^>(X) &= \set{G \in \M 1(X) : G > X}, 
\end{aligned}
\]
with respect to which $\Mtot(X)$ decomposes as a disjoint union
\[
	\Mtot(X) = \M 1^<(X) \cup \set X \cup \M 1^>(X).
\]
Note that $X \in \M 0(X)$ forms a basepoint for the ordered set $\Mtot(X)$, in that it is comparable to every other element and is preserved by the map
$f_\sharp : \Mtot(X) \to \Mtot(Y)$ associated to a morphism $f : X \to Y$ (as $f$ is interior); thus
the association $X \mapsto \Mtot(X)$ is a functor from the category
of manifolds with ordered corners to the category of \emph{pointed ordered sets}.
Note also that $f_\sharp(\M 1^<(X)) \subset \M1^<(Y) \cup \set Y$ and $f_\sharp(\M 1^>(X)) \subset \M1^>(Y) \cup \set Y$, with notation as above.

\subsection{Products} \label{S:ordcorn_product}
Products of manifolds with ordered corners can be motivated by considering the functor $X \mapsto \Mtot(X)$.
The product of $\Mtot(X)$ and $\Mtot(Y)$ as ordered sets is the cartesian product $\Mtot(X)\times \Mtot(Y)$ equipped with the product order, depicted schematically as follows:
\begin{equation}
\begin{tikzcd}[sep=small]
	\M 1^<(X) \times \M 1^>(Y) \ar[r] & X \times \M 1^>(Y) \ar[r] & \M 1^>(X)\times \M 1^>(Y)
	\\
	\M 1^<(X) \times Y \ar[r] \ar[u] & X \times Y \ar[r] \ar[u] & \M 1^>(X)\times Y \ar[u]
	\\
	\M 1^<(X) \times \M 1^<(Y) \ar[r] \ar[u] & X \times \M 1^<(Y) \ar[r] \ar[u] & \M 1^>(X)\times \M 1^<(Y) \ar[u]
\end{tikzcd}
	\label{E:ord_prod_bhs}
\end{equation}
\begin{conv}
Here and below we employ the following diagrammatic notation conventions for products of ordered sets: an arrow of the form $A \times B \to A \times B'$ 
between products with a single factor in common (and $B \cap B' = \emptyset$) means that $(a,b) < (a, b')$ for each $a \in A$, $b \in B$, and $b' \in B'$.
On the other hand, an arrow of the form $A \times B' \to A' \times B$ with no common factor on both sides (and $A \cap A' = B \cap B' = \emptyset$) means that $(a,b') < (a',b)$ for every $a \in A$, $a' \in A'$, $b \in B$, and $b' \in B'$. 
\label{Conv:ord_diag}
\end{conv}
However, the product in the category of \emph{pointed ordered sets} is the subset of pairs which are comparable to the base point $(X,Y) \in \M0(X)\times \M0(Y)$:
\begin{equation}
\begin{tikzcd}[sep=small]
	 & X \times \M 1^>(Y) \ar[r] & \M 1^>(X)\times \M 1^>(Y)
	\\
	\M 1^<(X) \times Y \ar[r]  & X \times Y \ar[r] \ar[u] & \M 1^>(X)\times Y \ar[u]
	\\
	\M 1^<(X) \times \M 1^<(Y) \ar[r] \ar[u] & X \times \M 1^<(Y)  \ar[u] & 
\end{tikzcd}
	\label{E:point_ord_prod}
\end{equation}

It is easy to see that the cartesian product $X\times Y$ cannot be a product in the ordered corners category in general; its set $\Mtot(X\times Y)$ of boundary hypersurfaces may be identified with the subset 
of $\Mtot(X)\times \Mtot(Y)$ depicted by
\[
\begin{tikzcd}[sep=small]
	 & X \times \M 1^>(Y)  &
	\\
	\M 1^<(X) \times Y \ar[r]  & X \times Y \ar[r] \ar[u] & \M 1^>(X)\times Y 
	\\
	 & X \times \M 1^<(Y) \ar[u] &
\end{tikzcd}
\]
and hypersurfaces in $\M1^>(X)\times Y$ and $X \times \M 1^>(Y)$ meet yet are incomparable in the product order (and similarly for 
faces in $\M1^<(X) \times Y$ and $X \times \M 1^<(Y)$).
There is therefore no natural way to put an ordered corners structure on $X\times Y$ such that both projection maps $X\times Y \to X$ and $X\times Y \to Y$ are ordered.
On the other hand, faces in $\M1^<(X) \times Y$ (resp.\ $\M1^>(X)\times Y$) are comparable to those in $X \times \M 1^>(Y)$ (resp.\ $X\times \M 1^<(Y)$), 
so in the case that one of $X$ or $Y$ is interior minimal while the other is interior maximal, $X\times Y$ \emph{is} naturally a manifold with ordered corners when equipped with the product
order, and as we will see below, it satisfies the universal property of the product in this case.

\begin{rmk}
It is also possible to equip $X\times Y$ with one of the two \emph{lexicographic} orders rather than the product order, which gives it an ordered structure with respect to which
one of the projections (but generally not the other) is ordered; we will not make use of this.
\end{rmk}

\begin{defn}
The \emph{ordered product} $X\ttimes Y$ is the iterated blow-up
\begin{equation}
\begin{aligned}
	X\ttimes Y &= [X\times Y; \M 1^<(X)\times \M 1^<(Y), \rM 1^>(X) \times \rM 1^>(Y)]
	\\&\cong [X\times Y; \rM 1^>(X) \times \rM 1^>(Y),\M 1^<(X)\times \M 1^<(Y)]
\end{aligned}
	\label{E:ord_prod}
\end{equation}
where the blow-up is performed in any order consistent with the partial order on the products, and 
where here $\rM 1^>(\bullet)$ denotes the ordered set $\M 1^>(\bullet)$ with the opposite order.
\label{D:ordprod}
\end{defn}

\begin{thm}
The space $X\ttimes Y$ in \eqref{E:ord_prod} is a well-defined manifold with ordered corners, with $\Mtot(X\ttimes Y)$ isomorphic to the ordered set $\set{(G,H) \in \Mtot(X)\times \Mtot(Y) : (G,H) \sim (X,Y)}$ as depicted in \eqref{E:point_ord_prod}.

It is the product of $X$ and $Y$ in the category of manifolds with ordered corners.
More precisely, the projections of $X\times Y$ onto $X$ and $Y$ lift to morphisms (in fact ordered b-fibrations) $X\ttimes Y \to X$ and $X\ttimes Y \to Y$,
and $X\ttimes Y$ satisfies the universal property that if $W$ is a manifold with ordered corners and $f : W \to X$ and $g: W \to Y$ are morphisms, then there exists a unique
morphism $W \to X\ttimes Y$ making the following diagram commute:
\[
\begin{tikzcd}
	& W \ar[d, dashed, "\exists !"] \ar[dl, swap, "f"] \ar[dr, "g"] &
\\ 	X & X\ttimes Y \ar[l] \ar[r] & Y
\end{tikzcd}
\]
\label{T:ord_prod}
\end{thm}

\begin{cor}
If $X = X_\tmin$ and $Y = Y_\tmax$ are interior minimal and maximal, respectively, 
then the cartesian product $X \times Y \equiv X \ttimes Y$ is already the product
of $X$ and $Y$ in the category of manifolds with ordered corners.
\label{C:cart_prod}
\end{cor}

\begin{proof}[Proof of Theorem~\ref{T:ord_prod}]
While it is possible to prove these results by local coordinate computations, the sheer number of coordinate charts on 
$X\ttimes Y$ makes this tedious.
Instead, we employ the generalized blow-up machinery discussed in Section~\ref{S:bkg_genblow} to identify the monoid refinement
of $X\ttimes Y$ as a generalized blow-up of $X\times Y$, from which its ordered corners structure and universal property are easily derived.
We will see that the refinement consists of monoids which are freely generated by sums of generators $g + h$ forming maximally 
ordered chains in $\Mtot(X)\times \Mtot(Y)$, 
under the identification of hypersurfaces $G \in \M 1 (X)$ and $H \in \M 1(Y)$ with their associated monoid generators $g \in \bM G$ and $h \in \bM H$, and with $X$ and $Y$ identified with $0$.

Thus,
consider an arbitrary corner in $X\times Y$. 
By relabeling if necessary, we may assume this has the form
\[
	(G_{-m'} \cap \cdots \cap G_{-1} \cap G_1 \cap \cdots \cap G_m)\times (H_{-n'} \cap \cdots \cap H_{-1} \cap H_1 \cap \cdots \cap H_n)
\]
for totally ordered chains $G_{-m'} < \cdots < G_{-1} < X < G_{1} < \cdots < G_m$ and $H_{-n'} < \cdots < H_{-1} < Y < H_{1} < \cdots < H_n$ of hypersurfaces of $X$ and $Y$, respectively, which, since we work locally near
the corner, we may assume constitute all of the hypersurfaces of $X$ and $Y$.
The monoid associated to this corner of $X\times Y$ has the form
\begin{multline}
	\bbZ_+\pair{g_{-m'},\ldots, g_{-1},g_1,\ldots,g_m}\times \bbZ_+\pair{h_{-n'},\ldots,h_{-1},h_1,\ldots,h_n} 
	\\= \bbZ_+\pair{g_{-m'},\ldots,g_{-1},g_1,\ldots,g_m,h_{-n'},\ldots,h_{-1},h_1,\ldots, h_n}
	\label{E:orig_monoid}
\end{multline}
where we denote the generators associated to the hypersurfaces by lower case letters. 
Identifing $g_0 := 0 \in \bbZ_+\pair{g_{-m'},\ldots,g_m}$ and $h_0 := 0 \in \bbZ_+\pair{h_{-n'},\ldots,h_n}$ with $X$ and $Y$, respectively, determines
an identification between the ordered set $\Mtot(X)\times \Mtot(Y)$ and the elements 
$\set{g_i+h_j: -m \leq i \leq m, -n' \leq j \leq n} \subset \bbZ_+\pair{g_{-m'},\ldots,g_m,h_{-n'},\ldots,h_n}$,
and therefore determines a partial order on the latter set.
In other words, we regard $g_i$ and $h_j$ as $g_i + 0 = g_i + h_0$ and $0 + h_j = g_0 + h_j$, respectively and equip the sums (where we fix the order of addition) with the product order induced by
\begin{equation}
\begin{aligned}
	g_{-m'} < \cdots < g_{-1} < 0&=g_0 < g_{1} < \cdots < g_m, 
	\quad \text{and} \quad
	\\h_{-n'} < \cdots < h_{-1} < 0&=h_0 < h_{1} < \cdots < h_n,
	\label{E:mon_gen_order}
\end{aligned}
\end{equation}

Assuming temporarily that $X$ and $Y$ are minimal in $\Mtot(X)$ and $\Mtot(Y)$, so $m' = n' = 0$, consider the iterated blow-up
$[X\times Y; \rM 1^>(X) \times \rM 1^>(Y)]$.
The first blow-up of $G_m\times H_n$ subdivides \eqref{E:orig_monoid} into the pair of monoids
\[
	\bbZ_+\pair{g_1,\ldots, g_{m}+h_n, h_1, h_2, \ldots, h_{n}}
	\tand
	\bbZ_+\pair{g_1, \ldots, g_{m}, h_1,h_2, \ldots, g_m+h_{n}}.
\]
In particular, the incomparable pair $\set{g_m,h_n} = \set{g_m+0,0+h_n}$ is replaced by either of the comparable pairs $\set{g_m+h_n, g_m + 0}$ or $\set{g_m+h_n,0+h_n}$, and in the resulting monoids, the generator $g_m + h_n$ is comparable
to all other generators $g_i = g_i + 0$, $i < m$ and $h_j = 0 + h_j$, $j < n$.
Proceding through the rest of the iterated blow-up of $\rM 1^>(X)\times \rM 1^>(Y)$ (in any order consistent with the product order) has the effect of subsequently subdividing each remaining monoid with incomparable generators 
into two new monoids
by replacing the highest indexed incomparable pair $\set{g_i,h_j}$ with one of the comparable pairs $\set{g_i+h_j, g_i+0}$ or $\set{g_i+h_j,0+h_j}$, at the end of which process
\eqref{E:orig_monoid} is replaced by the collection of monoids generated by totally ordered chains among 
\[
	\set{g_i+h_j : 0 \leq i \leq m, 0 \leq j \leq n}.
\]
%
Note that had we instead blown up in the opposite order $\M 1^>(X)\times \M 1^>(Y)$, we would have 
started with the blow-up of $G_1 \times H_1$, in which $\set{g_1+0,0+h_1}$ is replaced by either $\set{g_1 + h_1, g_1 + 0}$ or $\set{g_1 + h_1, 0 + h_1}$;
however, the sum $g_1 + h_1$ is then incomparable to the remaining generators $g_i + 0$, $i > 0$ and $0 + h_j$, $j > 0$ in the resulting monoids, an incomparability which is not resolved by any subsequent blow-up.
In particular, the iterated blow up $[X\times Y; \M 1^>(X)\times \M 1^>(Y)]$ (or indeed the iterated blow-up in any other order besides $\rM 1^>(X) \times \rM 1^>(Y)$) fails to obtain
an ordered corners structure consistent with \eqref{E:point_ord_prod}.

Likewise, assume temporarily that $X$ and $Y$ are maximal in $\Mtot(X)$ and $\Mtot(Y)$, so $m= n = 0$, and 
consider the iterated blow-up $[X\times Y, \M 1^<(X)\times \M 1^<(Y)]$.
Now the first blow-up of $G_{-m'}\times H_{-n'}$ subdivides \eqref{E:orig_monoid} into the pair of monoids
\begin{multline*}
	\bbZ_+\pair{g_{-m'}+h_{-n'},g_{-m'+1},\ldots, g_{-1}, h_{-n'}, h_{-n'+1}, \ldots, h_{-1}}
	\\ \text{and} \quad
	\bbZ_+\pair{g_{-m'}, g_{-m'+1},\ldots, g_{-1}, g_{-m'}+h_{-n'},h_{-n'+1}, \ldots, h_{-1}},
\end{multline*}
in which the incomparable pair $\set{g_{-m'}+0,0+h_{-n'}}$ is replaced by either of the comparable pairs $\set{g_{-m'}+h_{-n'}, g_{-m'} + 0}$ or $\set{g_{-m'}+h_{-n'},0+h_{-n'}}$, with the result that the sum $g_{-m'}+h_{-n'}$ is comparable to all other generators $g_i = g_i + 0, i > -m'$ and $h_j = 0 + h_j, j > -n'$.
Proceding through the rest of the iterated blow-up of $\M 1^<(X)\times \M 1^<(Y)$ has the effect of iteratively subdividing each remaining monoid with incomparable generators 
$g_i$ and $h_j$, $-m' \leq i \leq -1$ and $-n' \leq j \leq -1$ into two new monoids, 
by replacing its lowest indexed incomparable pair $\set{g_i,h_j}$ with one of the comparable pairs $\set{g_i+h_j, g_i+0}$ or $\set{g_i+h_j,0+h_j}$.
At the end of this process \eqref{E:orig_monoid} is replaced by the collection of monoids generated by maximal totally ordered chains among the generators
\begin{equation}
	\set{g_i+h_j : -m' \leq i \leq 0, -n' \leq j \leq 0}.
	\label{E:sum_generators}
\end{equation}
%
Again, had we begun instead by blowing up $G_{-1} \times H_{-1}$, then 
generator pairs $\set{g_{-1}+0,0+h_{-1}}$ would be replaced by either $\set{g_{-1} + h_{-1}, g_{-1} + 0}$ or $\set{g_{-1} + h_{-1}, 0 + h_{-1}}$, leaving 
$g_{-1} + h_{-1}$ incomparable to generators $g_{i} + 0 = g_i + h_0, i < -1$ and $0 + h_j, j < -1$.

In the general case, since the monoids generated by totally ordered chains
decompose as products of factors with generators indexed by $i,j < 0$ on the one hand and 
factors with generators indexed by
$i,j > 0$ on the other, 
it follows from Theorem~\ref{T:lifting_b-maps} 
that an equivalent blow-up is obtained by combining the above procedures in either order;
in other words
$[X\times Y; \M 1^<(X)\times \M 1^<(Y), \rM 1^>(X)\times \rM 1^>(Y)] \cong [X\times Y; \rM 1^>(X)\times \rM 1^>(Y), \M 1 ^<(X)\times \M 1^<(Y)]$,
as both are associated to the same refinement of monoids.

Since $X\ttimes Y$ has boundary hypersurfaces associated to the one dimensional
monoids $\bbZ_+\pair{g_i + h_j}$ in the refinement, with hypersurfaces meeting
if and only if their associated monoids generate a two dimensional monoid, it follows
immediately that $X\ttimes Y$ admits an ordered corners structure with
$\Mtot(X\ttimes Y)$ order isomorphic to the pointed product 
\[
	\set{(G,H) \in \Mtot(X)\times \Mtot(Y) : (G,H) \sim (X,Y)}.
\]

To prove the universal property, suppose that
an arbitrary corner in $W$ has the local form $F_{-l'} \cap \cdots \cap F_{-1} \cap F_1 \cap \cdots \cap F_l$ for a totally ordered chain $F_{-l'} < \cdots < F_{-1} < F_0 := W < F_{1} < \cdots < F_l$ of boundary hypersurfaces
of $W$, with associated monoid $\bbZ_+\pair{f_{-l'},\ldots,f_{-1}, f_{1}, \ldots f_l}$.

Given ordered morphisms $f: W \to X$ and $g: W \to Y$, with respect to which
this corner of $W$ maps into the corner $G_{-m'} \cap \cdots \cap G_{-1} \cap G_1 \cap \cdots \cap G_{m}$
of $X$ and $H_{-n'} \cap \cdots \cap H_{-1} \cap H_1 \cap \cdots \cap H_n$ of $Y$,
it follows that the monoid
homomorphism
\[
	\bbZ_+\pair{f_{l'},\ldots,f_{-1}, f_{1}, \ldots ,f_l} \to \bbZ_+\pair{g_{-m'},\ldots, g_{-1}, g_{1}, \ldots,g_m}
\]
is determined on generators by a map of pointed ordered sets
\[
	\set{f_{-l'} < \cdots < f_{-1} < f_0 := 0 < f_{1} < \cdots < f_l} \to \set{g_{-m'} < \cdots < g_{-1} < g_0 := 0 < g_{1} < \cdots < g_m },
\]
which is encoded by an assignment $f_i \mapsto g_{\alpha(i)}$ for an increasing sequence $\alpha(-l') \leq \cdots \leq\alpha(l)$ 
with $\alpha(0) = 0$.
Similarly, the homomorphism $\bbZ_+\pair{f_{-l'},\ldots,f_{-1},f_{1},\ldots,f_l} \to \bbZ_+\pair{h_{-n'},\ldots,h_{-1}, h_{1}, \ldots h_n}$ is determined by 
an 
assignment $f_i \mapsto h_{\beta(i)}$ for an increasing sequence $\beta(-l') \leq \cdots \leq \beta(l)$ with $\beta(0) = 0$.
%
It follows that the product homomorphism
\[
	\bbZ_+\pair{f_{-l'},\ldots,f_l} \to \bbZ_+\pair{g_{-m'},\ldots,g_m,h_{-n'},\ldots,h_n}
\]
is determined by
\[
	f_i \mapsto g_{\alpha(i)} + h_{\beta(i)}
\]
and hence its image is the submonoid generated by the totally ordered chain $g_{\alpha(-l')} + h_{\beta(-l')} \leq  \cdots  \leq 0 \leq \cdots \leq g_{\alpha(l)} + h_{\beta(l)}$
within \eqref{E:sum_generators}, which is precisely a monoid in the refinement discussed above (or a face thereof).
As a consequence of Theorem~\ref{T:lifting_b-maps}, the map $W \to X\times Y$ factors uniquely through the blow-up $X\ttimes Y$, and 
moreover the map $W \to X\ttimes Y$ is simple and b-normal since generators are mapped to generators
in the associated monoid homomorphisms, completing the proof.
\end{proof}

\begin{rmk}
As is evident from the proof, the order of blow-ups of corners necessary to obtain
an order structure consistent with \eqref{E:ord_prod_bhs} follows the rule of thumb that elements further away from the base point $(X,Y) \in \M 0(X) \times \M 0(Y)$ are to be blown up prior to those which are closer to it.
\end{rmk}

The boundary hypersurfaces of $X\ttimes Y$ are the lifts and/or front faces of 
the corners $G\times H \subset X\times Y$ for $(G,H) \in \Mtot(X)\times \Mtot(Y)$, 
and we denote these by
\[
	\tlift G H := \text{lift of $G\times H$ in } \Mtot(X\ttimes Y).
\]
In the next section we determine the structure of these boundary hypersurfaces.

\subsection{Boundary hypersurfaces} \label{S:ord_bhs}
We show below (and it is straightforward to see) that a face of the form $\tlift X H$ or $\tlift G Y$ of $X\ttimes Y$ is essentially an ordered product $X\ttimes H$ or $G\ttimes Y$, respectively, at least when $X$ and $Y$ are both either maximal or minimal.
On the other hand, 
when both $G \in \M 1(X)$ and $H \in \M 1(Y)$ are proper hypersurfaces, $\tlift G H$ is the (lifted) front face of the blow-up of the codimension 2 corner $G\times H$, and as such has dimension
strictly larger than that of the product of $G$ and $H$. 
One description of $\tlift G H$, as an iterated blow-up of $G\times H \times
I$, where $I = [0,1]$, is obtained by considering the
restriction to $G\times H$ of the sequence of steps in the iterated blow-up $X\ttimes Y \to X\times Y$; we record this
below as Proposition~\ref{P:bhs_ordprod}.
Far more important than this description, however, is the description of the \emph{product structure} of $\tlift{G}{H}$
obtained when $G$ and $H$ decompose into cartesian products, as is the case locally when the spaces have fibered corners,
and we devote the remainder of the section to this latter situation.

\begin{prop}
Let $(G,H) \in \M 1(X) \times \M 1(Y)$ and decompose $\M 1(G)$ into the following sets
\[
\begin{aligned}
	\M 1^{<,<}(G) &= \set{G' \cap G : G' < G,\ G' < X},
	&\M 1^{<,>}(G) &= \set{G' \cap G : G' < G,\ G' > X},
	\\\M 1^{>,<}(G) &= \set{G' \cap G : G' > G,\ G' < X},
	&\text{and} \quad \M 1^{>,>}(G) &= \set{G' \cap G : G' > G,\ G' > X},
\end{aligned}
\]
with a similar decomposition for $\M 1(H)$.
The boundary hypersurface $\tlift G H$ of $X\ttimes Y$ has the following form.
\begin{itemize}
\item If $(G,H) \in \M 1^<(X)\times \M1^<(Y)$, then
\begin{equation}
\begin{aligned}
	\tlift{G}{ H} \cong 
	[&G \times H \times I; 
	\\&\M 1^{<,<}(G)\times \M 1^{<,<}(H)\times I, 
	\\&\M 1^{>,<}(G)\times H \times \set{0},\, G\times \M 1^{>,<}(H)\times \set 1,\, 
	\\& \M 1^{>,<}(G)\times \M 1^{>,<}(H) \times I,
\\&\rM 1^{>,>}(G)\times \rM 1^{>,>}(H)\times I]
	\label{E:bhs_ordprod_less}
\end{aligned}
\end{equation}
\item If $(G,H) \in \M 1^>(X)\times \M 1^>(Y)$, then
\begin{equation}
\begin{aligned}
	\tlift{G}{ H} \cong 
	[&G \times H \times I; 
	\\&\rM 1^{>,>}(G)\times \rM 1^{>,>}(H)\times I,
	\\&\rM 1^{<,>}(G)\times H \times \set{0},\, G\times \rM 1^{<,>}(H)\times \set 1,\, 
	\\&\rM 1^{<,>}(G)\times \rM 1^{<,>}(H) \times I,
	\\&\M 1^{<,<}(G)\times \M 1^{<,<}(H)\times I]
	\label{E:bhs_ordprod_gtr}
\end{aligned}
\end{equation}
\item Finally, if $G = X \in \M 0(X)$ or $H = Y \in \M 0(Y)$, then
\begin{equation}
\begin{aligned}
	\tlift{X}{H} &\cong [X\times H; \M 1^<(X) \times \M 1^{<,<}(H), \rM 1^>(X)\times \rM 1^{>,>}(H)] \qquad \text{or}
	\\ \tlift{G}{Y} &\cong [G\times Y; \M 1^{<,<}(G) \times \M 1^<(Y), \rM 1^{>,>}(G) \times \rM 1^>(Y)].
\end{aligned}
\end{equation}
In particular, if $X$ and $Y$ are both interior maximal or both interior minimal, these reduce to the ordered products $X\ttimes H$ and $G\ttimes Y$, respectively.
\end{itemize}
\label{P:bhs_ordprod}
\end{prop}
\begin{proof}
Suppose that $(G,H) < (X,Y)$ and consider the effect of the blow-up $X\ttimes Y \to X\times Y$ on the boundary face $G\times H$. 
Only 
blow-ups of the form $G'\times H'$ for $(G',H') \sim (G,H)$ need to be considered since the others are either disjoint or induce a blow-up of a boundary hypersurface
of $G\times H$, having no effect.

First, the blow ups of $G'\times H'$ for $(G',H') < (G,H)$ induce the blow-up $[G\times H; \M 1^{<,<}(G) \times \M 1^{<,<}(H)]$.
Note that the blow ups of $G\times H'$ and $G' \times H$ for $G' < G$ and $H' < H$ induce blow ups of hypersurfaces in $G\times H$ and have no effect.
Next comes the blow-up of (the lift of) $G\times H$ itself, introducing the product with an interval:
$[G\times H; \M 1^{<,<}(G) \times \M 1^{<,<}(H)] \times I
\cong [G\times H\times I ; \M 1^{<,<}(G) \times \M 1^{<,<}(H) \times I]$.
Next come the blow-up of faces $G\times H'$ and $G'\times H$ for $G < G' < X$ and $H < H' < Y$ which induce the blow-ups
$\M 1^{>,<}(G)\times H \times \set 0$ and $G \times \M 1^{>,<}(H)\times \set 1$, respectively.
Next come the blow-ups of $G'\times H'$ for $G < G' <X$ and $H < H' < Y$, which induce blow-ups of $\M 1^{>,<}(G) \times \M 1^{>,<}(H) \times I$,
and finally come the blow-ups (in reverse order) of $G' \times H'$ for $G' > X$ and $H' > Y$, which induce the blow-ups of $\rM 1^{>,>}(G) \times \rM 1^{>,>}(H)\times I$.

The case that $(G,H) > (X,Y)$ is similar, proceding in the order indicated on the second line of \eqref{E:ord_prod}.
The case of $(X,H)$ (or $(G,Y)$) is also similar, with the omission of the blow-up of $X\times H$ itself (since it is already codimension 1), as well as the omission of the blow-ups of the faces $G\times H'$
for $G < X$ and $H' > H$ or $G > X$ and $H' < H$ since these are incomparable as noted above.
\end{proof}

We say that a boundary hypersurface $G$ is \emph{product-type} if
it is a product, and hence an ordered product by Corollary~\ref{C:cart_prod}, $G = \bff {} \times \bfb {} = {\bff {}}_\tmin \ttimes {\bfb {}}_\tmax$ of an interior minimal manifold $\bff {} = {\bff{}}_\tmin$ and an interior
maximal manifold $\bfb {} = {\bfb{}}_\tmax$.
As noted, this holds locally when $X$ is a manifold with fibered corners, where 
$G$ is a fiber bundle $G \to \bfb {}$ with fiber $\bff {}$ (hence our choice of notation). 
Moreover this is always true locally in general; indeed,
a manifold with ordered corners can always be decomposed locally as the product of an interior minimal and an interior maximal manifold.
For a product-type hypersurface, the ordered corners structures on $G$, $\bff {}$ and $\bfb {}$ are identified with the following subsets of $\Mtot(X)$:
\begin{equation}
\begin{aligned}
	\Mtot(G) &\cong \set{G' \in \M 1(X) : G' \sim G}, &\pa_{G'} G = G \cap G' &\leftrightarrow G'
	\\ \Mtot(\bff {}) &\cong \set{G' \in \M 1(X) : G' \geq G }, & \pa_{G'} \bff {} = (\bff {} \times \pt) \cap G' &\leftrightarrow G'
	\\ \Mtot(\bfb {}) &\cong \set{G' \in \M 1(X) : G' \leq G }, & \pa_{G'} \bfb {} = (\pt \times \bfb {}) \cap G' &\leftrightarrow G'
\end{aligned}
	\label{E:hs_order_ident}
\end{equation}
and we note that $\M 1^<(G) = \bff{}\times\M 1(\bfb{})$ while $\M 1^>(G) = \M 1(\bff{}) \times \bfb{}$.

Fix a product-type hypersurface $G = \bff{}\times \bfb{}\in \M 1(X)$ and consider the ordered corners structure of the normal model $\Np G \cong \bff{}\times\bfb{} \times [0,\infty)$ as a model for $X$ near $G$. 
When $G < X$, $\Mtot(\Np G)$ has the order structure
\begin{equation}
\begin{tikzcd}[baseline=(top.base), sep=small]
	& |[alias=top]| \M 1^{>X}(\bff{}) \times \bfb{} \times [0,\infty)
	\\\bff{} \times \bfb{} \times 0 \ar[r] \ar[ur, end anchor=south west] \ar[dr, end anchor=north west]  & \bff{} \times \bfb{} \times [0,\infty) \ar[u]
	\\  & \M 1^{<X}(\bff{}) \times \bfb{} \times [0,\infty) \ar[u]
	\\ & \bff{} \times \M 1(\bfb{}) \times [0,\infty) \ar[u] \ar[uul, start anchor=north west]
\end{tikzcd}
	\label{E:NpH_below}
\end{equation}
while for $G > X$, $\Mtot(\Np G)$ has the order structure
\begin{equation}
\begin{tikzcd}[baseline=(top.base), sep=small]
	& |[alias=top]| \M 1(\bff{}) \times \bfb{} \times [0,\infty)
	\\ & \bff{} \times \M 1^{>X}(\bfb{}) \times [0,\infty) \ar[u] \ar[dl, start anchor=south west]
	\\ \bff{} \times \bfb{} \times 0 \ar[uur, end anchor=south west] & \bff{} \times \bfb{} \times [0,\infty) \ar[u] \ar[l]
	\\ & \bff{} \times \M 1^{< X}(\bfb{}) \times [0,\infty) \ar[u] \ar[ul, start anchor=north west]
\end{tikzcd}
	\label{E:NpH_above}
\end{equation}
where we decompose $\M 1(\bff{}) = \M 1^{<X}(\bff{}) \cup \M 1^{>X}(\bff{})$ or $\M 1(\bfb{}) = \M 1^{<X}(\bfb{}) \cup \M 1^{>X}(\bfb{})$ into subsets
of elements which are above or below $X$ in the order with respect to the identifications \eqref{E:hs_order_ident}.

Discarding the factor $\bfb{}$ in the first case or $\bff{}$ in the second leads to the following `relative cone' construction, which plays an important role
below.
\begin{defn}
Let $G = \bff{} \times \bfb{}$ be a product-type hypersurface of a manifold with ordered corners $X$. 
If $G < X$, the \emph{cone on $\bff{}$ relative to $X$} is the ordered corners manifold $C_X(\bff{}) := \bff{}\times [0,\infty)$ with the 
order on 
$\Mtot(C_X(\bff{}))$
depicted by the following diagram:
\begin{equation}
\begin{tikzcd}[baseline=(top.base), sep=small]
	&|[alias=top]| \M 1^{> X}(\bff {}) \times [0,\infty)
	\\ \bff {}\times 0 \ar[dr] \ar[r] \ar[ur]& \bff {}\times [0,\infty) \ar[u]
	\\ & \M 1^{< X}(\bff {})\times [0,\infty) \ar[u]
\end{tikzcd}
	\label{E:ord_cone_F}
\end{equation}
Thus $\bff{}\times 0$ is minimal, while $\bff{} \times [0,\infty)$ plays the role of $X$ in the order.
Likewise, if $G  > X$, the \emph{cone on $\bfb{}$ relative to $X$} is the ordered corners manifold $C_X(\bfb{}) = \bfb{}\times [0,\infty)$ with the
order on  
$\Mtot(C_X(\bfb{}))$
depicted by the following diagram:
\begin{equation}
\begin{tikzcd}[baseline=(top.base), sep=small]
	& |[alias=top]|\M 1^{> X}(\bfb{}) \times [0,\infty) \ar[dl]
	\\ \bfb{} \times 0 & \bfb{} \times [0,\infty) \ar[u] \ar[l]
	\\ & \M 1^{< X}(\bfb{}) \times [0,\infty) \ar[u] \ar[ul]
\end{tikzcd}
	\label{E:ord_cone_B}
\end{equation}
We call $F_0 := \bff {}\times 0$ or $B_0 := \bfb{}\times 0$ the \emph{end} of the cone.
If $X$ is either maximal or minimal, then so is $C_X(\bfb{})$ or $C_X(\bff{})$; in this case we use the alternate notation $C_\tmax(\bff{})$ or $C_\tmin(\bfb{})$, respectively.
\label{D:rel_cone}
\end{defn}

\begin{defn}
If $G = \bff G \times \bfb G$ and $H = \bff H \times \bfb H$ are product-type hypersurfaces of $X$ and $Y$, respectively, 
then the \emph{relative join} of $\bff G$ and $\bff H$ for $(G,H) < (X,Y)$, or $\bfb G$ and $\bfb H$ for $(G,H) > (X,Y)$, is the boundary hypersurface of the 
product of relative cones associated to the ends:
\[
\begin{aligned}
	\bff G \jtimes_{X,Y} \bff H &:= \tlift{\bff G\times 0}{\bff H\times 0} \subset C_X(\bff G)\ttimes C_Y(\bff H),
	\\ \bfb G \jtimes_{X,Y} \bfb H &:= \tlift{\bfb G\times 0}{\bfb H\times 0} \subset C_X(\bfb G)\ttimes C_Y(\bfb H).
\end{aligned}
\]
In view of Proposition~\ref{P:bhs_ordprod} these spaces can be expressed as iterated blow-ups
\[
\begin{aligned}
	\bff G \jtimes_{X,Y} \bff H \cong 
	[&\bff G \times \bff H \times I; 
	\\&\M 1^{<X}(\bff G)\times \bff H \times \set{0},\, \bff G\times \M 1^{<Y}(\bff H)\times \set 1,\, 
	\\&\M 1^{<X}(\bff G)\times \M 1^{<Y}(\bff H) \times I,
	\\&\rM 1^{>X}(\bff G)\times \rM 1^{>Y}(\bff H)\times I],
\\\bfb G \jtimes_{X,Y} \bfb H \cong 
	[&\bfb G \times \bfb H \times I; 
	\\&\rM 1^{>X}(\bfb G)\times \bfb H \times \set{0},\, \bfb G\times \rM 1^{>Y}(\bfb H)\times \set 1,\, 
	\\&\rM 1^{>X}(\bfb G)\times \rM 1^{>Y}(\bfb H) \times I,
	\\&\M 1^{<X}(\bfb G)\times \M 1^{<Y}(\bfb H)\times I],
\end{aligned}
\]
specializing to \eqref{E:intro_join_fiber} and \eqref{E:intro_join_base}
in the case that $X$ and $Y$ are maximal or minimal, respectively.
\label{D:rel_join}
\end{defn}

Note that, in view of 
\eqref{E:NpH_below} and \eqref{E:ord_cone_F},
or
\eqref{E:NpH_above} and \eqref{E:ord_cone_B},
the projection maps
\begin{equation}
\begin{aligned}
	\pr_{\bfb{}} : \Np G &= \bff{}\times \bfb{}\times [0,\infty) \to \bfb{} && \text{for $G < X$, or}
	\\ \pr_{\bff{}} : \Np G &= \bff{} \times \bfb{} \times [0,\infty) \to \bff{} && \text{for $G > X$}
\end{aligned}
	\label{E:normal_model_simple_projection}
\end{equation}
are ordered morphisms.
In contrast, neither of the other factor projection maps
\[
\begin{aligned}
	\Np G &= \bff{}\times \bfb{}\times [0,\infty) \to \bff{} \times [0,\infty) = C_X(\bff{}) && \text{for $G < X$, or}
	\\\Np G &= \bff{} \times \bfb{} \times [0,\infty) \to \bfb{} \times [0,\infty) = C_X(\bfb{}) && \text{for $G > X$}
\end{aligned}
\]
are ordered morphisms,
since in the first case, faces of the form $\bff{} \times \M 1(\bfb{}) \times [0,\infty)$, which lie below those of the form $\M 1(\bff{}) \times \bfb{} \times [0,\infty)$ relative to the order, 
are mapped to $\bff{} \times [0,\infty)$, which is situated above faces of the form $\M 1(\bff{}) \times [0,\infty)$ in the order on $C_X(\bff{})$, with a similar order violation in the second case.

The remedies for this are the following `compressed projections', defined relative to the differentials, in the sense of \S\ref{S:bkg_normal}, of a choice of boundary defining functions on $X$,
lifted to
$\Np G$.
These compressed projections play an important role at various points throughout the remainder of the paper.

\begin{lem}
Let $G = \bff{} \times \bfb{}$ be a product-type hypersurface of $X$, and
fix a set of boundary defining functions $\set{\rho_{G'}}$ on $X$, with differentials $d\rho_{G'} = \nu^\ast \rho_{G'}$ for $G' \neq G$ and $t := d\rho_G : \Np G \to [0,\infty)$ (which we assume
coincides with the projection to $[0,\infty)$ with respect to the trivialization $\Np G \cong \bff{}\times \bfb{} \times [0,\infty)$) on $\Np G$.
Set $\rho_{\leq G} = t \prod_{G' < G} \nu^\ast \rho_{G'}$ and $\rho_{\geq G} = t \prod_{G' > G} \nu^\ast \rho_{G'}$ on $\Np G$.

Then the \emph{compressed projections}
\begin{equation}
\begin{aligned}
	\wt \pr_{C_X(\bff{})} := \pr_{\bff{}} \times \rho_{\leq G} : \Np G &\to C_X(\bff{}), & (f,b,t) &\mapsto \big(f, t \prod_{G'<G} \rho_{G'}(b,f)\big), &&\text{for $G < X$, and }
	\\ \wt \pr_{C_X(\bfb{})}:= \pr_{\bfb{}} \times \rho_{\geq G} : \Np G &\to C_X(\bfb{}), &(f,b,t) &\mapsto\big(b, t\prod_{G'> G} \rho_{G'}(b,f)\big), &&\text{for $G > X$}
\end{aligned}
	\label{E:compressed_projection}
\end{equation}
are morphisms of ordered corners.
\label{L:compressed_projection}
\end{lem}
\begin{proof}
The maps are clearly simple and b-normal, and satisfy the ordered condition since in the first case all hypersurfaces of the form $\bff{} \times \M 1(\bfb{}) \times [0,\infty)$ are mapped
to the minimal hypersurfaces $\bff{}\times 0$ and in the second case all hypersurfaces of the form $\M 1(\bff{}) \times \bfb{}\times [0,\infty)$ are mapped 
to the maximal hypersurface $\bfb{} \times 0$.
\end{proof}

We are now in a position to state the main result about product-type hypersurfaces, which identifies the product-type structure of their lifts in $X\ttimes Y$.
\begin{thm}
For product-type boundary hypersurfaces $G = \bff G \times \bfb G \in \M 1(X)$ and $H = \bff H \times \bfb H \in \M 1 (Y)$, the hypersurface
$\tlift G H \in \M 1(X\ttimes Y)$, which is canonically diffeomorphic to $\tlift{G_0}{H_0} \in \M 1(\Np G \ttimes \Np H)$,
has a product-type structure as follows.
\begin{itemize}
\item If $(G,H) < (X,Y)$, then 
\begin{equation}
	\tlift G H \cong (\bff G \jtimes_{X,Y} \bff H)_\tmin \times (\bfb G \ttimes \bfb H)_\tmax
	\label{E:prod_bhs_below}
\end{equation}
with the projections coinciding with the restriction to $\tlift G H \subset \Np G \ttimes \Np H$ of the lifts
$\wt \pr_{C_X(\bff{G})} \ttimes \wt \pr_{C_Y(\bff{H})} : \Np G\ttimes \Np H \to C_X(\bff{G}) \ttimes C_Y(\bff{H})$
and $\pr_{\bfb{G}} \ttimes \pr_{\bfb{H}} : \Np G \ttimes \Np H \to \bfb G \ttimes \bfb H$, respectively.
\item If $(G,H) > (X,Y)$, then 
\begin{equation}
	\tlift G H \cong (\bff G \ttimes \bff H)_\tmin \times (\bfb G \jtimes_{X,Y} \bfb H)_\tmax.
	\label{E:prod_bhs_above}
\end{equation}
with the projections coinciding with the restriction to $\tlift G H \subset \Np G \ttimes \Np H$ of the lifts
$\pr_{\bff{G}} \ttimes \pr_{\bff{H}} : \Np G \ttimes \Np H \to \bff G \ttimes \bff H$ and 
and $\wt \pr_{C_X(\bfb{G})} \ttimes \wt \pr_{C_Y(\bfb{H})} : \Np G\ttimes \Np H \to C_X(\bfb{G}) \ttimes C_Y(\bfb{H})$, respectively.
\end{itemize}
\label{T:ord_bhs_str}
\end{thm}

\begin{proof}
To establish the canonical diffeomorphism $\tlift{G}{H} \cong \tlift{G_0}{H_0}$, we note that the normal differentials of the lifted projections $X\ttimes Y \to X$ and $X\ttimes Y \to Y$ determine ordered morphisms $\Np \tlift{G}{H} \to \Np G$ and $\Np \tlift{G}{H} \to \Np H$ which by the universal property for the ordered product factor through a unique morphism
$\Np \tlift{G}{H} \to \Np G \ttimes \Np H$, and this restricts to a morphism $\tlift{G}{H} \cong \tlift{G}{H}_0 \to \tlift{G_0}{H_0}$.
This map is clearly a diffeomorphism since $G$ (resp.\ $H$) has diffeomorphic neighborhoods in $X$ and $\Np G$ (resp.\ $Y$ and $\Np H$).
Thus it suffices from now on to replace $X$ and $Y$ with $\Np G$ and $\Np H$, respectively.

We consider in detail the case $(G,H) > (X,Y)$ as it has the most relevance in \S\ref{S:fibcorn}; the case that $(G,H) < (X,Y)$ is similar.
From the universal property of the ordered product applied to the ordered morphisms \eqref{E:normal_model_simple_projection} and \eqref{E:compressed_projection} we obtain morphisms
\[
\begin{aligned}
	\pr_{\bff G} \ttimes \pr_{\bff H} &: \Np G \ttimes \Np H \to \bff G \ttimes \bff H, \quad \text{and}
	\\
	\wt \pr_{C_X(\bfb G)} \ttimes \wt \pr_{C_Y(\bfb H)} &: \Np G \ttimes \Np H \to C_X(\bfb G)\ttimes C_Y(\bfb H).
\end{aligned}
\]
As morphisms these are simple and b-normal, and they are also easily seen to be b-surjective (for example using \cite[Lemmas 2.5 and 2.7]{HMM} applied to the b-fibrations 
$\Np G \times \Np H \to \bff G \times \bff H$
and
$\Np G \times \Np H \to C_X(\bfb G)\times C_Y(\bfb H)$, noting that all blow-ups of the domain are either lifts of blow-ups of the target or are transversal).
In particular they are b-fibrations. 
We claim that the restrictions 
\begin{equation}
\begin{aligned}
	\pr_{\bff G} \ttimes \pr_{\bff H} &: \tlift G H  \to \bff G \ttimes \bff H, \quad \text{and}
	\\
	\wt \pr_{C_X(\bfb G)} \ttimes \wt \pr_{C_Y(\bfb H)} &: \tlift G H  \to \bfb G \jtimes_{X,Y} \bfb H \subset C_X(\bfb G)\ttimes C_Y(\bfb H)
\end{aligned}
	\label{E:restricted_lifted_projections}
\end{equation}
of these b-fibrations to the boundary hypersurface $\tlift G H$ are in fact transverse surjective submersions of manifolds with corners, from which it follows that the product map
\begin{equation}
	(\pr_{\bff G} \ttimes \pr_{\bff H}) \times (\wt \pr_{C_X(\bfb G)} \ttimes \wt \pr_{C_Y(\bfb H)}) : \tlift G H \to (\bff G \ttimes \bff H) \times (\bfb G \jtimes_{X,Y} \bfb H)
	\label{E:boundary_product_diffeo}
\end{equation}
is a surjective submersion of compact manifolds with corners of the same dimension.
This latter map cannot be a nontrivial cover or else its descent to $G\times H \to (\bff G \times \bff H) \times (\bfb G \times \bfb H)$ would be a nontrival cover as well; hence 
it must be a diffeomorphism.
Moreover, since $\bff G \ttimes \bff H$ is interior minimal and $\bfb G \jtimes_{X,Y} \bfb H$ is interior maximal, the target in \eqref{E:boundary_product_diffeo} is an ordered product
and the diffeomorphism is an isomorphism of ordered corners manifolds.

To show that \eqref{E:restricted_lifted_projections} are surjective submersions
we note first of all that they are simple b-fibrations (being restrictions to a
boundary hypersurface of simple b-fibrations), and proceed to show that they
have the property discussed in \S\ref{S:bkg}, that each boundary hypersurface of the target is mapped onto by a unique boundary hypersurface of the source.
The boundary hypersurfaces of $\bff G \ttimes \bff H$ are indexed by $(G',H')$ for $G' \geq G$ and $H' \geq H$ (except for $(G,H)$ itself, which corresponds to the interior) and consist of the lifts of $\bface{G'} {\bff G} \times \bface{H'}{\bff H}$.
Each of these is mapped onto uniquely by the face of $\tlift G H$ given by its intersection with $\tlift{G'}{H'}$ in $\Np G \ttimes \Np H$, and all other boundary hypersurfaces of $\tlift{G}{H}$
map to the interior of $\bff G\ttimes \bff H$.
Regarding the second map, boundary hypersurfaces of $\bfb G\jtimes_{X,Y} \bfb H$ can be grouped into four types: 

First are those indexed by $(G',H')$ for $G' \leq X$ and $H' \leq Y$, which arise as the intersection
of $\bfb G \jtimes_{X,Y} \bfb H$ with the lift of $(\bface{G'} {\bfb{G}} \times [0,\infty)) \times (\bface{H'}{\bfb H} \times [0,\infty))$ in the ordered product of cones, and each of these is mapped onto
uniquely by the boundary hypersurface of $\tlift G H$ given by its intersection with $\tlift{G'}{H'}$.

Second are those indexed by 
$(G', H')$ for $X \leq G' < G$ or
and $Y \leq H' < H$ (except $(X,Y)$ itself), which similarly arise as the intersection of $\bfb G \jtimes_{X,Y} \bfb H$ with the lift of $(\bface{G'} {\bfb G} \times [0,\infty)) \times (\bface{H'} {\bfb H} \times [0,\infty))$,
and again each of these is mapped onto uniquely by the boundary hypersurface of $\tlift G H$ given by 
its intersection with $\tlift{G'}{H'}$.

Third are those associated to 
$(G,Y)$ or
$(X,H)$,
given by the intersection of $\bfb G \jtimes_{X,Y} \bfb H$ with the lift of 
$(\bfb G \times 0) \times(\bfb H \times [0,\infty))$
or
$(\bfb G \times [0,\infty)) \times (\bfb H \times 0)$, respectively. 
While the latter boundary hypersurfaces of $C_X(\bfb G)\ttimes C_Y(\bfb H)$ typically have many preimage hypersurfaces in $\Np G \ttimes \Np H$, namely 
$\tlift{G'}{Y}$ for each $G' \geq G$
or 
$\tlift{X}{H'}$ for each $H' \geq H$, of these only 
$\tlift{G}{Y}$
and 
$\tlift{X}{H}$ 
are comparable to $\tlift G H$ itself, so the restriction $\tlift G H \to \bfb G\jtimes_{X,Y} \bfb H$ has unique preimages of these boundary hypersurfaces.

Last  are the boundary hypersurfaces associated to $(G,H')$ for $Y < H' < H$ and $(G',H)$ for $X < G' < G$, given by the intersection of $\bfb G \jtimes_{X,Y} \bfb H$
with the lift of $(\bfb G \times 0)\times(\bface{H'} {\bfb H} \times [0,\infty))$ or $(\bface{G'}{\bfb G} \times [0,\infty))\times (\bfb H \times \set 0)$, respectively, 
and each of these is mapped onto uniquely by the boundary hypersurface of $\tlift G H$ given by its intersection with $\tlift{G}{H'}$ or $\tlift{G'}{H}$, respectively.
All other boundary hypersurfaces of $\tlift{G}{H}$, given by its intersection with $\tlift{G'}{H'}$ for $G' \geq G$ and $H' \geq H$, are mapped to the interior of $\bfb G \jtimes_{X,Y} \bfb H$.

Thus we have shown that \eqref{E:restricted_lifted_projections} are surjective
submersions and it remains to see that they are transverse (i.e., the kernels
of their pointwise differentials intersect trivially). 
As this is a local property we can consider local coordinates, and
we can essentially ignore any interior (meaning non-boundary defining) coordinates, since these are unaffected by boundary blow-ups and are carried along as products locally; 
moreover transversality is clear for interior coordinates since they decompose into factors belonging to the interiors of the factors $\bff \bullet$ or $C_\bullet(\bfb \bullet)$, respectively.
On the other hand, for a simple surjective submersion, the behavior of the boundary defining coordinates is entirely combinatorial: the kernel of the differential of $\tlift G H \to \bff G \ttimes \bff H$ is spanned by the coordinate vector fields of boundary defining coordinates for hypersurfaces mapping to the interior of the target and likewise for $\tlift G H \to \bfb G \jtimes_{X,Y} \bfb H$.
As we have just shown above that these sets of boundary hypersurfaces are complementary, transversality follows, completing the proof.

The proof of the analogous product-type structure of $\tlift G H$ for $(G, H) < (X,Y)$ is essentially the same but with orders reversed, so we omit the details.
\end{proof}

\subsection{Fiber products} \label{S:ord_fib_prod}
We now show that the ordered corners category contains fiber products of appropriately transverse maps.
In general, recall that interior b-maps $f : X \to Z$ and $g : Y \to Z$ are said to be \emph{b-transverse} \cite{KMgen,Joycegc} provided that whenever $f(x) = g(y) = z$, 
\begin{equation}
	\bd f_\ast \bT_x X + \bd g_\ast \bT_y Y = \bT_z Z.
	\label{E:b-transverse}
\end{equation}
Such maps are transverse in the ordinary sense over the interiors of $X$, $Y$, and $Z$, where the b-tangent bundles coincide with the ordinary tangent bundles, and here
it is a standard result that 
\begin{equation}
	X^\circ\times_Z Y^\circ = \set{(x,y) \in X^\circ \times Y^\circ : f(x) = g(y)} \subset X^\circ\times Y^\circ 
	\label{E:prefibprod}
\end{equation}
is a smoothly embedded submanifold.
On the other hand, the closure of this set in the cartesian product $X \times Y$ is typically singular, and even when it can be given a smooth manifold with corners structure, it is almost
never a p-submanifold apart from extremely trivial cases.
The main result below says that the closure of \eqref{E:prefibprod} in the \emph{ordered} product $X\ttimes Y$ is well-behaved, and satisfies the universal property of the fiber product.
\begin{rmk}
In \cite{KMgen} the notion of a `binomial subvariety' is defined, generalizing sets like the closure of \eqref{E:prefibprod} in $X\times Y$, and the authors determine a monoidal condition
under which such a subvariety lifts to be smooth under a generalized blow-up of the ambient manifold, with applications to fiber products more specifically.
In a related but slightly different direction, Joyce in \cite{Joycegc} develops the differential topology of `manifolds with generalized corners', an intrinsic structure that can be considered on a binomial subvariety
without reference to the ambient manifold, and which always contains b-transverse fiber products; in particular the closure of \eqref{E:prefibprod} in $X\times Y$
can always be given a manifold with generalized corners structure when $f$ and $g$ are b-transverse.
Finally, in \cite{Kgc} the theory of generalized blow-up is extended to manifolds with generalized corners, which could be used to characterize resolutions of the singular fiber product.
While these theories apply to the present context, they involve quite a bit of machinery, so we opt for a self-contained treatment that can be understood independently.
\end{rmk}

\begin{lem}[{\cite[Prop.~11.3]{KMgen}}]
If $f : X \to Z$ and $g : Y \to Z$ are b-transverse interior b-maps, 
then the restrictions $f\rst_E : E \to G$ and $g \rst_F : F \to G$ are also b-transverse for every pair of boundary faces $E \in \M {}(X)$ and $F \in \M {}(Y)$
such that $f_\sharp(E) = g_\sharp(F) = G \in \M {} (Z)$.
\label{L:iterated_transv}
\end{lem}
\begin{proof}
The short exact sequence $\bN_x E \to \bT_x X \to \bT_x E$ for $E \in \M{}(X)$ along with similar sequences for $Y$ and $Z$ fit into a commutative diagram
\[
\begin{tikzcd}[sep=small]
	0 \ar[r] & \bN_x E \ar[d,swap, "\bd f_\ast"] \ar[r] & \bT_x X \ar[d, swap, "\bd f_\ast"] \ar[r]& \bT_x E \ar[r] \ar[d, "\bd (f\rst_E)_\ast"] & 0
	\\0 \ar[r] & \bN_z G \ar[r] & \bT_z Z \ar[r] & \bT_z G \ar[r] & 0
	\\ 0 \ar[r] & \bN_y F \ar[u, "\bd g_\ast"] \ar[r] & \bT_y Y \ar[u, "\bd g_\ast"] \ar[r] & \bT_y F \ar[r] \ar[u,swap, "\bd (g \rst_F)_\ast"] & 0
\end{tikzcd}
\]
in light of which it follows that $\bd (f\rst_E)_\ast \bT_x E + \bd (g\rst_F)_\ast \bT_y F = \bT_z G$.
\end{proof}

\begin{thm}
If $f : X \to Z$ and $g : Y \to Z$ be ordered morphisms which are b-transverse,
then the lift of \eqref{E:prefibprod} to $X\ttimes Y$ is a smooth interior p-submanifold
\begin{equation}
	X\ttimes_Z Y := \ol{\set{(x,y) \in X^\circ \times Y^\circ : f(x)=g(y)}} \subset X\ttimes Y.
	\label{E:fib_product}
\end{equation}
Moreover, $X\ttimes_Z Y$ satisfies the universal property of the fiber product
in the category of ordered corners: whenever there are morphisms $W \to
X$ and $W \to Y$ commuting with the morphisms to $Z$, then there exists a unique morphism $W \to X\ttimes_Z Y$ fitting in the commutative diagram
\begin{equation}
\begin{tikzcd}[sep=small]
	W \ar[dd] \ar[rr] \ar[dr, dashed, "\exists !"] && Y  \ar[dd]
	\\ & X\ttimes_Z Y \ar[dl] \ar[ur] &
	\\ X \ar[rr] && Z
\end{tikzcd}
	\label{E:fib_prod_diagram}
\end{equation}
\label{T:fib_prod}
\end{thm}
\begin{proof}
To show that \eqref{E:fib_product} is a p-submanifold, we work locally in coordinates 
$(x,y) \in \bbR_+^{l} \times \bbR^{m}$ for $X$, 
$(x', y') \in \bbR_+^{l'}\times \bbR^{m'}$ for $Y$, 
and $(x'',y'') \in \bbR_+^{l''} \times \bbR^{m''}$ for $Z$,
taken in normal form \eqref{E:bnormal_normal}
for the b-normal maps $f$ and $g$
so that
\[
\begin{aligned}
	f(x,y) &= \big(x^\nu, a(x,y)\big) = (x'', y'') \qquad \text{and}
	\\g(x',y') &= \big((x')^\mu, b(x',y')\big) = (x'', y'')
\end{aligned}
\]

%
By Lemma~\ref{L:iterated_transv}, the maps $y \mapsto a(0,y)$ and $y' \mapsto b(0,y')$ are transverse in the ordinary sense, from which it follows
that $d(a_j - b_j)$, $1 \leq j \leq m''$, are independent over $(x,x') = (0,0)$ and hence in a suitably small neighborhood thereof. 
Thus $\ol y_j = a_j(x,y) - b_j(x',y')$, $1 \leq j \leq m''$, can be taken as 
the first $m''$ of $m + m'$ local coordinates
$(x,x', \ol y) \in \bbR_+^{l}\times \bbR_+^{l'} \times \bbR^{m +m'}$ on $X\times Y$
in which \eqref{E:prefibprod} takes the form
\[
	\set{\tfrac{(x)^{\nu_i}}{(x')^{\mu_i}} = 1, \ \ol y_j = 0 :
	 1 \leq i \leq l'',
	\ 1 \leq j \leq m''},
\]
and it remains only to show that $\set{\tfrac{(x)^{\nu_i}}{(x')^{\mu_i}} = 1 : 1 \leq i \leq l''}$ lifts to an interior p-submanifold.

According to Proposition~\ref{P:p-sub_lift}, this is the case provided that 
each monoid $M \subset \bbZ_+\pair{g_1,\ldots,g_l,h_1,\ldots,h_{l'}}$, $g_i = x_i \pa_{x_i}$, $h_j = x'_j \pa_{x'_j}$,
associated to the blow-up $X\ttimes Y$ meets the subspace $\bigcap_{i=1}^{l''} \ker (\nu_i - \mu_i)$ along a face of $M$.
However, this subspace is spanned explicitly by all sums of the form $g_i + h_j$
such that $f_\sharp(G_i) = g_\sharp(H_j) \in \M 1(Z)$, and since $M$ is generated by a totally ordered chain $\set{g_i + h_j}$,
this intersection is precisely the face of $M$ generated by those $g_i + h_j$ with $\bd f (g_i) = \bd g(h_j)$.
It follows that $X\ttimes_Z Y$ is an interior p-submanifold of $X\ttimes Y$.

To see that it satisfies the universal property, suppose that $W$ admits morphisms to $X$ and $Y$ forming a commutative diagram
\[
\begin{tikzcd}
	W \ar[r] \ar[d] & Y  \ar[d] 
	\\ X \ar[r] & Z
\end{tikzcd}
\]
Then $W$ factors uniquely through a morphism to $X\ttimes Y$, with $W^\circ$ mapping into the lift of $X^\circ \times_{Z^\circ} Y^\circ$.
By continuity, the image of $W$ in $X\ttimes Y$ lies in the p-submanifold $X\ttimes_Z Y$, giving the diagram \eqref{E:fib_prod_diagram}.
\end{proof}

\begin{thm}
The boundary hypersurfaces of $X\ttimes_Z Y$ are given by its intersection with the boundary faces $\tlift G H$ of $X\ttimes Y$ such that $f_\sharp(G) = g_\sharp(H) \in \Mtot(Z)$, and have
the following form:
\begin{itemize}
\item If $G \in \M 1(X)$ and $H \in \M 1(Y)$ are proper hypersurfaces with $f_\sharp(G) = g_\sharp(H) = Z \in \M 0(Z)$, then
\begin{equation}
	(X\ttimes_Z Y) \cap \tlift G H \cong \ol{\set{G^\circ\times_Z H^\circ \times I}} \subset \tlift G H,
	\label{E:fib_prod_bad_face}
\end{equation}
meaning the lift of $G \times_Z H \times I$ within $\tlift G H$, considred as a blow-up of $G\times H \times I$ as in 
\eqref{E:bhs_ordprod_less} or
\eqref{E:bhs_ordprod_gtr}.
\item In all other cases, 
\begin{equation}
	(X\ttimes_Z Y) \cap \tlift G H \cong G\ttimes_K H
	\label{E:fib_prod_good_face}
\end{equation}
is naturally identified with the fiber product of $G$ and $H$ over $K = f_\sharp(G) = g_\sharp(H) \in \Mtot(Z)$.
\end{itemize}
\label{T:bhs_of_fib_prod}
\end{thm}
\begin{rmk}
Taking $Z = \pt$ recovers $X\ttimes_Z Y = X \ttimes Y$ with the structure of its boundary faces.
\end{rmk}

\begin{proof}
As an interior p-submanifold of $X\ttimes Y$, the boundary hypersurfaces of $X\ttimes_Z Y$ are precisely its interior intersections with the hypersurfaces $\tlift G H$ of $X\ttimes Y$, 
and we observe that $X\ttimes_Z Y$ can only meet the interior of $\tlift G H$ provided that $f_\sharp(G) = g_\sharp(H) = K \in \Mtot(Z)$. 
In particular $\Mtot(X\ttimes_Z Y) \cong \Mtot(X)\times_{\Mtot(Z)}\Mtot(Y)$ as pointed ordered sets.

In the first case, if $G \in \M 1(X)$ and $H \in \M 1(Y)$ are proper
hypersurfaces mapping to the interior of $Z$, then it follows that none of the
rational boundary defining function equations defining $X\ttimes_Z Y$ involve
boundary defining functions of $G$ or $H$.
Considering the steps of the blow-up $X\ttimes Y$ as in the proof of Proposition~\ref{P:bhs_ordprod} and the lift of $\set{(x,y) : f(x) = g(y)}$ at each step, 
it follows in particular that in the blow-up of $G\times H$ itself, the lift meets the front face in the set $\ol{\set{G^\circ\times_Z H^\circ \times I}}$, i.e., it is independent
of the interval $I$ since boundary defining functions for $G$ or $H$ may be taken as coordinates along this interval.
With this observation, \eqref{E:fib_prod_bad_face} follows along the lines of the proof of Proposition~\ref{P:bhs_ordprod}.

In the other cases, we show that $\tlift G H_K := X\ttimes_Z Y \cap \tlift G H$ can be identified with the fiber product $G\ttimes_K H$.
If $G = X$, then $K = Z$ by necessity, and since $\tlift X H \cong X \ttimes H$ in this case, it is straightforward to see that the lift of $\set{(x,y) \in X \times Y : f(x) = g(y)}$
to $X\ttimes Y$ meets $\tlift X H$ in a set equivalent to the lift of $\set{(x,y) \in X \times H : f(x) = g(y)}$ to $X\ttimes H$.
A similar consideration applies to the case that $H = Y$. 

Finally, if $G$, $H$, and $K$ are all proper hypersurfaces, then our strategy is to show that $\tlift G H_K$
satisfies the universal property of the fiber product $G\ttimes_K H$,
namely that every commutative diagram formed by the outermost square of 
\begin{equation}
\begin{tikzcd}[sep=small]
	F \ar[dd,"\phi"] \ar[rr,"\psi"] \ar[dr, dashed, "\exists !"] && H  \ar[dd,"g"]
	\\ & \tlift G H_K \ar[dl] \ar[ur] &
	\\ G \ar[rr,"f"] && K
\end{tikzcd}
	\label{E:bhs_fib_diagram}
\end{equation}
determines a unique morphism $F \to \tlift G H_K$ filling in the remainder of the diagram.
It suffices to replace $X$, $Y$, and $Z$ by the normal models $\Np G$, $\Np H$ and $\Np K$, respectively. 
We proceed by first showing that the outer square of \eqref{E:bhs_fib_diagram} can be `thickened out' to a diagram of ordered corners morphisms of the form
\begin{equation}
\begin{tikzcd}
	F\times \bbR_+ \ar[r,"\wt \psi"] \ar[d,"\wt \phi"] & \Np H \ar[d, "dg"]
	\\ \Np G \ar[r, "df"]  & \Np K
\end{tikzcd}
	\label{E:bhs_fib_diagram_thick}
\end{equation}
uniquely up to an automorphism of $F\times \bbR_+$ which restricts to the identity on $F$,
whence the restriction of the map $F\times \bbR_+ \to \Np G\ttimes_{\Np K} \Np H$ to $F\times \set 0$ provides the required unique morphism characterizing 
$\tlift{G}{H}_K$ as the fiber product $G\ttimes_K H$.

Indeed, with respect to trivializations of the normal models, $df : \Np G \to \Np K$ has the form
\[
\begin{gathered}
	\Np G \cong G \times \bbR_+ \stackrel{df}{\to} K\times \bbR_+ \cong \Np K
	\\ (p,t) \mapsto \big(f(p), t\sigma(p)\big)
\end{gathered}
\]
where $\sigma = \prod_{G' \in f_\sharp^\inv(K)} \rho_{G'}$ is a product of boundary defining functions on $G$ for those faces $\pa_{G'} G$ which map to $K$ itself.
Similarly, $dg : \Np H \to \Np K$ has the form $(p',t') \mapsto \big(g(p'), t' \sigma'(p')\big)$ for a product $\sigma' = \prod_{H' \in g_\sharp^\inv(K)} \rho_{H'}$ on $H$.

Fixing boundary defining functions on $F$, it follows from commutativity of \eqref{E:bhs_fib_diagram} that $\phi^\ast \sigma = a_G \prod_{F' : F' \mapsto K} \rho_{F'}$
and $\psi^\ast \sigma' = a_H \prod_{F' : F' \mapsto K} \rho_{F'}$ for the same set $\set{\rho_{F'} : F' \mapsto K}$ of boundary defining functions on $F$, namely those associated to hypersurfaces mapped to the interior of $K$ by $\psi\circ g = \phi \circ f$ 
in \eqref{E:bhs_fib_diagram}, but generally different positive functions $a_G, a_H \in C^\infty(F; (0,\infty))$. 
But then \eqref{E:bhs_fib_diagram_thick} commutes if we define
\[
\begin{aligned}
	\wt \phi &: F\times \bbR_+ \ni (q,t) \mapsto \big(\phi(q), a_G^\inv(q) t\big) \in G \times \bbR_+ \cong \Np G
	\\ \text{and} \qquad 
	\wt \psi &: F\times \bbR_+ \ni (q,t) \mapsto \big(\psi(q), a_H^\inv(q) t\big) \in H \times \bbR_+ \cong \Np H
\end{aligned}
\]
Moreover, these are the only possible extensions of $\phi$ and $\psi$ up to an $(0,\infty)$-equivariant automorphism of $F\times \bbR_+$ of the form $(q,t) \mapsto (q,b(q)t)$ for $b \in C^\infty(F; (0,\infty))$, which in particular restricts to the identity on $F\times \set 0$, 
so the restriction $F \times \set 0 \to \tlift{G}{H}_K$ of the map $F\times \bbR_+ \to \Np G \ttimes_{\Np K} \Np H$ 
induced by any such choice is unique, giving the characterization of $\tlift{G}{H}_K$ as the fiber product $G\ttimes_K H$. 
Note that the above strategy fails if $K = Z$, since in this case the maps $\Np G \cong G \times \bbR_+ \to Z$ and $\Np H \cong H \times \bbR_+ \to Z$ are respectively independent of $t \in \bbR_+$, and while extensions $\wt \phi : F \times \bbR_+ \to \Np G$ and $\wt \psi : H \times \bbR_+ \to \Np H$ may be defined, they are certainly not unique.
\end{proof}

As an application of the previous result, note that 
if $f : X \to Y$ is any morphism, then the fiber product of $f$ with the identity $1 : Y \to Y$ (which is b-transversal to every b-map) gives an embedding of the \emph{graph} $\Gr(f) = X\ttimes_Y Y \subset X\ttimes Y$ of $f$ as a p-submanifold of $X\ttimes Y$.
In particular this shows that every morphism can be factored as a p-submanifold inclusion $X \hookrightarrow \Gr(f) \subset X\ttimes Y$ and a b-submersion $X\ttimes Y \to Y$. 

Another application, which will be important in the forthcoming work \cite{KR1}, is the following.
Let $\rho_X = \prod_{G \in \M 1(X)} \rho_G$ and $\rho_Y = \prod_{H \in \M 1
(Y)} \rho_H$ be total boundary defining functions on interior minimal manifolds
$X$ and $Y$ with ordered corners, respectively.
These constitute ordered corners morphisms $\rho_X : X \to [0,\infty)_\tmin$
and $\rho_Y : Y \to [0,\infty)_\tmin$, with target given the interior minimal
ordered corners structure, and these are b-transverse in a sufficiently small
neighborhood of the boundaries of $X$ and $Y$ where $d\rho_X \neq 0$ and $d
\rho_Y \neq 0$.

\begin{cor}
In the situation above, the fiber product $\set{\rho_X/\rho_Y = 1} \subset X \ttimes Y$ is a well-defined p-submanifold sufficiently near the boundary of $X\ttimes Y$, meeting $\tlift G H$ in a p-submanifold
of the form $G\ttimes H = G\ttimes_{\set 0} H$ for each pair $(G,H) \in \M 1(X)\times \M 1(Y)$.
\label{C:s_equals_1}
\end{cor}

\section{Fibered corners} \label{S:fibcorn}
We now turn to the consideration of manifolds with fibered corners.
\begin{defn}[c.f.~\cite{DLR}]
A \emph{fibered corners structure} 
(also known as a `resolution structure' \cite{AM} or an `iterated boundary fibration structure' \cite{ALMP}) on a manifold with corners $X$ consists of the structure
of a fiber bundle\footnote{In fact, by Corollary~\ref{C:loc_trivial} it is sufficient for $\bfib G$ to be a fibration, i.e., a proper surjective submersion; 
by an Ehresmann lemma for manifolds with fibered corners (Proposition~\ref{P:fib_strat_lift}), such a map is locally trivial.} of manifolds with corners
\[
	\bfib G : G \to \bfb G 
\]
for each (collective) boundary hypersurface $G \in \M 1(X)$, 
with typical fiber\footnote{If $\bfb G$ is disconnected, different connected components may have different fibers.} denoted $\bff G$, satisfying the condition
that 
whenever $G \cap G' \neq \emptyset$ then $\dim(\bfb G) \neq \dim(\bfb{G'})$;
moreover
if $\dim(\bfb G) < \dim(\bfb{G'})$, say, then 
$\bfib {G'}$ constitutes a coarser fibration on $G \cap G'$ compared to $\bfib {G}$
in the following sense:
\begin{itemize}
\item 
$\bfib {G}$ maps $G \cap G'$ surjectively onto $\bfb G$ (in other words, $G\cap G'$ is `horizontal' with respect to $\bfib{G}$),
whereas $G \cap G' = \bfib{G'}^\inv(\bface G{\bfb{G'}})$ for a boundary hypersurface
$\bface G {\bfb{G'}}$ of $\bfb{G'}$ 
(in other words, $G \cap G'$ is `vertical' with respect to $\bfib {G'}$), and
\item $\bface G {\bfb {G'}}$ has a fiber bundle structure
$
	\phi_{G G'} : \bface{G}{\bfb{G'}} \to \bfb G
$
such that the following diagram commutes:
\begin{equation}
\begin{tikzcd}
	G \cap G' \ar[r, "\phi_{G'}" ] \ar[dr,swap, "\phi_{G}"] & \pa_G \bfb{G'} \ar[d, "\phi_{GG'}"]
	\\ & \bfb G.
\end{tikzcd}
	\label{E:fc_diagram}
\end{equation}
\end{itemize}
This
induces an ordering on $\M 1(X)$ in the sense of Definition~\ref{D:ordcorn} (to be extended to $\Mtot(X)$ momentarily)
defined by 
\begin{equation}
	G < G' \iff
	G \cap G' \neq \emptyset
	\quad \text{and}\quad
	\dim(\bfb G) < \dim(\bfb {G'}).
	\label{E:size_ord}
\end{equation}

We require here in addition the structure of a fiber bundle $\phi_X : X \to \bfb X$ on $X$ itself, with fiber $\bff X$.
Of course $X \cap G = G \neq \emptyset$ for every $G \in \M 1(X)$, and in order to extend the compatibility condition above, there are two possibilities. 
We require either:
\begin{itemize}
\item 
$G = X \cap G$ is the restriction of $\phi_X$ over a proper boundary hypersurface $\pa_G \bfb X$ of $\bfb X$, which fibers over $\bfb G$ forming a commutative diagram
\[
\begin{tikzcd}
	G \ar[r, "\phi_X"] \ar[dr,swap, "\phi_G"] & \pa_G \bfb X \ar[d, "\phi_{GX}"]
	\\ &\bfb G
\end{tikzcd}
\]
in which case we set $G < X$, or else
\item
$G$ is trivially the restriction of $\phi_G$ over the `honorary boundary hypersurface' $\pa_X \bfb G = \bfb G$, which fibers over $\bfb X$
forming a commutative diagram
\[
\begin{tikzcd}
	G \ar[r, "\phi_G"] \ar[dr,swap, "\phi_X"] & \pa_X \bfb G = \bfb G \ar[d,"\phi_{XG}"]	
\	\\ &\bfb X
\end{tikzcd}
\]
in which case we set $X < G$.
\end{itemize}
This extends the order on $\M 1(X)$ to the set of principal faces $\Mtot(X)$, and in particular
every manifold with fibered corners is naturally a manifold with ordered corners
in the sense of Definition~\ref{D:ordcorn}.
In the second case above we relax the strict dimension comparison on the base manifolds to require merely that $\dim(\bfb X) \leq \dim(\bfb G)$, to account
for the possibility that $\phi_{XG} = \id$ has trivial fibers (which may occur).
As a matter of notation, we denote the fiber of $\bfib{GG'} : \bface{G}{\bfb {G'}} \to \bfb G$ by 
\[
	\bfe{G}{G'} := \bfib{GG'}^\inv(\pt).
\]
\label{D:fc}
\end{defn}

\begin{rmk}
\mbox{}
\begin{itemize}
\item 
For the definition we do not necessarily require that boundary hypersurfaces are connected; 
as in \cite{AM}
we allow $G$ to be 
a
`collective boundary hypersurface'
i.e., a union of disjoint connected boundary hypersurfaces.
One reason for this is that a fiber $\bff G$, a fibered corners manifold in its own right,
may have distinct boundary hypersurfaces which are the restrictions to the fiber of a single connected boundary hypersurface of the total space,
as in the example of a finite width M\"obius strip fibering over a circle.

Another reason comes from the connection to stratified spaces as discussed in \S\ref{S:strat}, in which the $\bfb G$ and $\bff G$ are closely related to the strata and associated links, respectively,
of a stratified space. 
In the latter theory one often wants to consider stratifications in which the links, and hence the $\bff G$, may be disconnected; for example, a torus pinched along a non-bounding circle
has a singular stratum whose link consists of two disjoint circles, and the associated manifold with fibered corners is a cylinder with both boundary components considered as a single collective boundary
hypersurface fibering over a point.

Note that if $\bfb G$ is disconnected, then different components of the fibration $G$ may have non-diffeomorphic typical fibers; unfortunately our notational convention, in which we denote the fiber systematically by $\bff G$, doesn't support this scenario.
Rather than complicate the notation further, we leave it to the reader to make the necessary adjustments.
\item
In the existing literature it is typical to only consider fibrations only on the proper boundary hypersurfaces,
and in the absence of an explicit fibration on $X$ itself, there are two canonical choices: namely $\phi_X$ can be either of the trivial fibrations: $\phi_X = \id : X \to X$, making $X$ interior maximal, or $\phi_X : X \to \pt$, making $X$ interior minimal. 
We denote these by $X_\tmax$ and $X_\tmin$, respectively.

In Appendix~\ref{S:tubes} we characterize interior fibrations $\bfib X : X \to \bfb X$ as certain surjective submersions in the category of fibered corners manifolds: see Definition~\ref{D:int_fibn}
and Corollary~\ref{C:loc_trivial}.
\end{itemize}
\end{rmk}

As a locally trivial fiber bundle,
the boundary 
hypersurfaces
of $G$ (given by $G \cap G'$) come in two types, either the preimage with respect to $\phi_G$ of a boundary face $\pa_{G'} \bfb G$ of the base when $G' < G$, or a subbundle
of $\phi_G$ with fiber a (possibly disconnected) boundary face of $\pa_{G'} \bff G$ when $G' > G$.
From this the next result immediately follows.

\begin{prop}
For each $G \in \M 1(X)$, $\bff G$, $\bfb G$, and $G$ itself all inherit fibered corners structures as follows:
\begin{itemize}
\item $\Mtot(\bfb G) \cong \set{G' \in \M 1(X): G' \leq G}$, with fibrations $\bfib{G' G} : \bface{G'}{\bfb G} \to \bfb {G'}$
and the identity fibration on $\bfb G$ itself.
In particular $\bfb G$ is interior maximal.
\item $\Mtot(\bff G) \cong \set{G' \in \M 1(X) : G' \geq G}$, with fibrations $\bfib{G'} : \bface{G'} {\bff G} \to \bfe{G}{G'}$
(obtained by restricting \eqref{E:fc_diagram} over $\pt \in \bfb G$), and the trivial fibration $\bff G \to \pt$ on the interior.
In particular $\bff G$ is interior minimal.
\item $\Mtot(G) \cong \set{G' \in \M 1(X) : G' \sim G}$, with fibrations $\bfib{G'} : \bface{G'} G := G \cap G' \to \bfb{G'}$ for
$G' \leq G$ and $\bfib{G'} : \bface{G'} G := G \cap G' \to \bface{G}{\bfb{G'}}$ for $G' > G$.
\end{itemize}
In particular every $G \in \M 1(X)$ is locally a product type hypersurface in the sense of \S\ref{S:ord_bhs}, given locally by the ordered product $G \stackrel{\text{loc}}= \bff G \ttimes \bfb G = \bff G \times \bfb G$.
\label{P:induced_fc}
\end{prop}

Prior to defining morphisms of manifolds with fibered corners, we first consider some natural conditions that may be imposed.

\begin{defn}
Let $X$ and $Y$ be manifolds with fibered corners.
We say a b-map $f : X \to Y$ is \emph{fibered} if the restriction of $f$ to each $G \in \Mtot(X)$ is a map of fiber bundles: thus
\begin{equation}
\begin{tikzcd}
	G \ar[d, swap, "\bfib G"] \ar[r, "f\rst_G"] & H = f_{\#}(G) \ar[d, "\bfib H"]
	\\ \bfb G \ar[r, "f_G"] & \bfb H
\end{tikzcd}
	\label{E:fibered}
\end{equation}
for some map $f_G : \bfb G \to \bfb H$, which, as a consequence of the fibered property of $f$ for each $G' < G$, is itself fibered.
\label{D:fibered_map}
\end{defn}
This condition is natural from the point of view of stratified spaces (see \S\ref{S:strat}), since a fibered map $f : X \to Y$ descends to a well-defined
map $\ol f : \strat X \to \strat Y$ of the associated stratified spaces, obtained by collapsing the fibers of the respective boundary fibrations of $X$ and $Y$.
As it turns out, this condition alone is not sufficient to determine a category of fibered manifolds containing products; we need to impose some additional structure,
in support of which we introduce the following definitions.

\begin{defn}
Let $X$ be a manifold with fibered corners. 
\begin{enumerate}
\item 
A smooth function $u \in C^\infty(X;\bbR)$ is said to be \emph{basic at $G \in \M 1 (X)$} if its restriction to $G$ is pulled back from $\bfb G$:
$u \rst_G = \bfib G^\ast u_G$ for some $u_G \in C^\infty(\bfb G;\bbR)$; equivalently, $u$ is constant on fibers of $G$.
A function which is basic at every proper boundary hypersurface $G \in \M 1(X)$ will be simply called \emph{basic}.

\item
A boundary defining function $\rho_G$ for $G$ is \emph{good} if it is basic at $G'$ for every $G' > G$; in particular
$\rho_G \rst_{G'}$ is the pullback of a boundary defining function for $\bface G {\bfb {G'}}$ on $\bfb{G'}$. 
Note that it is not possible for $\rho_G$ to be basic at $G' < G$.

\item 
Two boundary defining functions $\rho_G$ and $\rho'_G$ for the same $G \in \M 1(X)$ are \emph{fc-equivalent} if the ratio $\rho_G/\rho'_G$ is basic (constant on fibers of each $G' \in \M 1(X)$).
This is an equivalence relation.
Note that for good boundary defining functions, $\rho_G/\rho'_G$ is automatically basic at every $G' > G$, so it is enough to require that the ratio
be basic at each $G' \leq G$.

\item
A simple, b-normal b-map $f : X \to Y$ is \emph{compatible} with given fc-equivalence classes of boundary defining functions on $X$ and $Y$, respectively,
provided that for some (and hence any) representatives $\set{\rho_G : G \in \M 1(X)}$ and $\set{\rho_H : H \in \M 1(Y)}$, 
\[
	f^\ast(\rho_H) = a_H \prod_{G \in f_\#^\inv(H)} \rho_G \implies a_H \in C^\infty(X; (0,\infty)) \text{ is basic for every $H \in \M 1(Y)$.}
\]
Equivalently, as a consequence of the normal form for b-normal maps discussed in \S\ref{S:bkg}, a map $f$ is compatible if and only if it is rigid with respect
to some representative boundary defining functions for the equivalence classes on $X$ and $Y$.
\end{enumerate}
\label{D:basic_bdf}
\end{defn}

\begin{lem}
An fc-equivalence class of good boundary defining functions for $G' \in \M1(X)$ on $X$ induces an fc-equivalence class of good boundary defining functions on each $G \in \M 1(X)$ such that $G \cap G' \neq \emptyset$, and likewise on $\bfb G$ if $G' < G$ and $\bff G$ if $G' > G$.
\label{L:induced_equiv_bdf}
\end{lem}
\begin{proof}
For $G$ and $\bff G$ the equivalence class is given by restriction of representative elements. 
For $\bfb G$, it follows by definition that 
$\rho_{G'} \rst_G = \phi_G^\ast \rho'_{G'}$ for a boundary defining
function $\rho'_{G'}$ for $\bface{G'}{\bfb{G}}$ on $\bfb G$; moreover the ratio of two such fc-equivalent boundary defining functions for $G'$ is pulled back from $\bfb{G''}$ 
when restricted over any $G'' < G$, hence they descend to fc-equivalent 
boundary defining functions 
for
$\bface{G'}{\bfb G}$ 
on $\bfb G$.
\end{proof}

It is a consequence of Proposition~\ref{P:rel_tube} (see also \cite[Prop.~3.7]{AM} or \cite[Lem.~1.4]{DLR})
that good boundary defining functions always exist. 
What is more, having fixed fc-equivalence classes of good boundary defining functions, local coordinates on $X$ can always be chosen of the form
\begin{equation}
	(y_{-n'},x_{-n'},y_{-n'+1},x_{-n'+1},\ldots,y_0,y_1,x_1,y_2,x_2,\ldots,x_n,y_{n+1})
	\label{E:coord_std_form}
\end{equation}
where the $x_i$ are representative local boundary defining functions for the
hypersurfaces $G_i$ forming a maximal chain $G_{-n'} < \cdots < G_{-1} < X < G_1 < \cdots < G_n$ and $y_i$ are coordinates in $\bbR^{k_i}$, such that
the boundary fibration $\bfib i = \bfib{G_i}$ for $-n' \leq i \leq n$ is given by the projection
\begin{equation}
	\bfib i : 
	(y_{-n'},x_{-n'},\ldots,x_{-1},y_0,y_1,x_1,y_2,x_2,\ldots,x_n,y_{n+1})\big|_{ x_i=0}
\mapsto 
	(y_{-n'},x_{-n'},\ldots,y_i)
	\label{E:coord_std_form_fibn}
\end{equation}
including the case $i = 0$ which corresponds to $\bfib X$.
This is proved as Corollary~\ref{C:str_coords} below.
In particular $(y_{-n'},x_{-n'},\ldots,y_i)$ are coordinates for $\bfb{G_i}$ and $(y_{i+1},x_{i+1},\ldots,y_{n+1})$ are coordinates for $\bff{G_i}$.
We say such coordinates are in \emph{standard form}.
\begin{rmk}
When $X$ is maximal, we may opt to enumerate hypersurfaces and coordinates with positive instead of negative integers, so using coordinates $(y_1,x_1,\ldots,x_n,y_{n+1})$
in which $\bfib i$ is still given by 
\[
	\bfib i : (y_1,x_1,\ldots,y_i,0,y_{i+1},\ldots,y_n,x_n,y_{n+1}) \mapsto (y_1,x_1,\ldots,y_i).
\]
\end{rmk}

We are now in a position to define the category of manifolds with fibered corners in which the ordered product will be shown to be the categorical product.

\begin{defn}
The \emph{category of manifolds with fibered corners} is defined as follows.
\begin{itemize}
\item An \emph{object} is a manifold with fibered corners $X$ equipped with an fc-equivalence class of good boundary defining functions for each $G \in \M1(X)$.
\item A \emph{morphism} is an ordered morphism $f : X \to Y$ in the sense of Definition~\ref{D:ordcorn} (i.e.\ a simple, interior, b-normal b-map for which $f_\sharp : \Mtot(X) \to \Mtot(Y)$ is ordered) which is fibered in the sense of Definition~\ref{D:fibered_map} and compatible with the fc-equivalence classes of boundary defining functions in the sense of Definition~\ref{D:basic_bdf}.
\end{itemize}
\label{D:fc_cat}
\end{defn}

The following result is a simple consequence of unwinding definitions.
\begin{lem}
For a manifold with fibered corners $X$, the fibration $\phi_G : G \to \bfb G$ for each $G \in \Mtot(X)$ is a morphism of fibered corners manifolds.
\label{L:bfib_mor}
\end{lem}

\begin{rmk}
Note that the condition that the boundary defining functions on $X$ are good is equivalent to the rigidity of each boundary fibration $\bfib G : G \to \bfb G$, $G \in \M 1(X)$.
\end{rmk}

\subsection{Products and fiber products} \label{S:fibcorn_product}
We now proceed to show that the ordered product $X\ttimes Y$, equipped with a suitable set of fibrations and fc-equivalence classes of boundary defining functions, is also a product in the category of manifolds with fibered corners.

The fibrations on $\tlift G H$ for $(G,H) \leq (X,Y)$ are straightforward.
Indeed, the composite maps $\tlift G H \to G \to \bfb G$ and $\tlift G H \to H \to \bfb H$ are ordered morphisms, so the universal property of $\bfb G \ttimes \bfb H$ as a manifold with ordered corners implies that there is a unique consistent map
\[
	\bfib{G,H} : \tlift G H \to \bfb G\ttimes \bfb H,
	\qquad (G,H) \leq (X,Y).
\]
and it then follows 
from a local application of Theorem~\ref{T:ord_bhs_str} that
this is in fact a fiber bundle with typical fiber the join $\bff G
\jtimes_{X,Y} \bff H$ (or simply $\bff G \ttimes \bff H$ if $G = X$ or $H = Y$).
%

%
Moreover, by the universal property of products in the ordered corners category, it follows that whenever $W$ is a manifold with fibered corners with morphisms $f : W \to X$ and $g : W \to Y$ thereby inducing
an ordered morphism $W \to X\ttimes Y$, the restriction of this morphism to any $E \in \Mtot(W)$ with $E \leq W$ fits into a naturally commutative diagram
\[
\begin{tikzcd}
	E \ar[r] \ar[d,"\bfib E"] & \tlift G H \ar[d, "\bfib{\tlift G H}"]
	\\ \bfb E \ar[r] & \bfb G\ttimes \bfb H
\end{tikzcd}
\qquad G = f_\sharp(E), \quad H = g_\sharp(E),
\]
with the map $\bfb E \to \bfb G \ttimes \bfb H$ induced by $f_E : \bfb E \to \bfb G$ and $g_E : \bfb E \to \bfb H$.
In other words, the map $W \to X \ttimes Y$ is canonically fibered over each $E \leq W \in \Mtot(W)$.

Note that none of the above makes any reference to fc-equivalence classes of boundary defining functions on $X$, $Y$, or $W$.
In fact, it follows from this observation that, provided we restrict consideration to \emph{interior maximal} manifolds with fibered corners, we can define a weaker category in which the objects are merely required to have fibered corners (without specifying equivalence classes of boundary defining functions) and morphisms are merely required to be ordered and fibered; what we have just shown is that $X\ttimes Y$ is a product in this category:
\begin{thm}
The ordered product is a categorical product in the category whose objects are interior maximal manifolds with fibered corners and whose morphisms
are ordered morphisms in the sense of Definition~\ref{D:ordcorn} which are fibered in the sense of Definition~\ref{D:fibered_map}.
\label{T:weak_category}
\end{thm}

However, when the manifolds are not interior maximal, the extra data (of equivalence classes of boundary defining functions and compatible maps) is required to get a theory with products, as we now show.
Note that
while there is always a canonical map $\tlift G H \to \bfb G \ttimes \bfb H$ consistent with the composite projections $\tlift G H \to \bfb G$ and $\tlift G H \to \bfb H$, 
\emph{this is not a fiber bundle} 
when $(G,H) > (X,Y)$,
as can be seen from Theorem~\ref{T:ord_bhs_str}.
Instead, we expect $\tlift G H$ for $(G,H) > (X,Y)$ to fiber over the join $\bfb G \jtimes_{X,Y} \bfb H$.
To define these fibrations we consider the normal models $\Np G$ for $G \in \M 1(X)$.
\begin{defn}
Let $G \in \M 1(X)$.
The fibered corners structure on the normal model $\Np G$ consists of 
fibrations given by the differentials of those from $X$:
\[
\begin{gathered}
	d\bfib{G'} :\Np G \rst_{G' \cap G} \to \Np \pa_G \bfb {G'}, 
	\qquad \text{for $G' > G$},
	\\ 
	d\bfib{G'}: \Np G \rst_{G' \cap G} \to \bfb {G'}
	\qquad \text{for $G' \leq G$},
\end{gathered}
\]
where the latter factors as the composite $\Np G \to G \to \pa_{G'} \bfb G\to \bfb {G'}$.
In particular, the interior of $\Np G$ is equipped with the fibration $d\bfib{X} : \Np G \to \pa_X \bfb G = \bfb G \to \bfb X$ in case $G > X$ and 
$\Np G \to \Np \pa_G \bfb X$ in case $G < X$.
To keep the notation uncluttered, when $G' = G$ we denote $d\bfib{G}$ simply by $\bfib G : \Np G \to \bfb G$.
The order structure associated with these fibrations is consistent with \eqref{E:NpH_below} and \eqref{E:NpH_above}.

Fc-equivalence classes of good boundary defining functions on $X$ induce equivalence classes of $(0,\infty)$-equivariant defining functions on $\Np G$ 
by taking their normal derivatives as in \S\ref{S:bkg_normal}; for notational
convenience we shall simply refer to these by the same notation as the boundary defining functions on $X$ from which they are derived.
\label{D:Np_bfib}
\end{defn}

We also equip the relative cone $C_X(\bfb G)$ of Definition~\ref{D:rel_cone} with a fibered corners structure.
\begin{defn}
Let $G > X$.
The fibered corners structure on $C_X(\bfb G) = \bfb G \times [0,\infty)$ consists of fibrations
\[
	\bfb G\times \set 0 \to \bfb G, \qquad \pa_{G'} \bfb G \times [0,\infty) \to \bfb {G'},
	\qquad C_X(\bfb G) = \bfb G \times [0,\infty) \to \bfb X
\]
given by the obvious maps coming from the boundary fibrations $\bfib{G'G} : \bface{G'} {\bfb G} \to \bfb {G'}$ 
and using $\bfib{XG} : \bface{X}{\bfb G} = \bfb G \to \bfb{X}$ on $\bfb G$ itself (rather than the identity fibration) on the interior $\bfb G \times [0,\infty)$.
The ordered structure induced by these fibrations is consistent with Definition~\ref{D:rel_cone}.
Fc-equivalence classes of good boundary defining functions on $C_X(\bfb G)$ are given by pulling back $\bface{G'}{\bfb G}$ defining functions from $\bfb G$ for the boundary faces
$\bface{G'}{\bfb G} \times [0,\infty)$ and by the projection to $[0,\infty)$ for $\bfb G\times 0$.
\label{D:CX_bfib}
\end{defn}

Though we will not use it immediately, this is an appropriate moment to record the fibered corners structure of $C_X(\bff G)$ for $G < X$:
\begin{defn}
Let $G < X$.
The fibered corners structure on $C_X(\bff G) = \bff G \times [0,\infty)$ consists of boundary fibrations
\[
	C_X(\bff G) \to 
	\bfe{G}{X}\times[0,\infty)
	\qquad \bff G\times \set 0 \to \pt, \qquad \pa_{G'} \bff G \times [0,\infty) \to \bfe{G}{G'} \times [0,\infty)
\]
where we recall that $\bfe{G}{G'} = \bfib{GG'}^\inv(\pt)$, the fiber of $\bfib{GG'}: \bface{G}{\bfb {G'}} \to \bfb G$, forms the base of the boundary fibration
on $\bface{G'} {\bff G}$, and here we use the fibration induced by $\bfib{X}$ instead of $\bfib{G}$ on the interior $\bff G \times [0,\infty)$ itself.
The ordered structure induced by these boundary fibrations is consistent with the order structure in Definition~\ref{D:rel_cone}.
%
\label{D:CX_bfib_fiber}
\end{defn}

\begin{rmk}
The fibered corners structure induced on $\bfb G \jtimes_{X,Y} \bfb H$ as a boundary hypersurface of $C_X(\bfb G) \ttimes C_Y(\bfb H)$ in the case that $(G,H) > (X,Y)$ is the same as the one induced by considering it as the base of the hypersurface $\tlift{G}{H}$ of $X\ttimes Y$, and likewise, the 
fibered corners structure on $\bff G \jtimes_{X,Y} \bff H$ as a boundary hypersurface of $C_X(\bff G) \ttimes C_Y(\bff H)$ in the case that $(G,H) < (X,Y)$
is the same as the one induced by considering it as a fiber of $\tlift{G}{H}$ in $X\ttimes Y$.
\end{rmk}

The key observation leading to a well-defined theory of products in the fibered corners category is the following result, in which the compressed projection defined in 
\S\ref{S:ord_bhs} amounts to a natural thickening of the boundary fibrations $\bfib G : G \to \bfb G$ to morphisms $\tbfib G : \Np G \to C_X(\bfb G)$.

\begin{lem}
Let $G \in \M 1^>(X)$ and fix a choice of representative boundary defining functions $\set{\rho_{G'}}$ on $X$, denoting their lifts to $\Np G$ by the same notation.
\begin{enumerate}
\item 
\label{I:main_lemma_normal_one}
The compressed projection
\begin{equation}
	\tbfib G := \wt \pr_{C_X(\bfb G)}
	: \Np G \to C_X(\bfb G) 
	\qquad v \mapsto \big(\bfib G(v),\, \rho_{\geq G}(v)\big)
	\label{E:thickened_phi}
\end{equation}
is a fibered corners morphism extending $\bfib G : G \to \bfb G$, and is independent of the choice of representatives $\set{\rho_{G'}}$ up to an automorphism of $C_X(\bfb G)$. 
\item
\label{I:main_lemma_normal_funct}
The map \eqref{E:thickened_phi} is functorial in the following sense:
given a morphism $f: X \to Y$ sending $G > X$ to $H > Y$,
there exists a morphism $\wt f_G : C_X(\bfb G) \to C_Y(\bfb H)$ extending $f_G : \bfb G \to \bfb H$
such that the following diagram commutes:
\[
\begin{tikzcd}
	\Np G \ar[r,"df"] \ar[d,swap, "\tbfib G"] & \Np H \ar[d, "\tbfib H"]
	\\ C_X(\bfb G) \ar[r, "\wt f_G"] & C_Y(\bfb H)
\end{tikzcd}
\]
\end{enumerate}
\label{L:main_lemma_normal}
\end{lem}
\begin{proof}
The map \eqref{E:thickened_phi} is clearly simple, b-normal, and compatible with the fc-equivalence classes of boundary defining functions on $\Np G$ and $C_X(\bfb G)$.
The fact that it is fibered follows
from the commutative diagrams
\[
\begin{tikzcd}
	\Np G \rst_{G\cap G'} \ar[r,"\tbfib G"] \ar[d, "d\bfib{G'}"] & \bface{G'} {\bfb G} \times [0,\infty) \ar[d, "\bfib{G'}"]
	\\ \bfb{G'} \ar[r, "1"] & \bfb{G'}
\end{tikzcd}
\qquad
\begin{tikzcd}
	\Np G \rst_{G\cap G'} \ar[r, "\tbfib G"] \ar[d, "d\bfib{G'}"] & \bfb G\times \set 0 \ar[d, "\bfib{G}"]
	\\ \Np {\bface{G}{\bfb{G'}}} \ar[r, "d\bfib{GG'}"] & \bfb G
\end{tikzcd}
\qquad
\begin{tikzcd}
	G \times \set 0 \ar[r, "\tbfib G = \bfib G"] \ar[d, "\bfib G"] & \bfb G \times \set 0 \ar[d,"1"]
	\\ \bfb G \ar[r, "1"] & \bfb G
\end{tikzcd}
\]
for $G' < G$, $G' > G$, and $G' = G$, respectively.
Suppose now that $\set{\rho'_{G'}}$ is another choice of fc-equivalent boundary defining functions on $X$, defining another such map $\tbfib G ' : \Np G \to C_X(\bfb G)$.
Recall that for $G' \neq G$, the normal derivative $d\rho'_{G'} = \nu^\ast \rho'_{G'}$ on $\Np G$ is the pullback of the restriction of $\rho'_{G'}$ to $G$; hence the ratio $d\rho'_{G'}/d\rho_{G'}= \nu^\ast (\rho'_{G'})/\nu^\ast(\rho_{G'})$ is pulled back from $\bfb G$ by fc-equivalence.
Likewise, the $(0,\infty)$ invariant ratio $d\rho'_G/d\rho_G$ is pulled back from $\bfb G$, so it follows that $\rho'_{\geq G} = \bfib{G}^\ast a \rho_{\geq G}$ for some $a \in C^\infty(\bfb G; (0,\infty))$, and hence $\tbfib G' = \alpha \circ \tbfib G$ for the automorphism $\alpha : C_X(\bfb G) \to C_X(\bfb G)$ defined by $\alpha: (b,t) \mapsto (b,a(b) t)$.
In particular $\tbfib G$ is well-defined up to automorphism of the target.

For part \ref{I:main_lemma_normal_funct}, we may use representative boundary defining functions with respect to which $f$ is rigid, and then $df$ pulls back $\rho_{\geq H}$ to a function of the form
\[
	(df)^\ast \rho_{\geq H} = \rho_{\geq G_H} = \sigma \rho_{\geq G},
\]
where $G_H = \min \set{G' : f_\sharp(G') = H} \leq G$
and $\sigma = \prod_{G_H \leq G' < G} \rho_{G'}$ is a product of boundary defining functions for those $G'$ such that $G_H \leq G' < G$.
As a product of good boundary defining functions for $G' < G$, it follows that $\sigma = \bfib G^\ast \sigma' \in \bfib G^\ast C^\infty(\bfb G)$
is basic.
Then the map
\[
	\wt f_G : C_X(\bfb G) \to C_Y(\bfb H), \qquad (b, t) \mapsto \big(f_G(b), \sigma'(b) t\big)
\]
has the required property in view of the following commutative diagram:
\[
\begin{tikzcd}
	v \ar[r,cm bar-to,"df"] \ar[dd,cm bar-to,swap, "\tbfib G"] & df(v) \ar[d,cm bar-to, "\tbfib H"]
	\\ & \big(\bfib H(df(v)),\, \rho_{\geq H}(df(v))\big) \ar[d,equal]
	\\\big(\bfib G(v),\, \rho_{\geq G}(v)\big) \ar[r,cm bar-to, "\wt f_G"] &  \big(f_G(\bfib G(v)),\, \sigma'(\bfib G(v))\rho_{\geq G}(v)\big)
\end{tikzcd}
\]
\end{proof}

\begin{thm}
If $X$ and $Y$ are manifolds with fibered corners, 
then $X\ttimes Y$ has a canonical fibered corners structure 
with fibrations of the form
\begin{gather}
\begin{tikzcd}[ampersand replacement=\&]
	\bff G\ttimes \bff Y \ar[r,no head] \& \tlift G Y \ar[d, "\phi_{G,Y}"] \\ \& \bfb G \ttimes \bfb Y
\end{tikzcd}
\quad
\begin{tikzcd}[ampersand replacement=\&]
	\bff X\ttimes \bff H \ar[r,no head] \& \tlift X H \ar[d, "\phi_{X,H}"] \\ \& \bfb X \ttimes \bfb H
\end{tikzcd}
\\
\begin{tikzcd}[ampersand replacement=\&]
	\bff G\jtimes_{X,Y} \bff H \ar[r,no head] \& \tlift G H \ar[d, "\phi_{G,H}"] \\ \& \bfb X \ttimes \bfb H
\end{tikzcd}
\quad \text{for $(G,H) < (X,Y)$}
\\
\begin{tikzcd}[ampersand replacement=\&]
	\bff G\ttimes \bff H \ar[r,no head] \& \tlift G H \ar[d, "\phi_{G,H}"] \\ \& \bfb X \jtimes_{X,Y} \bfb H
\end{tikzcd}
\quad \text{for $(G,H) > (X,Y)$}
\end{gather}
and fc-equivalence classes of good boundary defining functions given locally by rational combinations of boundary defining functions from $X$ and $Y$.
The ordered product satisfies the universal property of a product
in the category of manifolds with fibered corners, 
meaning that if $f: W \to X$ and $g: W \to Y$ are fibered corners morphisms, then the associated morphism $W \to X\ttimes Y$ of ordered corners is in fact a morphism of fibered corners.
\label{T:fibcorn_product}
\end{thm}
\begin{proof}
The fibrations on $\tlift G H$ for $(G,H) \leq (X,Y)$, and the fibered property of $W \to X\ttimes Y$ with respect to these fibrations has been discussed above.
The fibrations for $(G,H) > (X,Y)$ are defined as follows.
First, we use the canonical diffeomorphism $\tlift{G}{H} \cong \tlift{G_0}{H_0} \subset \Np G \ttimes \Np H$ as in the proof of Theorem~\ref{T:ord_bhs_str}
to replace $X \ttimes Y$ by $\Np G \ttimes \Np H$.
Then it follows from Lemma~\ref{L:main_lemma_normal} and Theorem~\ref{T:ord_prod} that there is a natural morphism (of ordered corners) 
$\Np G \ttimes \Np H \to C_X(\bfb G) \ttimes C_Y(\bfb H)$ restricting to a fiber bundle $\tlift{G}{H} \to \bfb G \jtimes_{X,Y} \bfb H$ with fiber $\bff G \ttimes \bff H$, and this defines
$\bfib{G,H}$, with respect to which the lifted projections $X \ttimes Y \to X$ and $X \ttimes Y \to Y$ are clearly fibered.

If $f : W \to X$ and $g : W \to Y$ are fibered corners morphisms with the property that $f_\sharp(E) = G \in \M 1(X)$ and $g_\sharp(E) = H \in \M 1(Y)$ for $E \in \M 1(W)$
with $E > W$ (and hence $G > X$ and $H > Y$), then it follows again from Lemma~\ref{L:main_lemma_normal} and Theorem~\ref{T:ord_prod} that there is a commutative diagram of
ordered corners morphisms
\[
\begin{tikzcd}[sep=small]
	\Np E \ar[r] \ar[d] & \Np G\ttimes \Np H \ar[d]
	\\ C_W(\bfb E) \ar[r] & C_X(\bfb G)\ttimes C_Y(\bfb H)
\end{tikzcd}
\quad \text{restricting to}
\quad
\begin{tikzcd}[sep=small]
	 E\ar[d] \ar[r] &\tlift {G} {H} \ar[d]
	\\ \bfb E \ar[r] & \bfb G\jtimes_{X,Y} \bfb H
\end{tikzcd}
\]
and so the map $W \to X\ttimes Y$ is fibered.

The equivalence classes of boundary defining functions on $X \ttimes Y$ are defined locally by rational combinations of boundary defining functions from $X$ and $Y$.
More precisely,
if $\set{\rho_{G}}$ and $\set{\rho_{H}}$ are representative boundary defining functions 
on $X$ and $Y$, respectively,
then on a neighborhood in $X\ttimes Y$ which meets a totally ordered set of boundary hypersurfaces $\tlift{G_i}{H_j}$, the linear system
\[
	\bd \sigma(g_i + h_j) = \begin{cases}  1, & \text{if $(g_i,h_j) = (g,h)$,} \\ 0, & \text{otherwise} \end{cases}
\]
fixes the exponents of $\sigma = \prod_{i} \rho_{G_i}^{a_i} \prod_{j} \rho_{H_j}^{b_j}$ uniquely so that it constitutes a local boundary defining function for $\tlift{G}{H}$
on such a neighborhood in view of Proposition~\ref{P:lift_proj_comb}.
Indeed, the coefficients $\set{a_i,b_j}$ are precisely those for the basis of $\bN^\ast F$, $F = \bigcap_{i,j} \tlift{G_i}{H_j}$ which is dual to the basis $\set{g_i + h_j}$ of $\bN F$;
since the latter is unimodular, it follows that $a_i, b_j \in \set{-1,0,1}$ (and in fact either all the $a_i$ are non-negative while the $b_j$ are non-positive, or vice versa).
As fc-equivalence is a local property, the equivalence class of boundary defining functions on $X\ttimes Y$ which are locally fc-equivalent to rational combinations
as above is well-defined.

The fact that these rational boundary defining functions are good is equivalent to the rigidity of $\tlift{G}{H} \to \bfb G\jtimes_{X,Y} \bfb H$ for $(G,H) > (X,Y)$
(respectively $\tlift{G}{H} \to \bfb G \ttimes \bfb H$ for $(G,H) < (X,Y)$).
To see this, first note that $\Np G \to C_X(\bfb G)$ (resp.\ $\Np G \to \bfb G$) is rigid by definition and the fact that boundary defining functions on $X$ are good, and hence $\Np G \ttimes \Np H \to C_X(\bfb G)
\times C_Y(\bfb H)$ (resp.\ $\Np G \ttimes \Np H \to \bfb G \times \bfb H$) is rigid with respect to rational boundary defining functions on the domain.
It then follows from the local coordinate structure of the unique lift $\Np G \ttimes \Np H \to C_X(\bfb G)\ttimes C_Y(\bfb H)$ (resp.\ $\Np G \ttimes \Np H \to \bfb G \ttimes \bfb H$)
that this lift is rigid with respect to rational boundary defining functions on the domain and codomain, and restricting this to $\tlift{G}{H}$ gives the desired result.

The projections $X \ttimes Y \to X$ and $X \ttimes Y \to Y$ are compatible with the fc-equivalence classes on $X\ttimes Y$ essentially by definition, and 
the fact that $W \to X \ttimes Y$ is compatible follows from the general fact that
the lift to a generalized blow-up as in Theorem~\ref{T:lifting_b-maps}.(b) of a locally rigid map is rigid with respect to rational local boundary defining functions.
\end{proof}

We next address the fibered corners structure of the fiber products considered in \S\ref{S:ord_fib_prod}.
\begin{prop}
Suppose $f : X \to Z$ and $g : Y \to Z$ are morphisms of fibered corners manifolds which are b-transverse.
Then $X\ttimes_Z Y$ obtains a canonical fibered corners structure, with fibrations on the principal faces $\tlift{G}{H}_Z := \tlift{G}{H} \cap X \ttimes_Z Y$ as follows:
\begin{itemize}
\item For $(G,H) < (X,Y)$ with $G \in \M 1(X)$, $H \in \M 1(Y)$ and $f_\sharp(G) = g_\sharp(H) = Z$, the fibration has the form
\[
\begin{tikzcd}[sep=small]
	\bff G \jtimes_{X,Y,Z} \bff H \ar[r,-] & \tlift{G}{H}_Z \ar[d]
	\\ & \bfb G \ttimes_{\bfb Z} \bfb H
\end{tikzcd}
\]
where $\bff G \jtimes_{X,Y,Z} \bff H$ is the boundary hypersurface of $C_X(\bff G) \ttimes_{\bff Z} C_Y(\bff H)$ of the form \eqref{E:fib_prod_bad_face}.
\item For $(G,H) > (X,Y)$ with $G \in \M 1(X)$, $H \in \M 1(Y)$ and $f_\sharp(G) = g_\sharp(H) = Z$, the fibration has the form
\[
\begin{tikzcd}[sep=small]
	\bff G \ttimes_{\bff Z} \bff H \ar[r,-] & \tlift{G}{H}_Z \ar[d]
	\\ & \bfb G \jtimes_{X,Y,Z} \bfb H
\end{tikzcd}
\]
where $\bfb G \jtimes_{X,Y,Z} \bfb H$ is the boundary hypersurface of $C_X(\bfb G) \ttimes_{\bfb Z} C_Y(\bfb H)$ of the form \eqref{E:fib_prod_bad_face}.
\item In all other cases, the fibration has the form
\[
\begin{tikzcd}[sep=small]
	\bff G\ttimes_{\bff K} \bff H \ar[r,-] & \tlift{G}{H}_Z = G \ttimes_K H \ar[d]
	\\ & \bfb G\ttimes_{\bfb K} \bfb H
\end{tikzcd}
	\qquad \text{where $K = f_\sharp(G) = g_\sharp(H) \in \M 1(Z).$}
\]
\end{itemize}
Moreover $X\ttimes_Z Y$ satisfies the universal property of the fiber product in the category of manifolds with fibered corners.
\label{P:fib_corn_fib_prod}
\end{prop}
\begin{proof}
Suppose first that $G < X$ and $H < Y$ with $f_\sharp(G) = g_\sharp(H) = Z \in \M 0(Z)$, and replace $X$ and $Y$ by $\Np G$ and $\Np H$, respectively.
From the commutative diagram
\[
\begin{tikzcd}[sep=small]
	\Np G \ar[d] \ar[r]& Z\ar[d]& \Np H  \ar[d] \ar[l]
	\\\bfb G \ar[r] & \bfb Z &  \bfb H \ar[l]
\end{tikzcd}
\]
it follows that $\Np G \ttimes_Z \Np H$ admits a natural map to $\bfb G \ttimes_{\bfb Z} \bfb H$, and restricting over a point in the latter
shows that the fiber may be identified with $C_X(\bff G) \ttimes_{\bff Z} C_Y(\bff H)$, as it satisfies the relevant universal property. 
Restriction to the face $\tlift{G}{H}_Z \subset \Np G\ttimes_Z \Np H$ shows that this fibers over $\bfb G \ttimes_{\bfb Z} \bfb H$ with fiber $\bff G \jtimes_{X,Y,Z} \bff H$,
using the characterization of the associated principal face of $C_X(\bff G) \ttimes_{\bff Z} C_Y(\bff H)$ from 
Theorem~\ref{T:bhs_of_fib_prod}.

In a similar manner, if $G > X$ and $H > Y$ are such that $f_\sharp(G) = g_\sharp(H) = Z$, then the commutative diagram
\[
\begin{tikzcd}[sep=small]
	\Np G \ar[d] \ar[r]& Z\ar[d]& \Np H  \ar[d] \ar[l]
	\\C_X(\bfb G) \ar[r] & \bfb Z &  C_Y(\bfb H) \ar[l]
\end{tikzcd}
\]
obtained from Lemma~\ref{L:main_lemma_normal} shows that $\Np G \ttimes_Z \Np H$ admits a natural map to $C_X(\bfb G) \ttimes_{\bfb Z} C_Y(\bfb H)$, restricting over a point
of which shows that the fiber may be identified with $\bff G \ttimes_{\bff Z} \bff H$. 
The restriction to $\tlift{G}{H}_Z$ maps over the face $\bfb G \jtimes_{X,Y,Z} \bfb H \subset C_X(\bfb G) \ttimes_{\bfb Z} C_Y(\bfb H)$, giving the result in this case.

In all other cases, $\tlift{G}{H}_Z \cong G \ttimes_{K} H$, where $K = f_\sharp(G) = g_\sharp(H) \in \M 1(Z)$, fits into the diagram
\[
\begin{tikzcd}[sep=small]
	& G\ttimes_K H \ar[dr] \ar[dl] \ar[d]&
	\\ G \ar[r] \ar[d]& K \ar[d]& H\ar[d] \ar[l]
	\\ \bfb G \ar[r] & \bfb K & \bfb H \ar[l]
\end{tikzcd}
\]	
and consequently maps naturally to $\bfb G \ttimes_{\bfb K} \bfb H$, and restriction over a point in the latter shows that the fiber satisfies the universal property
of $\bff G \ttimes_{\bff K} \bff H$, proving the result in these cases.

Note that in all cases, the fibration just described coincides with the restriction to $\tlift{G}{H}_Z$ of the fibration $\bfib{G,H}$ on $\tlift{G}{H}$.
This is clear in all cases, except perhaps the last case when $(G,H) > (X,Y)$, but in that case the diagram of relevant maps
\[
\begin{tikzcd}[sep=small]
	\Np G \ar[d] \ar[r]& \Np K\ar[d]& \Np H  \ar[d] \ar[l]
	\\C_X(\bfb G) \ar[r] & C_Z(\bfb K) &  C_Y(\bfb H) \ar[l]
\end{tikzcd}
\]
leads to the natural map $\Np G \ttimes_{\Np K} \Np H \to C_X(\bfb G) \ttimes_{C_Z(\bfb K)} C_Y(\bfb H)$,
and arguing as in the proof of Theorem~\ref{T:bhs_of_fib_prod} shows that the relevant boundary face of $C_X(\bfb G) \ttimes_{C_Z(\bfb K)} C_Y(\bfb H)$ satisfies the universal
property of $\bfb G \ttimes_{\bfb K} \bfb H$, to which the restriction to $\tlift{G}{H}_Z$ maps.

The fc-equivalence classes of boundary defining functions of $X\ttimes_Z Y$ are inherited by restriction from the classes on $X\ttimes Y$, and the fact that these are good with respect to the 
fibrations described amounts to rigidity of all maps just described with respect to the fc-equivalence classes on $\Np G$, $\Np H$, and so on.
\end{proof}

Returning to the application discussed in Corollary~\ref{C:s_equals_1} in the fibered corners setting yields the following.
\begin{cor}
If $X$ and $Y$ are manifolds with fibered corners with representative total boundary defining functions $\rho_X = \prod_{G \in \M 1(X)} \rho_G : X \to [0,\infty)_\tmin$ and $\rho_Y = \prod_{H \in \M 1(Y)} \rho_H : Y \to [0,\infty)_\tmin$, then the fiber product
\[
	X\ttimes_{[0,\infty)_\tmin} Y = \set{\rho_X/\rho_Y = 1} \subset X\ttimes Y
\]
is a p-submanifold in a sufficiently small neighborhood of the boundary (where $d\rho_X \neq 0$ and $d \rho_Y \neq 0$), and has fibrations of the form
\[
\begin{tikzcd}[sep=small]
	\bff G\ttimes \bff H \ar[r,-] &\tlift{G}{H} \cap \set{\rho_X/\rho_Y = 1} \ar[d, "\bfib{G,H}"] 
	\\ & \bfb G \ttimes \bfb H
\end{tikzcd}
\]
with fiber $\bff G \ttimes \bff H = (\bff G)_\tmin \ttimes (\bff H)_\tmin$ and base $\bfb G \ttimes \bfb H = (\bfb G)_\tmax \ttimes (\bfb H)_\tmax$.
\label{C:s_equals_1_fibered}
\end{cor}

\section{Relation to other theories} \label{S:appl}

In this section we discuss how products of fibered corners manifolds relate and/or generalize products in other categories.
\subsection{Smoothly stratified spaces} \label{S:strat}

We briefly recall the definition of smoothly stratified spaces, referring to \cite{Albin,Pflaum} for a more complete account.
A \emph{stratified space} is a topological space $\strat X$ decomposed as a disjoint union $\strat X = \bigsqcup_i \sstrat i$
of locally closed subspaces $\sstrat i$ called \emph{strata}, each of which is itself a manifold of some dimension, satisfying the \emph{frontier condition} that $\sstrat i \cap \ol{\sstrat j} \neq \emptyset$ if and only if $\sstrat i \subset \ol{\sstrat j}$, which defines a partial order $\sstrat i < \sstrat j$.
In particular, there is a maximal \emph{principal stratum} $\sstrat 0 = \strat X^\circ$ which is dense, with the rest referred to as the \emph{singular strata}.
The closed strata $\ol {\sstrat i}$ are again stratified spaces with strata $\set{\sstrat j : \sstrat j < \sstrat i}$.
A \emph{smoothly stratified space} (aka `Thom-Mather stratified space') is a stratified space $\strat X$ along with \emph{control data} consisting of open neighborhoods $\nO i \supset \sstrat i$ in $\strat X$ called \emph{tubes}, equipped with retractions $r_i : \nO i \to \sstrat i$ and distance functions $\rho_i : \nO i \to \bbR_+$ such that $\sstrat i = \rho_i^\inv(0)$ and satisfying
the conditions that whenever $\sstrat i < \sstrat j$,
\[
	\rho_i \circ r_{j} = \rho_i, \quad r_i \circ r_{j} = r_i, 
	\qquad \text{and}
	\qquad (r_i, \rho_i) : \nO i \cap \sstrat j \to \sstrat i \times (0,\infty) \quad \text{is a submersion.}
\]
As a consequence of Thom's isotopy lemma 
for smoothly
stratified spaces (see \cite[3.9.2]{Pflaum}),
the $\nO i$ are locally trivial bundles of cones over
$\sstrat i$, the links of which are stratified spaces of lower depth.

Associated to each interior maximal manifold with fibered corners $X$ is 
a smoothly stratified space $\strat X$ obtained by iteratively collapsing the fibers over each boundary hypersurface in reverse
order (starting with the largest base and ending with the smallest).
The strata of $\strat X$ are the smooth manifolds $\sstrat G = \bfb G^\circ$, for $G \in \Mtot(X)$, with principal smooth stratum $X^\circ$, and the control data
are given by the passage to the quotient of tubular neighborhoods of the boundary hypersurfaces, their associated retractions, and the images of (good) boundary defining functions.
The closure $\ol{\sstrat G} = \strat{B}_G$ is the stratified space obtained by fiber collapse of the fibered corners manifold $\bfb G$ itself, and the partial
order 
$\sstrat G < \sstrat {G'}$ 
on strata
defined by
$\sstrat {G} \subset \ol{\sstrat {G'}}$ coincides with the order $G < G'$ induced by the fibered corners structure.
The link of the cone bundle $\nO G \to \sstrat G$ is the stratified space $\strat {F}_G$ obtained by collapsing the fibers of the boundary hypersurfaces of $\bff G$.

Conversely, as shown in \cite{ALMP,Albin} (a result originating in \cite{Verona}; see also \cite[3.9.4]{Pflaum}), every smoothly stratified space can be resolved to a manifold with fibered corners by iteratively replacing the cone bundle $\nO G$
of each singular stratum with an associated cylinder bundle in order from smallest stratum to largest, with the tubes $\nO G$ resolving to tubular neighborhoods of the boundary hypersurfaces, and distance functions resolving to boundary defining functions.

\begin{rmk}
While this constitutes an equivalence of objects between smoothly stratified spaces on one hand, and (interior maximal) manifolds with fibered corners on the other hand,
correspondingly little has been written about morphisms.
In \cite{ALMP}, Albin et.\ al.\ show that the resolution procedure above is functorial with respect to \emph{isomorphisms}, meaning controlled isomorphisms of stratified spaces
on one hand, and fibered diffeomorphisms of fibered corners manifolds on the other,
but in general, pinning down a suitable class of morphisms on the stratified spaces side is difficult. 
In addition to the requisite condition that a map $f : \strat X \to \strat Y$ be \emph{stratified}, meaning that it sends strata to strata and restricts to a smooth map of manifolds thereon (which corresponds to the fibered condition of Definition~\ref{D:fibered_map}), the typical conditions discussed in the literature are inadequate for our purposes.
Indeed, the weakest typical condition is that the map $f$ is `weakly controlled', meaning it maps tubes to tubes and commutes with the retractions.
This is on the one hand too strong, since it amounts to extending the boundary fibrations to open neighborhoods of the boundary hypersurfaces and demanding an extension of the fiber bundle maps 
\eqref{E:fibered}
to such neighborhoods, and it does not seem that every fibered corners morphism satisfies such a condition.
On the other hand, it is also too weak, since it does not impose the kind of algebraic behavior with respect to boundary defining functions that a general b-map satisfies.
From our point of view, the right class of morphisms are those maps of smoothly stratified spaces which are descended from fibered corners morphisms as defined here, but it
seems difficult to characterize these directly using only the data of stratified spaces.
\end{rmk}

There is little discussion of products of stratified spaces in the literature, but the definition of a smoothly stratified structure on the cartesian product $\strat X \times \strat Y$ is obvious enough.
It has strata $\sstrat G \times \sstrat H$ for $(G,H) \in \Mtot(X)\times \Mtot(Y)$, with control data $(\nO G \times \nO H, r_G\times r_H, \rho_G + \rho_H)$
(it seems any homogeneous combination of $\rho_G$ and $\rho_H$ would also suffice, such as $(\rho_G^p + \rho_H^p)^{1/p}$ for $p \geq 1$).
This satisfies the universal property of the product for most reasonable notions of morphism of stratified spaces.

If we knew that the projections $\strat X \times \strat Y \to \strat X$ and $\strat X \times \strat Y \to \strat Y$ lifted to fibered corners morphisms upon resolution, as well as the map into $\strat X \times \strat Y$ induced by a good map $\strat W \to \strat X$ and $\strat W \to \strat Y$, then an easy application of the universal property
of the product in both categories would give a quick proof of the following result. 
However, since this is unavailable at present, we settle for a coordinate based proof.

\begin{thm}
Let $X$ and $Y$ be interior maximal manifolds with fibered corners, with associated smoothly stratified spaces $\strat X$ and $\strat Y$.
Then the ordered product $X\ttimes Y$ is equivalent to the resolution of the product $\strat X \times \strat Y$ of stratified spaces.
\label{T:prod_resolves_prod}
\end{thm}
\begin{proof}
Let 
\begin{equation}
	(y_1,x_1,y_2,x_2,\ldots,y_n,x_n,y_{n+1})
	\label{E:str_coord_X}
\end{equation}
denote standard form coordinates on a (maximal) fibered corners manifold $X$.
Iteratively collapsing the boundary fibers leads to singular coordinates
\[
	(y_1,r_1,r_1y_2,r_2,r_2y_3,\ldots, r_n, r_n y_{n+1}), \qquad r_i = r_{i-1}x_i = x_1\cdots x_i
\]
on the associated stratified space $\strat X$; conversely, the resolution procedure is implemented in coordinates by 
iteratively dividing the coordinates to the right of each $r_i$ by $r_i$ starting with $r_1$. 

Likewise, if $Y$ is another fibered corners manifold with standard form coordinates
\begin{equation}
	(y'_1,x'_1,y'_2,x'_2,\ldots,y'_n,x'_{m},y'_{m+1})
	\label{E:str_coord_Y}
\end{equation}
with associated singular coordinates 
\[
	(y'_1,r'_1,r'_1y'_2,r'_2,r'_2y'_3,\ldots, r'_m, r'_m y'_{m+1})
\]
on $\strat Y$, then singular coordinates on $\strat X \times \strat Y$ are given by
\[
	(y_1,y'_1,r_1,r_1y_2,r'_1, r'_1y'_2, \ldots, r_n, r_n y_{n+1})
\]
assuming without loss of generality that $m \leq n$.
We will focus attention on the totally ordered chain of strata $(1,1)<(1,2)<(2,3)<\cdots<(m-1,m)<(m,0)<\cdots<(n,0)$; others are similar.

The resolution procedure consists of first replacing $r_1$ and $r'_1$ by $r_1 + r'_1$ and $s_1 = r_1/(r_1 + r'_1)$ (or $s'_1 = 1- s_1 = r'_1/(r_1 + r'_1)$), dividing all coordinates to the right by $r_1 + r'_1$.
Then, near the lift of (say) $s'_1 = 0$ and $r_2 = 0$, the next step is to replace $s'_1$ and $s_2 = r_2/(r_1 + r'_1)$ by $s'_1 + s_2$ and $s_2/(s'_1 + s_2) = r_2/(r'_1 + r_2)$, divide all coordinates to the right by $s'_1 + s_2$, and so on.
However, as is frequently the case with blow-up computations, it is more convenient here to use equivalent rational coordinates, in which 
$(r_1,r'_1)$ is replaced instead by $(r_1,r'_1/r_1)$ and all coordinates to the right  are divided by $r_1$; then $(r'_1/r_1, r_2/r_1)$ is replaced by $(r'_1/r_1, r_2/r'_1)$
and all coordinates to the right are divided by $r'_1/r_1$, and so on.
The end of this process results in standard form coordinates
\[
	\big(y_1,y'_1, x_1, y_2, \tfrac{x'_1}{x_1}, y'_2, \tfrac{x_1x_2}{x'_1},y_3, \tfrac{x'_1x'_2}{x_1x_2}, y'_3, \ldots, y'_{m+1}, \tfrac{x_1\cdots x_{m+1}}{x'_1\cdots x'_m}, y_{m+1}, \ldots, x_n, y_{n+1}\big)
\]
for the resolution of $\strat X \times \strat Y$ which are identical to the rational coordinates near $\tlift{H_1}{K_1} \cap \tlift{H_1}{K_2} \cdots \cap \tlift{Y}{H_n}$
for the ordered product $X\ttimes Y$ based on \eqref{E:str_coord_X} and \eqref{E:str_coord_Y}.
Other coordinate charts based on other totally ordered chains of strata involve similar computations.
\end{proof}

If $X$ is a manifold with fibered corners, then the natural fibered corners structure on the cone $C_\tmax(X)$ 
according to Definition~\ref{D:CX_bfib_fiber} consists of fibrations
\[
	X\times 0 \to \pt, \qquad G\times [0,\infty) \to \bfb G \times [0,\infty), \qquad X\times [0,\infty) \to X \times [0,\infty).
\]
It is clear that the stratified space associated to $C_\tmax(X)$ is the cone
\[
	\strat C (\strat X) = \big(\strat X \times[0,\infty)\big)/(\strat X \times \set 0)
\]
with link $\strat X$.
From the well-known fact that the product of cones is a cone over the join of the links, we obtain the following.

\begin{cor}
For fibered corners manifolds $X$ and $Y$, the stratified space associated to $C_\tmax(X)\ttimes C_\tmax(Y)$ is the cone
\[
	\strat C(\strat X) \times \strat C(\strat Y) \cong \strat C(\strat X \star \strat Y)
\]
and in particular, the stratified space associated to the maximal join $X\jtimes_\tmax Y$ is the topological join
\[
	\strat X \star \strat Y = (\strat X \times \strat Y\times [0,1])/\sim,
\]
where $(x,y,0) \sim (x',y,0)$ for all $x,x' \in \strat X$ and $y \in \strat Y$ and $(x,y,1) \sim (x,y',1)$ for all $x \in \strat X$, $y,y' \in \strat Y$.
\label{C:join}
\end{cor}

\subsection{Many body spaces} \label{S:mb}

Many body compactifications of vector spaces go back at least to \cite{Vasy}, and have been discussed more recently in \cite{Kmb,AMN}.
Using notation from \cite{Kmb}, recall that given a finite dimensional vector space $V$ and a \emph{linear system} $\cS_V$, meaning a finite set of subspaces of $V$ which is closed under intersection and contains $\set 0$ and $V$, 
the \emph{many body compactification} $\mb V$ of $V$ is the manifold with corners
\[
	\mb V = [\ol{V}; \set{\pa \ol S : S \in \cS_V}],
\]
given by the blow-up in the radial compactification $\ol V$ of the boundaries of the subspaces in $\cS_V$, taken in order of size.
The boundary hypersurfaces of $\mb V$ are indexed by the subspaces in $\cS_V$, ordered by containment; in fact $\Mtot(\mb V) \cong \cS_V$ as ordered sets, with the interior indexed by $\set 0$.
In addition, $\mb V$ has a natural interior minimal fibered corners structure, with fibrations consisting of products
\begin{equation}
	\phi_S : H_{S} = \bfb S\times \bff S \to \bfb S,
	\quad\text{where} \quad \bff S = \mb{V/S},
	\quad \bfb S = [\pa \ol S; \set{\pa\ol {S'} : S' < S \in \cS_V}]
	\label{E:mbfaces}
\end{equation}
for each $S \in \cS_V$, where $V/S$ is equipped with the linear system $\cS_{V/S} = \set{S'/(S/\cap S) : S' \in \cS_V}$, and $\bfb S = H_{S} \subset \mb S$ is the maximal boundary face
of the many body compactification of $S$ itself, with respect to the system $\cS_S = \set{S' \in \cS_V : S' \subset S}$. 
The boundary hypersurfaces of $\bfb S$ decompose as
\begin{equation}
	\pa_{S'} \bfb S = \bfb {S'} \times \bfb {S/S'}, \quad S' < S.
	\label{E:mbbasefaces}
\end{equation}
The fc-equivalence class of boundary defining functions on $\mb V$ is well-defined by lifting rational combinations of radial functions on $\ol V$ which define $\pa \ol S$ for $S \in \cS_V$.

As shown in \cite{Kmb}, the many body compactification is functorial with respect to \emph{admissible} linear maps, meaning $f : V \to W$ satisfies $f^\inv(\cS_W) \subset \cS_V$,
with such $f$ extending to b-maps $\mbf f : \mb V \to \mb W$. 
If in addition $f(\cS_V) = \cS_W$ --- what is called in \cite{Kmb} an \emph{admissible quotient} --- then $\mbf f$ is a b-fibration and satisfies the 
conditions to be a morphism in the category of interior minimal fibered corners. 
Indeed, the restriction of the lifted map splits as a product with respect to \eqref{E:mbfaces} and \eqref{E:mbbasefaces}.

It is clear that the stratified space associated to $\mb V$ by collapsing the fibers of all boundary fibrations (we do not collapse the fibers over the interior) is the radial compactification $\ol V$, with principal stratum $V$, and (closed) singular strata consisting of the boundaries $\pa \ol S$ of the subspaces in $\cS_V$, the links of which are the normal quotients $\ol{V/S}$ with strata $\pa\ol{S'/S}$ for $S'>S$.
In this category, a natural alternative to the usual product of $\ol V$ and $\ol W$ as stratified spaces (which does not resolve to a many body compactification of a vector space) is the space $\ol {V\times W}$ with closed singular strata $\pa(\ol{S\times T})$ for $(S,T) \in \cS_V\times \cS_W$, the resolution of which is the many body compactification $\mb{V\times W}$.

Note that, as spheres, the closed singular strata $\pa(\ol{S\times T})$ are homeomorphic to the topological joins $\pa\ol S \star \pa \ol T$ of the singular strata of the factors,
a property which generalizes to the interior minimal fibered corners setting since $\bfb G \jtimes_{\tmin} \bfb H$ has associated stratified space the join $\strat B_G \star \strat B_H$ of the stratified spaces associated to $\bfb G$ and $\bfb H$ as a consequence of the results of \S\ref{S:equiv}.

\begin{thm}
Let $\mb V$ and $\mb W$ be many body compactifications of $(V, \cS_V)$ and $(W, \cS_W)$, respectively.
Then the ordered product $\mb V\ttimes \mb W$ as interior minimal manifolds with fibered corners is canonically isomorphic to the many body compactification $\mb{V\times W}$ of $(V\times W, \cS_V \times \cS_W)$.
\label{T:mb_prod}
\end{thm}
\begin{proof}
The lifted projections $\mb{V\times W} \to \mb V$ and $\mb{V\times W} \to \mb W$ are interior minimal morphisms, so by the universal property
there is an associated map $h : \mb{V\times W} \to \mb V \ttimes \mb W$,
with induced map $h_\sharp : \Mtot(\mb{V\times W}) \to \Mtot(\mb V \ttimes \mb W)$ 
an isomorphism of ordered sets identified with $\cS_{V\times W} \cong \cS_V \times \cS_W$.
Both $\mb{V\times W}$ and $\mb V \ttimes \mb W$ have interior identified with $V \times W$, on which $h$ restricts to the identity,
so it remains to show that $h$ extends to an isomorphism over each boundary hypersurface.
We proceed by induction on the total depth of $(V,W)$, meaning the sum $\abs{\cS_V} + \abs{\cS_W}$, which is equivalent to the depth of $\mb{V \times W}$ or $\mb V \ttimes \mb W$ as fibered corners manifolds,
with depth zero coinciding with the trivial case $V = W = \set 0$.

The restriction to a hypersurface $h : H_{S\times T} \subset \mb{V\times W} \to \tlift{H_S}{H_T} \subset \mb{V}\ttimes \mb W$ is a fibered map
\begin{equation}
\begin{tikzcd}
	\bff{S\times T} \times \bfb {S\times T} \ar[d] \ar[r, "h"] &(\bff {S}\ttimes \bff T) \times (\bfb S \jtimes \bfb T) \ar[d]
	\\ \bfb{S\times T} \ar[r, "h_{S\times T}"] & \bfb S \jtimes \bfb T
\end{tikzcd}
	\label{E:mb_hs_map}
\end{equation}
where $\jtimes = \jtimes_\tmin$,
with $\bfb S \jtimes \bfb T$ replaced simply by $\bfb S$ or $\bfb T$ in case $T = 0$ or $S = 0$, respectively.
Restricting over a point in the base and recalling that
$\bff S = \mb{V/S}$,
$\bff T = \mb{W/T}$,
and
$\bff {S\times T} = \mb {(V\times W)/(S \times T)} = \mb{V/S\times W/T}$, 
it is clear that the map on fibers is the universal map $\mb{V/S\times W/T} \to \mb{V/S} \ttimes \mb{W/T}$, in particular independent of the point in the base, and hence the top row of \eqref{E:mb_hs_map}
splits as a product.
Since $(S,T) \neq (0,0)$ here (as we are considering a proper boundary hypersurface), the pair $(V/S,W/T)$ has strictly lower total depth, and so the first factor
$\bff{S\times T} \to \bff{S}\ttimes \bff T$ is an isomorphism by the inductive hypothesis.

Next we recall that the lower map $h_{S\times T}$ is induced by a thickened map $C_\tmin(\bfb {S \times T}) \to C_\tmin(\bfb S)\ttimes C_\tmin(\bfb T)$ afforded by Lemma~\ref{L:main_lemma_normal}; however
in this many body setting this has a natural reinterpretation.
Indeed,
here we may generally identify the cone $C_\tmin(\bfb S)$ with the normal model $\Np \bfb S$ for the maximal hypersurface $\bfb S \subset \mb S$ in the many body compactification for $S$ itself,
and the thickened horizontal map on cones in part (b) of Lemma~\ref{L:main_lemma_normal} then coincides with the normal derivative of the induced map between the many body compactifications of
the relevant subspaces.
In particular the map $\bfb {S\times T} \to \bfb S \jtimes \bfb T$ coincides with the restriction to the maximal boundary hypersurface of the map $\mb {S\times T} \to \mb S \ttimes \mb T$,
and again by induction it follows that the latter is an isomorphism, \emph{except} in the case that $(S,T) = (V,W)$ itself.

To take care of this final case, we verify by hand that the map $\mb {V \times W} \to \mb V \ttimes \mb W$ has full rank over the interior of the maximal boundary hypersurface $H_{V\times W} = \bfb {V\times W}$.
Representing points in $V \times W$ by `product radial' coordinates
\[
	(v,w) = (Rr \nu, Rs \omega), 
	\qquad \text{where} \quad R = \sqrt{\abs v^2 + \abs w^2}, \quad r = \frac{\abs v}{R}, \quad s = \frac{\abs{w}}{R} = \sqrt{1 - r^2}
\]
where $\nu$ and $\omega$ are respective unit vectors for some norms on $V$ and
$W$, leads to coordinates $(x,r, \nu, \omega)$ near the boundary
of the radial compactification $\ol {V\times W}$, with $x = R^\inv$.
Assuming now that
$r \in (0,1)$ and $\nu$ and $\omega$ are limited to open sets disjoint from any of the proper subspaces in $\cS_V$ or $\cS_W$,
these remain valid coordinates for the interior of the maximal boundary face $\bfb{V\times W} \subset \mb{V\times W}$.
Likewise, $(x', \nu)$, where $x' = \abs{v}^\inv$ and $(x'', \omega)$, where $x'' = \abs{w}^\inv$ form coordinates for $\ol V$ and $\ol W$ lifting to $\mb V$ and $\mb W$,
and the lifted projections $\mb{V\times W} \to \mb V$ and $\mb{V\times W} \to \mb W$ are given respectively by
\[
	(x,r,\nu,\omega) \mapsto (\tfrac{x}{r}, \nu) \qquad \text{and} \qquad (x,r,\nu,\omega) \mapsto (\tfrac{x}{\sqrt{1-r^2}}, \omega).
\]
It follows that coordinates near the maximal boundary hypersurface of $\mb{V}\ttimes \mb{W}$ 
are given by $(x',\tfrac{x'}{x''}, \nu,\omega)$, in which the lifted map $\mb {V\times W} \to \mb V \ttimes \mb W$ takes the form
\[
	(x,r,\nu,\omega) \mapsto (\tfrac{x}{r}, \tfrac{r}{\sqrt{1 - r^2}}, \nu, \omega)
\]
which has full rank down to $\set {x = 0}$ for $r \in (0,1)$.
\end{proof}

\section{Geometric structures} \label{S:geom}

In this section we consider certain classes of metrics on interior minimal and maximal manifolds with fibered corners.
A convenient and equivalent way to discuss these classes of metrics is in terms of their associated \emph{geometric structures}, as encoded by associated
\emph{rescaled tangent bundles}, on which a Riemannian metric of the given type on the interior extends to a uniformly bounded inner product up to all boundary
faces.
Both structures to be considered begin with the following.
\begin{defn}
The algebra of \emph{edge vector fields} is defined as
\[
	\eV(X) = \set{V \in \bV(X) : V\rst_G(\bfib G^\ast (C^\infty(\bfb G))) = 0}
\]
or those vector fields whose restriction to every $G \in \M 1(X)$ is tangent to the fiber
of $\bfib G$.
This definition has appeared in various forms in \cite{ALMP, ALMP2, AGR} and
others sometimes under the name `iterated edge' or `ie', and extends the
definition due to \cite{Mazzeo} in the depth 1 case.
This is a Lie subalgebra of $\bV(X)$ and forms a locally free sheaf of constant
rank, defining the \emph{edge tangent bundle} by
\begin{equation}
	\eT X \to X, \quad \text{where} \quad \eV(X) = C^\infty(X; \eT X).
	\label{E:eT}
\end{equation}
In local standard form coordinates, a frame for $\eT X$ is furnished by the vector fields
\begin{equation}
	v_1\pd{y_1},\,
	v_1\pd{x_1},\,
	v_2\pd{y_2},\,
	v_2\pd{x_2},\,
	\ldots, \,
	v_n\pd{y_n},\,
	v_n\pd{x_n},\,
	\pd{z},
	\qquad v_k = x_k\cdots x_n,
\label{E:vf_edge}
\end{equation}
where $\pd{y_j}$ is shorthand for various partial derivatives $\pd{y_j^l}$ associated to the components of $y_j \in \bbR^{k_j}$ and we use $\pd{z}$ in place of $\pd{y_{n+1}}$.
Since these vector fields will be of interest in both the interior maximal and minimal cases, and since the interior fibration won't play a role, we enumerate boundary hypersurfaces exclusively with positive integers in all cases throughout this section. 
\label{D:edge_str}
\end{defn}

The vector fields of particular interest in the setting of interior maximal manifolds are the following.
\begin{defn}
A \emph{wedge vector field} (called an `iterated incomplete edge' vector field in \cite{ALMP, ALMP2}, with the name `wedge' coming from \cite{GKMwedge} in the depth 1 case) is one of the form $\tfrac 1 {\rho_X} V$, where $V \in \cV_e(X)$ and $\rho_X = \prod_{G \in \M 1(X)} \rho_G$ is a total boundary defining function.
Such a vector field is singular at the boundary faces of $X$, and the set $\wV(X)$ of wedge vector fields does not form a Lie algebra.
Instead the \emph{wedge tangent bundle} may be defined as a rescaling (see \cite{MAPSIT} for a general discussion of this procedure) of $\eT X$ by the requirement
\[
	\rho_X C^\infty(X; \wT X) = C^\infty(X; \eT X).
\]
(See \cite{AGR} for an alternate definition via the cotangent space $\wT^\ast X$.)
Note that $\wV(X)$ and $\wT X$ are independent of the choice of $\rho_X$.
A local frame for $\wT X$ is given in standard form coordinates by
\begin{equation}
	\pd{y_1},\,
	\pd{x_1},\,
	w_1\pd{y_2},\,
	w_1\pd{x_2},\,
	\ldots,\,
	w_{n-1}\pd{y_n},\,
	w_{n-1}\pd{x_n},\,
	w_{n}\pd{z},
	\qquad w_k = \frac 1{x_1\cdots x_k}.
\label{E:vf_wedge}
\end{equation}
\label{D:wedge_str}
\end{defn}

As for $\bT X$, both $\eT X$ and $\wT X$ bundles are canonically isomorphic to the usual tangent bundle over the interior of $X$, and
a \emph{wedge metric} 
on $X$ is then by definition the Riemannian metric on $X^\circ$ induced by uniformly bounded inner product on $\wT X$ which is smooth on the interior and conormal up to the boundary (i.e., all b-derivatives are also uniformly bounded; this is in accord with the usual geometric analysis definition).
%
%
Wedge metrics are incomplete on $X$ (hence the term `incomplete iterated edge metric' of \cite{ALMP, ALMP2}), and contain the iterated conic metrics
studied by Cheeger in \cite{Cheeger} as special cases.
%
%
While a general wedge metric is defined by an arbitrary inner product on $\wT X$, a model wedge metric in standard form coordinates which shows off the iterated conic 
aspect of the geometry (and an example of what is called a `rigid wedge metric' in \cite{ALMP}) is one of the form
\[
	g_X = dy_1^2 + dx_1^2 + x_1^2\big(dy_2^2 + dx_2^2 + \cdots +x_{n}^2(dy_n^2 + dx_n^2 + x_n^2dz^2)\cdots \big)
	= g_{\bfb 1} + d\rho_1^2 + \rho_1^2 \kappa_{\bff 1}
\]
where $dy_i^2$ is shorthand for an inner product in the variables $y_i \in \bbR^{k_i}$, $g_{\bfb 1}$ denotes a Riemannian metric on $\bfb 1$ and $\kappa_{\bff 1}$ denotes
a family of wedge metrics on $\bff 1$.

In contrast, the following structure is of interest particularly in the setting of interior minimal manifolds.

\begin{defn}
The algebra of \emph{quasi fibered boundary (QFB), or simply \phistr\ vector fields} is the set
\[
	\phiV(X) = \set{V \in \eV(X) : V\rho_X \in \rho_X^2C^\infty(X)},
\]
consisting of those edge vector fields which annhilate a total boundary defining function $\rho_X$ to second order \cite{CDR}.
While $\eV(X)$ and $\wV(X)$ only depend on the fibered corners structure, $\phiV(X)$ also depends on the choice of total boundary defining function $\rho_X$ up to equivalence,
where two total boundary defining functions $\rho_X$ and $\rho'_X$ are equivalent if their ratio is basic \cite{CDR}.
In particular, $\phiV(X)$ is well-defined by a choice of fc-equivalence classes of boundary defining functions on $X$.

As with edge vector fields above, $\phiV(X)$ is a Lie subalgebra of $\bV(X)$ and forms a locally free sheaf of constant rank, defining the \emph{\phistr\ tangent bundle} (which is again 
canonically isomorphic to $TX$ in the interior) by
\begin{equation}
	\phiT X \to X, \quad \text{where} \quad \phiV(X) = C^\infty(X; \phiT X).
	\label{E:phi_T}
\end{equation}
In local standard form coordinates a local frame for $\phiT X$ is given by
\begin{equation}
\begin{gathered}
	v_1\pd{y_1},\,
	x_1v_1\pd{x_1}, \,
	v_2\pd{y_2},\,
	v_2\pns{x_2\pd{x_2} - x_1 \pd{x_1}},\,
	\ldots,\,
	v_n\pd{y_n},\,
	v_n\pns{x_n\pd{x_n} - x_{n-1}\pd{x_{n-1}}},\,
	\pd{z},
	\\ v_k = x_k\cdots x_n.
\label{E:vf_phi}
\end{gathered}
\end{equation}
Note that the basis element $v_1 x_1 \tpd{x_1}$ can be replaced by $v_1x_k \tpd{x_k}$ for any $k$ by taking appropriate linear combinations.
\label{D:phi_str}
\end{defn}
A \phistr\ metric on $X$ is by definition a (complete) Riemannian metric on $X^\circ$ induced by a uniformly bounded inner product on $\phiT X$.
Note that such a metric need not be smooth up to the boundary.
When the fibers of the maximal hypersurfaces are trivial, a \phistr\ metric is known as a `quasi-asymptotically conical' (QAC) metric \cite{CDR, DM}.
A particularly nice model \phistr\ metric in local standard form coordinates is one of the form
\[
\begin{gathered}
	\frac{dv_1^2}{v_1^4}  +\frac{1}{v_1^2}\pns{dy_1^2 + dx_1^2 + x_1^2\big(dy_2^2 + dx_2^2 + \cdots +x_{n}^2(dy_n^2 + x_n^2dz^2)\cdots \big)}
	\\ = \frac{dv_1^2}{v_1^4} + \sum_{i=1}^{n-1} \frac{dx_i^2}{v_i^2} + \sum_{i=1}^n \frac{dy_i^2}{v_i^2} + dz^2.
\end{gathered}
\]
where again
$dy_i^2$ is shorthand for an inner product in the variables $y_i \in \bbR^{k_i}$.
Note that such a metric can support additional terms of the form $\tfrac{dv_i^2}{v_i^4}$, and in particular restricts to a \phistr\ metric
on any interior fiber of a boundary hypersurface.

\subsection{Products of \texorpdfstring{\phistr\ }{Phi\ }structures} \label{S:geom_phi}

We first consider the product of \phistr\ structures on interior minimal manifolds $X$ and $Y$ with fibered corners, proving
an isomorphism between $\phiT X \oplus \phiT Y$ and $\phiT (X\ttimes Y)$.
In fact we give two proofs, one based on direct analysis of \phistr-tangent bundles using 
Proposition~\ref{P:phi_splitting} below, which decomposes $\phiT X$ with respect
to the bundle structure of the compressed projection $\tbfib H : \Np H \to C_X(\bfb H)$ of Lemma~\ref{L:main_lemma_normal}.
The second proof is based on an analysis of \phistr\ metrics taken to be in a convenient local form.

The \phistr\ structure on $X\ttimes Y$ is well-defined by the fc-equivalence class of boundary defining functions given by local rational combinations of representative 
boundary defining functions on $X$ and $Y$, as discussed in \S\ref{S:fibcorn_product}.
However, the following characterization of total boundary defining functions on $X\ttimes Y$ is conceptually satisfying.
\begin{prop}
Let $\rho_X$ and $\rho_Y$ be representative total boundary defining functions on $X$ and $Y$. 
Then the reciprocal $p$-sum
\[
	\rho_{X,Y,p} = (\rho_X^{-p} + \rho_Y^{-p})^{-1/p} = 
	\frac{\rho_X\,\rho_Y}{(\rho_X^p + \rho_Y^p)^{1/p}}
\]
is a representative total boundary defining function on $X\ttimes Y$ for any $1 \leq p \in \bbR$.
\label{P:product_tbdf}
\end{prop}
\begin{proof}
Consider a neighborhood in $X\ttimes Y$ of maximal depth, meeting a maximal totally ordered chain of boundary hypersurfaces, 
which without loss of generality (exchanging the role of $X$ and $Y$ if necessary) contains hypersurfaces of the form $\tlift{G}{Y}$ but none of the form $\tlift{X}{H}$.
Denote by $\rho_G$ and $\rho'_H$ representative boundary defining functions for $G \in \M 1(X)$ and $H \in \M 1 (Y)$, respectively,
and denote by $g = \rho_G \pa_{\rho_G}$ and $h = \rho'_H \pa_{\rho'_H}$ the generators of the associated monoids.

We claim that in such a neighborhood, $\rho_X = \prod_{G \in \M 1(X)} \rho_G$ is a
representative total boundary defining function on $X\ttimes Y$, while $\rho_X/\rho_Y = \prod_G
\rho_G \prod_H {\rho'}_H^\inv$ is basic.
Indeed, both are rational combinations, and $\bd \rho_X$ pairs to $1$ with every generator $g + h$ or $g + 0$ of
the associated monoid characterizing this neighborhood of $X\ttimes Y$, so it is locally a total boundary defining function.
Likewise, $\bd (\rho_X/\rho_Y)$ pairs to $1$ with every generator of the form $h + 0$ and to $0$ with every generator of the form $g + h$, so it is locally
the product of boundary defining functions for faces of the form $\tlift{G}{Y}$.
On any such face $\rho_X/\rho_Y$ vanishes identically, while at any face of the form $\tlift{G'}{H}$ with $Y < H$, the boundary defining functions for $\tlift{G}{Y}$ 
are all basic since $\tlift{G}{Y} < \tlift{G'}{H}$, thus $\rho_X/\rho_Y$ is basic.

The result now follows from the observation that
\[
	\frac{\rho_X}{\rho_{X,Y,p}} = \frac{(\rho_X^p + \rho_Y^p)^{1/p}}{\rho_Y} = \pns{\bpns{\tfrac{\rho_X}{\rho_Y}}^{p} + 1}^{1/p}
\]
is a smooth function of $\rho_X/\rho_Y$, hence basic, so $\rho_X$ and $\rho_{X,Y,p}$ are locally equivalent for any $p$.
\end{proof}

\begin{prop}
Let $G \in \M 1(X)$ where $X$ is an interior minimal manifold with fibered corners, with normal model $\Np G$. 
Then the differential of the compressed projection $\tbfib G : \Np G \to C_X(\bfb G) = C_\tmin(\bfb G)$ of Lemma~\ref{L:main_lemma_normal} extends by continuity from the interior to a short exact sequence
\begin{equation}
	0 \to \phiT \bff G \to \phiT \Np G \to \phiT C_\tmin(\bfb G) \to 0
	\label{E:phiT_exact_seq}
\end{equation}
of vector bundles over $G\times 0$.
\label{P:phi_splitting}
\end{prop}
\begin{proof}
We verify this in standard form local coordinates 
assuming $G = G_k$,
in which case $\tbfib k$ takes the form
\[
	(y_1,x_1,y_2,x_2, \ldots, y_k, x_k, \ldots, y_n,x_n, z) \mapsto (y_1,x_1,\ldots,y_k,\, t)
	\quad t = x_k\cdots x_n.
\]
By direct computation, the basis \eqref{E:vf_phi} 
transforms as follows:
\[
\begin{aligned}
	{x_1\cdots x_n} x_1\npd{x_1} &\mapsto {x_1\cdots x_{k-1}t} x_1\npd{x_1},	
	\\ {x_i \cdots x_{n}} \npd{y_i} &\mapsto {x_i \cdots x_{k-1}t} \npd{y_i}, \quad i \leq k
	\\{x_i \cdots x_{n}}(x_i\npd{x_i} - x_{i-1}\npd{x_{i-1}})  &\mapsto {x_i \cdots x_{k-1}}t(x_i\npd{x_i} - x_{i-1}\npd{x_{i-1}}),
	   \quad i < k,
	\\{x_k\cdots x_n}(x_k \npd{x_k} - x_{k-1}\npd{x_{k-1}}) &\mapsto t(t \npd{t} - {x_{k-1}}\npd{x_{k-1}}),
	\\ {x_{k+1}\cdots x_n} (x_{k+1}\npd{x_{k+1}} - x_k \npd{x_k}) &\mapsto  0,
	\\ {x_j\cdots x_{n}} \npd{y_{j}} &\mapsto 0,
	\quad j \geq k+1,
	\\ {x_j\cdots x_{n}}(x_{j}\npd{x_j} - x_{j-1}\npd{x_{j-1}}) &\mapsto 0,
	\quad j \geq k+2,
	\\ \npd z &\mapsto 0.
\end{aligned}
\]
The first four lines of the above constitute a basis for $\phiT C_\tmin(\bfb {k})$, with coordinates $(y_1,x_1,\ldots, x_{k-1},y_k,t)$ while the rest, spanning the kernel, constitute
a basis for $\phiT \bff {k}$, with coordinates $(y_{k+1}, \ldots, y_n,x_n,z)$.
\end{proof}

\begin{thm}
If $X$ and $Y$ are interior minimal fibered corners,
then the lifted projection maps $\pi_X : X\ttimes Y \to X$ and $\pi_Y : X\ttimes Y \to Y$ 
induce an isomorphism
\[
	\phiT (X\ttimes Y) \cong \pi_X^\ast \phiT X \oplus \pi_Y^\ast \phiT Y.
\]
In particular, if $g_X$ and $g_Y$ are \phistr-metrics on $X$ and $Y$, then $\pi_X^\ast g_X \oplus \pi_Y^\ast g_Y$ is a \phistr-metric on $X\ttimes Y$.
\label{T:phi_metric}
\end{thm}
\begin{proof}
The isomorphism is unique if it exists, since it extends the canonical isomorphism $T(X^\circ \times Y^\circ) \cong \pi_X^\ast TX^\circ \oplus \pi_Y^\ast TY^\circ$ over the interior.
To show existence, we proceed by induction on depth of $X\ttimes Y$, starting with the isomorphism $T(X\times Y) \cong TX\oplus TY$ in depth 0.

In general, a point of maximal depth in $X\ttimes Y$ lies in an intersection of boundary hypersurfaces which must include a minimal hypersurface, without loss of generality of the form $\tlift G Y$, 
where $G$ is a minimal boundary hypersurface of $X$; in particular $\bfb G$ has depth 0.
%
%
The projection $X\ttimes Y \to X$ induces via Lemma~\ref{L:main_lemma_normal} a diagram
\[
\begin{tikzcd}[sep=small]
	\Np \tlift{G}{Y} \ar[r]\ar[d] &\Np G \ar[d]
	\\ C(\bfb G) \ar[r] & C(\bfb G)
\end{tikzcd}
\]
where $C = C_\tmin$,
which we can arrange to be the identity along the bottom row since we may take $\pi_X^\ast \rho_{\geq G} = \rho_{\geq \tlift{G}{Y}}$, 
and the restriction to the fibers over the cone end
of which is the projection $\bff G \ttimes Y \to \bff G$.
By induction the projections induce an isomorphism $\phiT (\bff G \ttimes Y) \cong \phiT \bff G \oplus \phiT Y$ (since $\bff G \ttimes Y$ has strictly lower depth),
and in light of Proposition~\ref{P:phi_splitting} we obtain a diagram of short exact sequences
\[
\begin{tikzcd}
	0 \ar[r] & \phiT (\bff G \ttimes Y) \ar[r] \ar[d,"\cong"] & \phiT \Np \tlift{G}{Y} \ar[r] & \phiT C(\bfb G) \ar[d,equal] \ar[r] & 0
	\\ 0 \ar[r] & \phiT \bff G \oplus \phiT Y \ar[r] & \phiT \Np G \oplus \phiT Y \ar[r] & \phiT C(\bfb G) \ar[r] & 0
\end{tikzcd}
\]
in which the left and right maps are isomorphisms, inducing an isomorphism in the middle,
leading to 
the desired isomorphism
upon identifying the normal models $\Np \tlift G Y$ and $\Np G$ with neighborhoods in $X\ttimes Y$ and $X$, respectively.
\end{proof}

\begin{proof}[Alternate proof of Theorem~\ref{T:phi_metric}]
For the second proof, we use the characterization of $\phiT X$ as the unique vector bundle (up to isomorphism) which extends $TX^\circ$ and to which any \phistr\ metric
$g_X$ on $X^\circ$ extends to be nondegenerate and nonsingular.
We take advantage of the freedom to choose $g_X$ to be in a particularly convenient form.

When $X$ or $Y$ is of depth $0$, the result is trivial, and proceeding by induction on the depth of $X\ttimes Y$, we may assume that $X$ and $Y$ are of positive depth.
Letting $G$ and $H$ be minimal boundary hypersurfaces of $X$ and $Y$, it suffices by induction to prove the result near the lift of $G\times H$.  
Working locally,
we may assume $\phi_G$ and $\phi_H$ are trivial, and choose the defining \phistr\ metrics $g_X$ and $g_Y$ to have the simple form
\begin{equation}
\begin{aligned}
	g_X&= \frac{d\rho_X^2}{\rho_X^4}+ \frac{g_{\bfb G}}{\rho_X^2} + \kappa_G,  \quad \mbox{in the region} \quad \rho_X<\rho_{>G}, \\
	g_Y&= \frac{d\rho_Y^2}{\rho_Y^4}+ \frac{g_{\bfb H}}{\rho_Y^2} + \kappa_H,  \quad \mbox{in the region}\quad \rho_Y<\rho_{>H},\end{aligned}
\end{equation}
where $g_{\bfb G}$ and $g_{\bfb H}$ are smooth metrics on $\bfb G$ and $\bfb H$, while $\kappa_G$ and $\kappa_H$ are \phistr\ metrics associated to 
$\bff G$ and $\bff H$ with total boundary defining functions $\rho_{>G} = \prod_{G' > G} \rho_{G'}$ and $\rho_{>H}$.  
Since $\rho_G= \frac{\rho_X}{\rho_{>G}}$ is a boundary defining function for $G$, the condition $\rho_X<\rho_{>G}$ is equivalent to $\rho_G<1$.  
With $u_X = \rho_X^\inv$ and $u_Y = \rho_Y^\inv$, this suggests using the polar coordinates $r = \sqrt{u_X^2 + u_Y^2}$,
$\theta= \arctan \frac{u_X}{u_Y} = \arctan \frac{\rho_Y}{\rho_X}$, 
leading to the total boundary defining function $\rho := \rho_{X,Y,2} = r^\inv = (\rho_X^{-2} + \rho_Y^{-2})^{-1/2}$ on $X\ttimes Y$, with 
respect to which
the cartesian product $g$ of $g_X$ and $g_Y$ takes the form
\begin{equation}
\begin{aligned}
	g &= dr^2 + r^2\left( d\theta^2+ \sin^2\theta g_{\bfb G}+ \cos^2\theta g_{\bfb H} \right) + \kappa_G + \kappa_H \\
   	  &= \frac{d\rho^2}{\rho^4}+ \frac{1}{\rho^2}\left(d\theta^2+ \sin^2\theta g_{\bfb G}+ \cos^2\theta g_{\bfb H}\right)+ \kappa_G+ \kappa_H 
	\quad \mbox{in the region} \quad \rho< (\rho_{> G}^{-2} + \rho_{> H}^{-2})^{-1/2}
\end{aligned}
\label{qfb.2}
\end{equation} 
on $\nU{\tlift G H} \cong (\bfb G\jtimes \bfb H)\times (\bff G\ttimes \bff H) \times [0,1)_{\rho}$.
By induction, the cartesian product $\kappa_G+ \kappa_H$ is a \phistr\ metric associated to the manifold with fibered corners $\bff G\ttimes \bff H$ with total boundary defining function $(\rho_{>G}^{-2} + \rho_{>H}^{-2})^{-1/2}$, 
while $d\theta^2 + \sin^2 \theta g_{\bfb G} + \cos^2 \theta g_{\bfb H}$ is manifestly a wedge metric on $\bfb G\jtimes \bfb H = \bfb G \times \bfb H \times [0,\pi/2]_\theta$.
It follows that \eqref{qfb.2} is locally a \phistr\ metric on $X\ttimes Y$ near $\tlift G H$ with respect to the total boundary defining function $\rho$, 
and hence by induction, is a \phistr\ metric globally on $X\ttimes Y$.
\end{proof}


\subsection{Products of wedge structures} \label{S:geom_wedge}

We now turn to the product statement for wedge structures on interior maximal manifolds, in the form of Theorem~\ref{T:wedge_metric} below.
An initial attempt along the lines of the previous section would begin by observing that the map $\bfib G : \Np G \to \bfb G$
induces a short exact sequence
\[
	0 \to \rho_{<G}^\inv \wT C(\bff G) \to \wT \Np G \to \wT \bfb G \to 0
\]
which while true, turns out not to be immediately useful since it is not obvious that the lifted projections off of $X\ttimes Y$ implement the 
correct rescalings necessary to identify the rescaled bundles $\rho_{<\tlift{G}{H}}^\inv \wT \big(C(\bff G) \ttimes C(\bff H)\big)$ with $\rho_{<G}^\inv \wT C(\bff G)$ and
$\rho_{<H}^\inv \wT C(\bff H)$.

Instead, we implement a dual approach involving a somewhat unnatural map from $\Np G$ to $C(\bff G)$, assuming $G = \bff G \times \bfb G$ is a product.
Since $G$ is a product locally, and since we need not be concerned with uniqueness---the desired isomorphism $\wT (X\ttimes Y) \cong \wT X \oplus \wT Y$ being the unique extension of $T(X^\circ\times Y^\circ) \cong \pi_X^\ast TX^\circ \times \pi_Y^\ast TY^\circ$ as long as it exists---this unnaturality is tolerable.

\begin{prop}
Let $G = \bff G \times \bfb G$ be a product-type hypersurface of an interior maximal manifold $X$ with fibered corners. 
Then the compressed projection 
\[
	\Np G \to C_\tmax(\bff G),
	\quad v \mapsto (\pi_{\bff G}(v), \rho_{\leq G}(v)), \qquad \rho_{\leq G} = \prod_{G' \leq G} \rho_{G'}
\]
induces a short exact sequence
\[
	0 \to \wT \bfb G \to \wT \Np (\bff G \times \bfb G) \to \wT C_\tmax(\bff G) \to 0
\]
over the 0-section $G \subset \Np G$.
\label{P:wT_exact_seq}
\end{prop}
\begin{proof}
We compute in standard form local coordinates supposing $G = G_k$, in which the map is given by 
\[
	(y_1,x_1,\ldots,y_k,x_k,y_{k+1},\ldots,y_n,x_n,z) \mapsto (t,y_{k+1},x_{k+1},\ldots, y_n,x_n,z)
	\quad t = x_1\cdots x_k.
\]
The basis \eqref{E:vf_wedge} is transformed as follows:
\[
\begin{aligned}
	\tfrac 1{x_1\cdots x_{j-1}} \npd{x_j} &\mapsto  
		x_{j+1}\cdots x_n \pa_t, \quad j \leq k-1,
	\\\tfrac 1{x_1\cdots x_{j-1}} \npd{y_j} &\mapsto 0, 
		\quad j \leq k,
	\\ \tfrac 1{x_1\cdots x_{k-1}} \npd{x_k} &\mapsto \npd t,
	\\ \tfrac 1{x_1\cdots x_{j-1}} \npd{x_j} &\mapsto \tfrac 1{tx_{k+1}\cdots x_{j-1}} \npd{x_j}, \quad j \geq k+1,
	\\ \tfrac 1{x_1\cdots x_{j-1}} \npd{y_j} &\mapsto \tfrac 1{tx_{k+1}\cdots x_{j-1}} \npd{y_j}, \quad j \geq k+1,
	\\ \tfrac 1{x_1\cdots x_{n}} \npd{z} &\mapsto \tfrac 1{tx_{k+1}\cdots x_{n}} \npd{z}.
\end{aligned}
\]
The right hand sides of the first two lines, constituting the kernel of the map over $G_{k} = \set{x_k=0}$, form a basis for $\wT \bfb k$
while the right hand side of the remaining lines form a basis for $\wT C_\tmax(\bff G)$.
\end{proof}

\begin{thm}
For $X$ and $Y$ interior maximal manifolds with fibered corners, the lifted projections induce an isomorphism
\begin{equation}
	\wT(X\ttimes Y) \cong \pi_X^\ast \wT X\oplus \pi_Y^\ast \wT Y,
	\label{E:wedge_tangent_prod}
\end{equation}
In particular, if $g_X$ and $g_Y$ are wedge metrics on $X$ and $Y$,
respectively, then $\pi_X^\ast g_X \oplus \pi_Y^\ast g_Y$ is a wedge metric on
$X\ttimes Y$.
\label{T:wedge_metric}
\end{thm}
\begin{proof}
Such an isomorphism is unique if it exists, being the continuation of the standard isomorphism 
$T(X^\circ\times Y^\circ) \cong TX^\circ \oplus TY^\circ$
over the interior, which also constitutes the result when $X$ and $Y$ both have depth 0.
Proceeding then by induction on the depth of $X\ttimes Y$, it suffices to consider a point in $X\ttimes Y$ of maximal depth,
lying in a maximal boundary hypersurface, which without loss of generality we may assume 
has the form $\tlift G Y $ where $G \in \M 1 (X)$ is a maximal boundary hypersurface of $X$.

Restricting attention to a sufficiently small neighborhood, we may assume $G = \bff G \times \bfb G$ and $\tlift{G}{Y} = (\bff G) \times (\bfb G\ttimes Y)$
are products, and consider via Proposition~\ref{P:wT_exact_seq}
the compressed projection diagram 
\[
\begin{tikzcd}[sep=small]
	\Np \tlift{G}{Y} \ar[r] \ar[d] & \Np G \ar[d]
	\\ C(\bff G) \ar[r] & C(\bff G)
\end{tikzcd}
\]
where $C = C_\tmax$, and which we can assume is the identity along the bottom row
since we may take $\rho_{\leq \tlift G Y} = \pi_X^\ast \rho_{\leq G}$ by maximality.
Over the cone end the diagram is fibered, given by the projection $\bfb G \ttimes Y \to \bfb G$ on fibers.
Since $\bfb G\ttimes Y$ has lower depth, Proposition~\ref{P:wT_exact_seq}
and induction give a diagram
\[
\begin{tikzcd}
	0 \ar[r] & \wT (\bfb G\ttimes Y) \ar[r]\ar[d,"\cong"] & \wT\Np\tlift {G}{Y} \ar[r] & \wT C(\bff G) \ar[d, equal] \ar[r] & 0
	\\0 \ar[r] &  \wT \bfb G\oplus \wT Y \ar[r] &\wT \Np G \oplus \wT Y \ar[r] & \wT C(\bff G) \ar[r] & 0
\end{tikzcd}
\]
of short exact sequences, in which the left and right hand vertical maps are isomorphisms, assembling to the desired isomorphism 
in the middle upon identifying $\Np G$ and $\Np \tlift{G}{Y}$ with suitable neighborhoods in $X$ and $X \ttimes Y$.
\end{proof}

\begin{proof}[Alternate proof of Theorem~\ref{T:wedge_metric}]
For the second proof, we use the characterization of $\wT X$ as the unique vector bundle (up to isomorphism) which extends $TX^\circ$ and to which any wedge metric
$g_X$ on $X^\circ$ extends to be nondegenerate and nonsingular.

The result is trivial when $X$ or $Y$ is of depth $0$, and we
then proceed by induction on the depth of $X\ttimes Y$.
Let $G$ and $H$ be minimal hypersurfaces of $X$ and $Y$ respectively; 
by induction it suffices to prove the result near the lift of $G\times H$.  
In particular, $\bfb G$ and $\bfb H$ are closed manifolds of depth 0, 
and, by localizing as necessary, we may assume that $\phi_G$ and $\phi_H$ are trivial, so $G = \bfb G \times \bff G$ and $H = \bfb H \times \bff H$.

In order to prove the isomorphism \eqref{E:wedge_tangent_prod} we may freely choose the defining wedge metrics $g_X$ and $g_Y$.
Hence on neighborhoods $\bfb G \times \bff G \times [0,1)$ and $\bfb H \times \bff H \times [0,1)$
we may assume that $g_X$ and $g_Y$ have the form
\begin{equation}
	g_X=d\rho_G^2+ g_{\bfb G}+ \rho_G^2\kappa_G  
	\quad \mbox{and} \quad 
	g_Y=d\rho_H^2+ g_{\bfb H}+ \rho_H^2\kappa_H
\label{wedge.2}
\end{equation}  
where 
$\kappa_G$ and $\kappa_H$ are wedge metrics on $\bff G$ and $\bff H$, respectively, and where $g_{\bfb G}$ and $g_{\bfb H}$ are smooth metrics on $\bfb G$ and $\bfb H$, respectively.
Using the polar coordinates $r= \sqrt{\rho_G^2+\rho_H^2}, \theta= \arctan \frac{\rho_G}{\rho_H}$, the cartesian product $g = \pi_X^\ast g_X + \pi_Y^\ast g_Y$ takes the form
\[
	g= dr^2 + g_{\bfb G}+ g_{\bfb H}+ r^2\left( d\theta^2+ \sin^2\theta \kappa_G+ \cos^2\theta \kappa_H \right).
\]
Clearly, this is a wedge metric on the neighborhood 
\[
	\tlift{G}{H} \times [0,1) \cong \bfb G\times \bfb H \times (\bff G\jtimes \bff H) \times [0,1) 
\]
provided we know that the geometric join $d\theta^2+ \sin^2\theta \kappa_G+ \cos^2\theta \kappa_H$ lifts to a wedge metric on $\bff G\jtimes \bff H$.
However, by our inductive assumption on the depth, $g$ is known to be a wedge metric away from $r = 0$, which in particular forces this geometric join to be 
a wedge metric on $\bff G \jtimes \bff H$.
%
\end{proof}

\section{Equivalence of minimal and maximal joins} \label{S:equiv}
As noted previously, the ordered products $X_\tmax \ttimes Y_\tmax$ and $X_\tmin \ttimes Y_\tmin$ are generally not diffeomorphic.
%
By contrast, in this section we prove a rather remarkable diffeomorphism between the minimal and maximal joins $X\jtimes_\tmax Y$ and $X\jtimes_\tmin Y$.
Note that since $X \jtimes_\tmax Y$ is interior minimal while $X \jtimes_\tmin Y$ is interior maximal, we don't expect (nor do we obtain) a natural
fibered corners isomorphism from one to the other; however it happens that the diffeomorphism does interwtine all the boundary fibrations (just not the interior fibration);
in particular 
it follows that these two extremal versions of the join product have the same underlying 
stratified space, namely the topological join of the two associated stratified spaces.

\begin{thm}
Let $X$ and $Y$ be manifolds with fibered corners with total boundary defining functions $\rho_X = \prod_{G \in \M 1(X)} \rho_G$ and $\rho_Y = \prod_{H \in \M 1(Y)} \rho_H$. 
Then the b-map
\begin{equation}
	(X \times \bbR_+)\times (Y \times \bbR_+) \to (X \times \bbR_+) \times (Y\times \bbR_+),
	\quad \big((x,t),(y,s)\big) \mapsto \big((x,s\rho_Y(y)), (y,t\rho_X(x))\big)
\label{E:join_cone_map}
\end{equation}
lifts to a b-map (not a fibered corners morphism)
\[
	C_\tmax(X) \ttimes C_\tmax(Y) \to C_\tmin(X) \ttimes C_\tmin(Y)
\]
the restriction of which to the principal boundary hypersurface is a diffeomorphism
\begin{equation}
	X\jtimes_\tmax Y \cong X \jtimes_\tmin Y
	\label{E:join_diffeo}
\end{equation}
In addition, \eqref{E:join_diffeo} is fibered over boundary hypersurfaces, and constitutes a fibered corners isomorphism 
$X\jtimes_\tmax Y\cong (X\jtimes_\tmin Y)_\tmin$.

In particular, it follows from Corollary~\ref{C:join} that the stratified space associated to $X\jtimes_\tmin Y$ is also the topological join $\strat X \star \strat Y$
of the stratified spaces associated to $X$ and $Y$.
\label{T:equiv_joins}
\end{thm}
\begin{rmk}
Note that the fibered corners structure on $C_\tmin(X)$ or $C_\tmax(X)$ does not make use of the interior fibration on $X$ itself.
Viewing the joins as blow-ups of $X \times Y \times [0,1]$, the isomorphism \eqref{E:join_diffeo} identifies 
boundary hypersuraces in a manner which is somewhat surprising, namely:
\begin{itemize}
\item the lift of $X\times Y \times \set{0}$ is exchanged with the lift of $X \times Y \times \set 1$,
\item the lift of $G\times H \times [0,1]$ for $G \in \M 1(X)$ and $H \in \M 1(Y)$ is sent to itself; however,
\item the lift of $G \times Y \times \set 0$ is exchanged with the lift of $G \times Y \times [0,1]$, and
\item the lift of $X \times H \times \set 1$ is exchanged with the lift of $X \times H \times [0,1]$.
\end{itemize}
\end{rmk}

\begin{proof}
Enumerate the hypersurfaces of $X$ and $Y$ by $G_1,\ldots,G_m$ and $H_1,\ldots,H_n$, respectively, and suppose that they are ordered in the sense that $G_i < G_j$ 
implies $i < j$ and similarly for $H_i$. 
On $C_{\text{max/min}}(X)$ (and its product with $C_{\text{max/min}}(Y)$) denote by $g_i$ the monoid generator associated to the boundary hypersurface $G_i \times \bbR_+$ and by $\xi$ the generator associated to the 
boundary hypersurface $X \times 0$.
Likewise,
on $C_{\text{max/min}}(Y)$ (and its product with $C_{\text{max/min}}(X)$) denote by $h_i$ the monoid generator associated to the boundary hypersurface $H_i \times \bbR_+$ and by $\eta$ the generator associated to the
boundary hypersurface $Y \times 0$.
On the unresolved product $C_\tmax(X)\times C_\tmax(Y)$, the map \eqref{E:join_cone_map} is associated to monoid homomorphisms generated by
\begin{equation}
	\xi \mapsto \eta, 
	\quad \eta \mapsto \xi,
	\quad g_i \mapsto g_i + \eta,
	\quad h_j \mapsto \xi + h_j.
	\label{E:swindle_generators}
\end{equation}
We proceed to show that the map lifts to the ordered products.
The ordered product of the domain, $C_\tmax(X) \ttimes
C_\tmax(Y)$ is characterized as a blow-up via Theorem~\ref{T:lifting_b-maps}
by monoids generated by maximal ordered chains of sums
$a + b$ where
\[
\begin{gathered}
	a \in \set{\xi < g_1 < \cdots < g_m < 0},
	\\ b \in \set{\eta < h_1 < \cdots < h_n < 0}.
\end{gathered}
\]
Such a chain must begin with $\xi + \eta$, and may have terms of the form $g_i + \eta$ or $\xi + h_j$ but not both, and also of the form $g_i + 0$ or $0 + h_j$ but not both.
Though it is not really a generator, it is useful to imagine $0 + 0$ at the end of the chain.
On the other hand,
the ordered product of the range, $C_\tmin(X)\ttimes C_\tmin(Y)$ is characterized by monoids generated by maximal ordered chains of sums $a + b$ where
\[
\begin{gathered}
	 a \in \set{0 < g_1 < \cdots < g_m < \xi}
	\\ b \in \set{0 < h_1 < \cdots < h_n < \eta}.
\end{gathered}
\]
Such an ordered chain must end with $\xi + \eta$, and may have terms of the form $g_i + \eta$
or $\xi + h_j$ but not both, and also of the form $g_i +
0$ or $0 + h_j$ but not both.
Again, though it is not a generator, it is useful to imagine $0 + 0$ at the beginning of the chain.

To each maximal ordered chain of the source $C_\tmax(X)\ttimes C_\tmax(Y)$, we associate a maximal ordered chain of the target $C_\tmin(X) \ttimes C_\tmin(Y)$
by exchanging $\xi$ with $0$ and $\eta$ with $0$ in the list of generators (in particular $\xi + \eta$ is exchanged with $0 + 0$ at the start or end), and we claim that the monoid homomorphism generated by \eqref{E:swindle_generators} induces
a homomorphism from the monoid generated by this source chain into (but not onto) the monoid generated by the target chain.
Indeed, the generators map as
\[
\begin{aligned}
	\xi + \eta &\mapsto \xi + \eta 
	& g_i + h_j &\mapsto g_i + h_j
	\\ \xi + 0 &\mapsto 0 + \eta
	& 0 + \eta &\mapsto \xi + 0
	\\ g_i + 0 &\mapsto g_i + \eta
	& 0 + h_j &\mapsto \xi + h_j
	\\ g_i + \eta &\mapsto (g_i + 0) + (\xi + \eta)
	& \xi + h_i &\mapsto (0+ h_i) + (\xi + \eta)
\end{aligned}
\]
It follows from Theorem~\ref{T:lifting_b-maps} that \eqref{E:join_cone_map} lifts to a b-map $C_\tmax(X) \ttimes C_\tmax(Y) \to C_\tmin(X)\ttimes C_\tmin(Y)$.
Note that the monoid homomorphism is not an isomorphism (it is injective but not surjective) owing to the last line; in particular the lifted map is not b-normal.

To understand what the map looks like when restricted to $X \jtimes_\tmax Y \subset  C_\tmax(X)\ttimes C_\tmax(Y)$, recall that the b-normal monoids of a boundary hypersurface
are the quotient of those from the ambient space by the generator of the said hypersurface, which is this case is $\xi + \eta$. 
Thus the monoids of $X\jtimes_\tmax Y$ are generated by ordered chains as described above with the omission of $\xi + \eta$ at the start, while the monoids of $X\jtimes_\tmin Y$ are generated by ordered chains as for $C_\tmin(X)\ttimes C_\tmin(Y)$ with the omission of $\xi + \eta$ at the end.
Moreover, the induced map on the quotient monoids is generated as follows:
\begin{equation}
	[g_i + \eta] \mapsto [(g_i + \eta) + \xi] = [g_i + 0],
	\quad [\xi + h_j] \mapsto [\eta + (\xi + h_j)] = [0 + h_j]
	\quad [0 + \eta] \mapsto [\xi + 0]	
	\quad [\xi + 0 ] \mapsto [0+ \eta]	
	\label{E:swindle_generators_on_join}
\end{equation}

Thus each monoid of $X\jtimes_\tmax Y$ is mapped isomorphically to a related monoid of $X \jtimes_\tmin Y$ obtained by exchanging $\xi$ with $0$ and $\eta$ with $0$.
In particular it follows that $X\jtimes_\tmax Y \to X \jtimes_\tmin Y$ is simple, b-normal and ordered if the domain is made interior minimal (or the target maximal).

To see that it is also fibered (and to check that the diffeomorphism on the interior extends up to the boundary faces, which does not follow automatically from 
the above), we resort to using coordinates.
Before recording the general result, it is illustrative to consider some explicit examples setting $n = m = 2$ for simplicity.
Begining with standard form coordinates 
$(t, y_1,x_1,y_2,x_2,y_3)$ on $C_\tmax(X)$ and $(s,y'_1,x'_1,y'_2,x'_2,y'_3)$ on $C_\tmax(Y)$
taken so that the $x_i$ and $x'_i$ agree with the boundary defining functions used to construct the map,
standard form coordinates on $C_\tmax(X)\ttimes C_\tmax(Y)$ associated to the chain 
\[
	\xi + \eta < \xi + h_1 < g_1 + h_1 < g_1 + h_2 <  g_2 +h_2 < g_2 + 0 <( 0 + 0)
\]
(with notation as above)
are given by 
\begin{equation}
	(s, y'_1, \tfrac{t}{s}, y_1, \tfrac{sx'_1}{t}, y'_2, \tfrac{tx_1}{sx'_1}, y_2, \tfrac{sx'_1x'_2}{tx_1}, y'_3, \tfrac{tx_1x_2}{sx'_1x'_2}, y_3)
\label{E:max_cone_prod_coords}
\end{equation}
(Recall that the exponents of the rational boundary defining functions of the blow-up are the coefficients of the basis vectors which are dual
to the basis defined by the chain above.)
In a similar manner, with coordinates
$(y_1,x_1,y_2,x_2,y_3, x_3, y_{4},\tau)$ on $C_\tmin(X)$ and $(y'_1,x'_1,y'_2,x'_2,y'_3,x'_3,y'_{4},\sigma)$ on $C_\tmin(Y)$, standard form coordinates 
on $C_\tmin(X)\ttimes C_\tmin(Y)$ associated to the chain
\[
	(0 + 0 )< 0 + h_1 < g_1 + h_1 < g_1 + h_2 < g_2 + h_2 < g_2 + \eta < \xi + \eta
\]
are given by 
\begin{equation}
	(y'_1, \tfrac{x'_1x'_2\sigma}{x_1x_2\tau}, y_1, \tfrac{x_1x_2\tau}{x'_2\sigma}, y'_2, \tfrac{x'_2\sigma}{x_2 \tau}, y_2, \tfrac{x_2 \tau}{\sigma}, y'_3, \tfrac{\sigma}{\tau}, y_3, \sigma)
\label{E:min_cone_prod_coords}
\end{equation}
The maximal join is given by $s = 0$ in \eqref{E:max_cone_prod_coords} and the minimal join is given by $\sigma = 0$ in \eqref{E:min_cone_prod_coords}, and the map
\eqref{E:join_cone_map}, determined in coordinates by setting $\tau = sx'_1x'_2$ and $\sigma = tx_1 x_2$, perfectly identifies $s = 0$ in \eqref{E:max_cone_prod_coords}
with $\sigma = 0$ in \eqref{E:min_cone_prod_coords}, in particular identifying the boundary fibrations which are given by projections since the coordinates are in standard form.

For another example, 
standard form coordinates on $C_\tmax(X)\ttimes C_\tmax(Y)$ associated to the chain 
\[
	\xi + \eta < \xi + h_1 < \xi + h_2 < \xi + 0 < g_1 + 0 < g_2 + 0 < (0+0)
\]
are given by 
\begin{equation}
	(s, y'_1, x'_1, y'_2, x'_2, y'_3, \tfrac{t}{sx'_1x'_2}, y_1, x_1, y_2, x_2, y_3)
\label{E:max_cone_prod_coords_again}
\end{equation}
whereas coordinates on $C_\tmin(X)\ttimes C_\tmin(Y)$ associated to the chain
\[
	(0+0 )< 0 + h_1 < 0 + h_2 < 0 + \eta < g_1 + \eta < g_2 + \eta < \xi + \eta
\]
are given by 
\begin{equation}
	(y'_1, x'_1, y'_2, x'_2, y'_3, \tfrac{\tau x_1x_2}{\sigma}, y_1, x_1, y_2, x_2, y_3,\sigma)
	\label{E:min_cone_prod_coords_again}
	\end{equation}
and again $\set{s = 0}$ in \eqref{E:max_cone_prod_coords_again} is identified with $\set{\sigma = 0}$ in \eqref{E:min_cone_prod_coords_again} by $\tau = sx'_1x'_2$ and
$\sigma = tx_1x_2$.

In the general setting, standard form coordinates on $C_\tmax(X)\ttimes C_\tmax(Y)$ associated to a given maximal ordered `source chain' have the form 
\begin{equation}
	(\sigma_0, \zeta_1,\sigma_1,\zeta_2, \ldots, \zeta_{m+n},\sigma_{m+n}, \zeta_{m+n+1})
	\label{E:swindle_gen_coords_source}
\end{equation}
where $(\sigma_0,\ldots,\sigma_{m+n})$ are rational boundary defining
coordinates made up from the $(s,x_1,\ldots,x_m)$ and  $(t,x'_1,\ldots,x'_n)$ with exponents given by the
coefficients of the basis which is dual to the source chain (with
$\set{\sigma_0=0}$ defining principal boundary hypersurface $X\jtimes_\tmax
Y$), and $\zeta_i$ stands for those interior coordinates ($y_j$ or $y'_j$ for some $j$) for the base of the
boundary fibration on $\set{\sigma_i = 0}$ which belong to the fiber of the
boundary fibration on $\set{\sigma_{i-1} = 0}$. 
Likewise, standard form coordinates on $C_\tmin(X)\ttimes C_\tmin(Y)$ associated to the ordered chain given by exchanging $\xi \leftrightarrow 0$ and $\eta \leftrightarrow 0$ 
in the source chain 
have a similar form
\begin{equation}
	(\zeta_1, \sigma'_1, \zeta_2, \sigma'_2, \ldots, \zeta_{m+n}, \sigma'_{m+n}, \zeta_{m+n+1}, \sigma'_{m+n+1})
	\label{E:swindle_gen_coords_target}
\end{equation}
where $(\sigma'_1,\ldots,\sigma'_{m+n+1})$ are rational boundary defining coordinates made up from $(x_1,\ldots,x_m,\sigma)$ and $(x'_1,\ldots,x'_n,\tau)$,
with $\set{\sigma'_{m+n+1} = 0}$ defining the principal hypersurface $X\jtimes_{\tmin} Y$, and with the same ordered sequence of interior coordinates $\zeta_j$. 
The form of the boundary defining coordinates is entirely algebraic, and since \eqref{E:swindle_generators_on_join} is an isomorphism, it follows that
$\tau = sx_1\cdots x_m$ and $\sigma = tx'_1\cdots x'_n$ lifts to an identification $\sigma_i = \sigma'_i$ for $1 \leq i \leq m + n$. 
The interior coordinates $\zeta_j$ are identical between the source and target, and it follows that the induced map from $\set{\sigma_0 = 0}$ in \eqref{E:swindle_gen_coords_source} to $\set{\sigma'_{m+n+1} = 0}$ in \eqref{E:swindle_gen_coords_target} is a diffeomorphism which intertwines the boundary fibrations.
\end{proof}

\appendix
\section{Tube systems and fibrations} \label{S:tubes}

When working with manifolds with corners, it is often convenient to make use of tubular neighborhoods of boundary hypersurfaces.
\begin{defn}
Let $X$ be a manifold with corners.
A \emph{tube} for $G \in \M 1(X)$ is an open neighborhood $\nU G$ of $G$
with one of the following data, which are equivalent up to possibly replacing $\nU G$ by a smaller neighborhood:
\begin{enumerate}
\item A map $\nU G \to  G\times [0,\infty)$ which is a diffeomorphism onto its image and which restricts to the identity $1 : G \subset \nU G \to G\times \set0$.
\label{I:tube_diff}
\item A retraction $r_G : \nU G \to G$ and a nondegenerate local boundary defining function $\rho_G : \nU G \to [0,\infty)$.
\label{I:tube_bdf_retr}
\item A vector field $\xi_G$ on $\nU G$ which is inward pointing at $G$.
\label{I:tube_vfield}
\end{enumerate}
\label{D:tube}
\end{defn}

The equivalence of \ref{I:tube_diff} and \ref{I:tube_bdf_retr} is given by equating the diffeomorphism with the product $r_G \times \rho_G : \nU G \to G\times [0,\infty)$,
and the equivalence of \ref{I:tube_diff} and \ref{I:tube_vfield} follows from the flow-out of $G$ by $\xi_G$ in one direction (shrinking $\nU G$ if necessary) and the pullback of $\pa_t$ on $[0,\infty)$
in the other direction.

\begin{defn}
Let $X$ be a manifold with fibered corners.
A \emph{tube system} on $X$ consists of a tube $(\nU G, \xi_G)$ for each $G \in \M 1(X)$ 
as well as tubes $(\nU {G,G'}, \xi_{G,G'})$ for $\bface{G} {\bfb {G'}} \subset \bfb{G'}$ for each $G < G'$
such that
\begin{enumerate}
\item $\nU G \cap \nU {G'} = \emptyset$ if $G \cap G' = \emptyset$; otherwise $[\xi_G, \xi_{G'}] = 0$ on $\nU G \cap \nU {G'}$,
\label{I:tube_system_commute}
\item if $G < G'$, then $\xi_{G'}$ is tangent to the fibers of $\bfib G$ at $G \cap \nU {G'}$,
\label{I:tube_system_tangent}
\item if $G < G'$, then $\bfib{G'}(\nU G \cap G') \subset \nU {G,G'}$ and $\xi_{G} \rst_{\nU G \cap G'}$ is $\bfib {G'}$-related to $\xi_{G,G'}$ on $\nU {G,G'} \subset \bfb {G'}$, and
\label{I:tube_system_compat}
\item
$\bset{(\nU {G,G'}, \xi_{G,G'}) : G < G'}$ forms a tube system on $\bfb{G'}$.
\end{enumerate}
\label{D:tube_system}
\end{defn}

Note that the retractions and boundary defining functions for a tube system on $X$ correspond to the tubes of the associated stratified space $\strat X$ as discussed in \S\ref{S:strat}.
The existence of tube systems is proved in 
\cite[Prop.~3.7]{AM} and \cite[Lem.~1.4]{DLR}, 
on which the proof of Proposition~\ref{P:rel_tube} below is based.
In it we establish the existence of tube systems relative to a given interior fibration $\bfib X : X \to \bfb X$, though since we will also characterize interior fibrations in the process, proving an Ehresmann lemma for fibered corners, we make the following (temporary) definition:
\begin{defn}
Let $X$ be a manifold with fibered corners (the fibration $\bfib X$ need not be defined).
An \emph{interior fibration} is a surjective submersion $f : X \to Y$ onto an interior maximal manifold with fibered corners $Y = Y_\tmax$ which is ordered and fibered, and which has the property that whenever $f_\sharp(G) = H \in \M 1(Y)$, the bottom row of the commutative diagram
\[
\begin{tikzcd}
	G \ar[r, "f"] \ar[d,swap, "\bfib G"] & H \ar[d, "\bfib H"]
	\\ \bfb G \ar[r, "f_H"', "\cong"] & \bfb H
\end{tikzcd}
\]
is a diffeomorphism identifying $\bfb G \cong \bfb H$.
\label{D:int_fibn}
\end{defn}

\begin{prop}
Let $X$ be a manifold with fibered corners with representative boundary defining functions $\set{\rho_G}$ for the given fc-equivalence classes and $f : X \to Y$ an interior fibration in the sense
of Definition~\ref{D:int_fibn}.
Then there exists a tube system on $X$ with the following properties:
\begin{enumerate}
\item 
\label{I:rel_tube_fiber}
$\xi_G$ is tangent to the fibers of $f$ for every $G \in \M 1(X)$
such that $f_\sharp(G) = Y \in \M 0(Y)$.
\item 
\label{I:rel_tube_bdf}
$\xi_G(\rho_G) = 1$ at all boundary hypersurfaces of $X$ for each $G \in \M 1(X)$. 
In particular, the boundary defining functions determined by the tubes $\nU G \cong G \times [0,\ve_G)$ are fc-equivalent to the $\set{\rho_G}$.
\end{enumerate}
\label{P:rel_tube}
\end{prop}
\begin{rmk}
Observe that $X \to \pt$ and $1 : X \to X$ are always interior fibrations according to Definition~\ref{D:int_fibn} and neither imposes any additional conditions on a tube system.
Thus the result with either of these interior fibrations furnishes the existence of tube systems in general.
\end{rmk}
\begin{proof}
The proof is by induction on the depth of $X$; if $X$ has depth 0 then the statement is vacuous.
Thus assume the result holds for all spaces of depth strictly less than that of $X$; in particular each $\bfb G$ is equipped with a tube system 
satisfying properties \ref{I:rel_tube_fiber} and \ref{I:rel_tube_bdf} with respect to representative boundary defining functions on $\bfb G$ and 
the 
interior fibration $\bfb G \to Y$ in case $f_\sharp(G) = Y$ or $\bfb G \stackrel \cong \to \bfb H$ in case $f_\sharp(G) = H \in \M 1(Y)$. 

We proceed to construct $\xi_G$ for each $G \in \M 1(X)$ by another induction on $\M 1(X)$ in reverse order.
For $G$ maximal, there are two cases to consider: if $f_\sharp(G) = Y$, then $NG \subset \ker df$ and we may choose $\xi_G$ near $G$ to be transversal to $G$ and tangent to the fibers of $f$.
The condition that $\xi_G(\rho_G) = 1$ at $G$ is consistent with tangency to the fibers of $f$ since $\rho_G$ and $f$ are transversal.
Note that for each $G' < G$, either $f_\sharp(G') = Y$, in which case $G' \to Y$ factors as the composition of $\bfib{G'} : G' \to \bfb {G'}$ with $f_{G'} : \bfb {G'} \to Y$, and taking $\xi_G$ tangent to the fibers
of $\bfib{G'}$ on $G'$ is consistent with $\xi_G$ being tangent to the fibers of $f$; otherwise
$f_\sharp(G') = H \in \M 1(Y)$ and in this case $\bfib {G'} : G' \to \bfb {G'}$ factors as the composition of $f: G' \to H$ with $\bfib H : H \to \bfb{H} \cong \bfb{G'}$, so taking $\xi_G$
tangent to the fibers of $f$ makes it automatically tangent to the fibers of $\bfib{G'}$ at $G'$.
Either case is consistent with the condition that $\xi_G(\rho_G) = 1$ over $G'$ since $\rho_G$ is transverse to $\bfib {G'}$ for $G' < G$ as well as to $f$ as noted above.
Note that the flow of $\xi_G$ just constructed preserves the distribution $\ker df$ by integrability of the latter.
If on the other hand $f_\sharp(G) = H \in \M 1(Y)$, then we may simply take $\xi_G$ to be transversal to $G$, tangent to the fibers of $\bfib{G'}$ for each $G' < G$, as well as 
$f$-related to some vector normal field to $H$; this again ensures that the flow of $\xi_G$ preserves $\ker f$ by the identity $f_\ast([\xi_G, \eta]) = [f_\ast \xi_G, f_\ast \eta] = 0$
for $\eta$ tangent to the fibers of $f$.
Again the condition that $\xi_G(\rho_G) = 1$ at every $G' \in \M 1(X)$ can be imposed since $\rho_G$ is transverse to $\bfib{G'}$ for $G' < G$. 

Proceeding now by induction, suppose that $\xi_{G'}$ has been chosen for all $G' > G$ to satisfy the conditions of Definition~\ref{D:tube_system} as well as conditions \ref{I:rel_tube_fiber} and \ref{I:rel_tube_bdf} above.
We first define $\xi_G \rst_{G'}$ for each $G' > G$ to be a lift by $\bfib{G'}$ of the normal vector field $\xi_{G,G'}$ for $\bface{G} {\bfb{G'}} \subset \bfb{G'}$;
such a lift is automatically tangent to the fibers of $f$ in case $f_\sharp(G) = Y$ since $f : G' \to Y$ factors through $\bfib{G'}$ in this case and $\xi_{G,G'}$ is already tangent to the fibers
of $\bfb{G'} \to Y$.
Likewise, since $\rho_G \rst_{G'}$ is the pull-back of a boundary defining function $\rho_{G,G'}$ for $\bface{G} {\bfb {G'}} \subset \bfb{G'}$ and we may assume that $\xi_{G,G'}(\rho_{G,G'}) = 1$,
the condition that $\xi_G(\rho_G) = 1$ is satisfied here.
Note that such lifts may be chosen consistently over $G' \cap G''$ by Definition~\ref{D:tube_system}.\ref{I:tube_system_compat} applied to $\xi_{G,G'}$ and $\xi_{G,G''}$. 
We then extend $\xi_G$ over the tubular neighborhoods $\nU {G'}$ of each $G' > G$ by flowing out along $\xi_{G'}$; by Definition~\ref{D:tube_system}.\ref{I:tube_system_commute} these flow-outs are consistent on $\nU {G'} \cap \nU {G''}$. 
Since the $\xi_{G'}$ preserve $\ker f$, the extensions of $\xi_G$ remain tangent to the fibers of $f$ if applicable. 
Finally, away from the neighborhoods $\nU {G'}$ for $G' > G$ we 
proceed as in the base case, taking $\xi_G$ to be normal to $G$, tangent to the fibers of $f$ if $f_\sharp(G) = Y$ (and $f$-related to a normal vector field
for $f_\sharp(G) \in \M 1(Y)$ otherwise), tangent to the fibers of $\bfib{G'}$ for $G' < G$, and such that $\xi_G(\rho_G) = 1$ over $G' \leq G$; then we use a partition of unity to assemble $\xi_G$ globally to satisfy 
Definition~\ref{D:tube_system}.\ref{I:tube_system_commute}--\ref{I:tube_system_compat} and with respect to the vector fields already constructed, as well as conditions \ref{I:rel_tube_fiber} and \ref{I:rel_tube_bdf} above.
This completes the inner induction on $G \in \M 1(X)$, at the end of which $X$ is equipped with a tube system having the required properties,
completing the outer induction.
\end{proof}

\begin{prop}
Let $f : X \to Y$ be an interior fibration of manifolds with fibered corners in the sense of Definition~\ref{D:int_fibn}.
%
Then there exists a connection on $f : X \to Y$ with respect to which 
the lift $\wt \eta$ of any vector field $\eta \in \cV(Y)$ satisfies the following properties 
with respect to each $G \in f_\sharp^\inv(Y) \subset \M 1(X)$:
\begin{enumerate}
\item 
$\wt \eta\rst_G$ is 
$\bfib G$-related to a lift $\wt \eta_G$ of $\eta$ to $\bfb G$
by a connection on $f_G : \bfb G \to Y$ satisfying the same properties, and
\label{I:fib_strat_lift_ind}
\item $\wt \eta(\rho_G) = 0$ on a sufficiently small neighborhood of $G$ 
for some boundary defining function $\rho_G$ (coming from a tube system) in the fc-equivalence class on $X$. 
\label{I:fib_strat_lift_bdf}
\end{enumerate}
\label{P:fib_strat_lift}
\end{prop}
\begin{proof}
We fix a tube system on $X$ satisfying the conditions of Proposition~\ref{P:rel_tube} and use the associated boundary defining functions.
The proof is by induction on the depth of the fiber of $f :X \to Y$; if this fiber has depth 0 then the two conditions are vacuous and any connection suffices.
Thus we may assume by induction that such a connection exists for every interior fibration with fiber depth less than that of $f$;
in particular we may assume each $f_G : \bfb G \to Y$ is equipped with a suitable connection for each $G \in f_\sharp^\inv(Y)$, satisfying \ref{I:fib_strat_lift_bdf} with
respect to the boundary defining functions of the induced tube system on $\bfb G$.

Consider now a maximal boundary hypersurface $G$ such that $f_\sharp(G) = Y$ and consider the tubular neighborhood $\nU G \cong G \times[0,\ve)$.
The fibration $\bfib G\circ \pr_1 : G\times [0,\ve) \to \bfb G$ is transverse to $\rho_G = \pr_2: G\times[0,\ve) \to [0,\ve)$, so we may construct
a connection on $\bfib G \circ \pr_1$ whose horizontal distribution annhilates $\rho_G$ in a possibly smaller neighborhood of $G$. 
Composing this connection with the connection on $f_G : \bfb G \to Y$ gives a connection with the required properties near $G$, since the lift $\wt \eta_G$
of $\eta \in \cV(Y)$
to $\bfb G$ annihilates the boundary defining function $\rho_{G',G}$ for $\bface{G'} {\bfb G}$ for each $G' < G \in f_\sharp^\inv(Y)$, and by
properties of the tube system $\rho_{G'}$ on $\nU G \cap \nU {G'}$ is a lift of $\rho_{G',G}$ on $\nU {G'} \cap \nU G$.

Having defined the connection on a neighborhood of $G$, we now consider the
double, $\wt X_G$, of $X$ across $G$; this consists of two copies of $X$ glued
along $G$ with opposite orientations, and $f: X \to Y$ extends to a smooth map
$\wt f: \wt X_G \to Y$ (provided the gluing is performed with respect to a
choice of normal direction along $\ker df \rst_G$) which is again an interior
fibration 
according to Definition~\ref{D:int_fibn}.
The complement of $\nU G \subset X$ has a neighborhood which may be identified with an open set in $\wt X_G$, and since $\wt f : \wt X_G \to Y$ 
has strictly lower depth fibers, the connection on $\nU G$ may be combined with a suitable connection on this neighborhood by the inductive assumption.
\end{proof}

\begin{cor}
An interior fibration in the sense of Definition~\ref{D:int_fibn} is a locally trivial fiber bundle of manifolds with fibered corners. 
More precisely, for every $p \in Y$ there exists a neighborhood $\cO \ni p$ and diffeomorphism $f^\inv(\cO) \cong f^\inv(p) \times \cO$
such that for each $G \in f_\sharp^\inv(Y) \cap \M 1(X)$ the diagrams
\[
\begin{tikzcd}
	f^\inv(\cO) \ar[r, "\cong"] \ar[dr,swap, "f"] & f^\inv(p) \times \cO \ar[d, "\pr_2"] 
	\\ & \cO
\end{tikzcd}
\qquad \text{and} \qquad
\begin{tikzcd}
	f^\inv(\cO) \cap G \ar[r, "\cong"] \ar[d,swap, "\bfib {G}"] & (f^\inv(p)\cap G) \times \cO \ar[d, "\bfib G \times 1"] 
	\\ f_{G}^\inv(\cO) \ar[r, "\cong"] \ar[dr,swap,  "f_G"] & f_G^\inv(p) \times \cO \ar[d, "\pr_2"]
	\\ & \cO
\end{tikzcd}
\]
commute 
and 
the 
diffeomorphism $\cO \times f^\inv(p)\cong f^\inv(\cO)$ can be arranged to pull back $\rho_G$ to be independent of $\cO$ in a sufficiently
small neighborhood of $f^\inv(p) \cap G$
for some
representative boundary defining function $\rho_G$.

In particular, if $f : X \to Y$ is an interior fibration, then $X$ can be equipped with the fibered corners
structure in which $\bfib X = f$.
\label{C:loc_trivial}
\end{cor}
\begin{proof}
Fix a neighborhood $\cO'$
of an arbitrary point $p \in Y$ with coordinates $(u_1,\ldots,u_n) \in \bbR_+^k \times \bbR^{n-k}$ and associated coordinate vector fields $\pa_{u_1},\ldots,\pa_{u_n}$. 
Let $\wt \eta_1,\ldots,\wt \eta_n$ be lifts of the $\pa_{u_i}$ on $f^\inv(\cO')$ with respect to a connection afforded by Proposition~\ref{P:fib_strat_lift} (these need not commute).
Then the ordered flow
\[
	\big((t_1,\ldots,t_n),q\big) \mapsto \exp(t_1\wt \eta_1)\cdots \exp(t_n \wt \eta_n)q
\]
induces a diffeomorphism $\cO \times f^\inv(p) \cong f^\inv(\cO)$
where
$\cO \subset \cO'$ is a suitably small neighborhood of $p$.
By the properties of the vector fields $\wt \eta_i$ the local diffeomorphism respects the fibered corners structure on fibers in that it trivializes the boundary fibrations and representative boundary 
defining functions for $H \in f_\sharp^\inv(Y)$ as required.
\end{proof}

Corollary~\ref{C:loc_trivial} is used in \cite{KR1,KR2} to provide local trivializations for bundles of QFB manifolds, a construction that is also useful to see that the results of \cite{Ammar} 
automatically hold for QFB metrics.

\begin{cor}
Let $X$ be a manifold with fibered corners with interior fibration $\bfib X : X \to \bfb X$. 
Then standard form coordinates exist near any point of $X$.
\label{C:str_coords}
\end{cor}
\begin{proof}
Fix a tube system on $X$ as per Proposition~\ref{P:rel_tube}.
The boundary defining functions $\rho_G : \nU G \cong G \times [0,\infty) \to
[0,\infty)$ associated to the tubes furnish local boundary defining coordinates
$(x_{-n'},\ldots, x_{-1},x_1,\ldots,x_n)$ satisfying the requisite conditions,
and by local triviality and the iterated condition of the boundary fibrations, interior coordinates $(y_{-n'},\ldots,y_0,y_1,\ldots,y_{n+1})$ can be chosen 
so that the boundary fibrations correspond to projections.
\end{proof}

%
%

\bibliographystyle{amsplain}
\bibliography{fcprod}

\providecommand{\bysame}{\leavevmode\hbox to3em{\hrulefill}\thinspace}
\providecommand{\MR}{\relax\ifhmode\unskip\space\fi MR }
\providecommand{\MRhref}[2]{%
  \href{http://www.ams.org/mathscinet-getitem?mr=#1}{#2}
}
\providecommand{\href}[2]{#2}
\begin{thebibliography}{10}

\bibitem{Albin}
Pierre Albin, \emph{On the {H}odge theory of stratified spaces}, Hodge theory
  and {$L^2$}-analysis, Adv. Lect. Math. (ALM), vol.~39, Int. Press,
  Somerville, MA, 2017, pp.~1--78. \MR{3751287}

\bibitem{AGR}
Pierre Albin and Jesse Gell-Redman, \emph{The index formula for families of
  {D}irac type operators on pseudomanifolds}, arXiv preprint arXiv:1712.08513
  (2017).

\bibitem{ALMP}
Pierre Albin, \'{E}ric Leichtnam, Rafe Mazzeo, and Paolo Piazza, \emph{The
  signature package on {W}itt spaces}, Ann. Sci. \'{E}c. Norm. Sup\'{e}r. (4)
  \textbf{45} (2012), no.~2, 241--310. \MR{2977620}

\bibitem{ALMP2}
Pierre Albin, Eric Leichtnam, Rafe Mazzeo, and Paolo Piazza, \emph{Hodge theory
  on {C}heeger spaces}, J. Reine Angew. Math. \textbf{744} (2018), 29--102.
  \MR{3871440}

\bibitem{AM}
Pierre Albin and Richard Melrose, \emph{Resolution of smooth group actions},
  Spectral theory and geometric analysis, Contemp. Math., vol. 535, Amer. Math.
  Soc., Providence, RI, 2011, pp.~1--26. \MR{2560748}

\bibitem{AMN}
Bernd Ammann, Jeremy Mougel, and Victor Nistor, \emph{A comparison of the
  {G}eorgescu and {V}asy spaces associated to the {N}-body problems}, arXiv
  preprint arXiv:1910.10656 (2019).

\bibitem{Ammar}
Mahdi Ammar, \emph{Polyhomog\'{e}n\'{e}it\'{e} des m\'{e}triques compatibles
  avec une structure de {L}ie \`a l'infini le long du flot de {R}icci}, Ann.
  Inst. H. Poincar\'{e} C Anal. Non Lin\'{e}aire \textbf{38} (2021), no.~6,
  1795--1840. \MR{4327898}

\bibitem{Carron}
Gilles Carron, \emph{On the quasi-asymptotically locally {E}uclidean geometry
  of {N}akajima's metric}, Journal of the Institute of Mathematics of Jussieu
  \textbf{10} (2011), no.~1, 119--147.

\bibitem{Cheeger}
Jeff Cheeger, \emph{On the {H}odge theory of {R}iemannian pseudomanifolds},
  Geometry of the {L}aplace operator ({P}roc. {S}ympos. {P}ure {M}ath., {U}niv.
  {H}awaii, {H}onolulu, {H}awaii, 1979), Proc. Sympos. Pure Math., XXXVI, Amer.
  Math. Soc., Providence, R.I., 1980, pp.~91--146. \MR{573430}

\bibitem{CDR}
Ronan Conlon, Anda Degeratu, and Fr\'{e}d\'{e}ric Rochon,
  \emph{Quasi-asymptotically conical {C}alabi-{Y}au manifolds}, Geom. Topol.
  \textbf{23} (2019), no.~1, 29--100, With an appendix by Conlon, Rochon and
  Lars Sektnan. \MR{3921316}

\bibitem{DLR}
C.~Debord, J.-M. Lescure, and F.~Rochon, \emph{Pseudodifferential operators on
  manifolds with fibred corners}, Ann. Inst. Fourier (Grenoble) \textbf{65}
  (2015), no.~4, 1799--1880. \MR{3449197}

\bibitem{DM}
Anda Degeratu and Rafe Mazzeo, \emph{Fredholm theory for elliptic operators on
  quasi-asymptotically conical spaces}, Proc. Lond. Math. Soc. (3) \textbf{116}
  (2018), no.~5, 1112--1160. \MR{3805053}

\bibitem{FKS}
Karsten Fritzsch, Chris Kottke, and Michael Singer, \emph{Monopoles and the
  {S}en conjecture: {P}art {I}}, arXiv preprint arXiv:1811.00601 (2018).

\bibitem{GKMwedge}
Juan~B. Gil, Thomas Krainer, and Gerardo~A. Mendoza, \emph{On the closure of
  elliptic wedge operators}, J. Geom. Anal. \textbf{23} (2013), no.~4,
  2035--2062. \MR{3107690}

\bibitem{Grieser}
Daniel Grieser, \emph{Scales, blow-up and quasimode constructions}, Geometric
  and computational spectral theory, Contemp. Math., vol. 700, Amer. Math.
  Soc., Providence, RI, 2017, pp.~207--266. \MR{3748527}

\bibitem{HMM}
Andrew Hassell, Rafe Mazzeo, and Richard~B. Melrose, \emph{Analytic surgery and
  the accumulation of eigenvalues}, Comm. Anal. Geom. \textbf{3} (1995),
  no.~1-2, 115--222. \MR{1362650}

\bibitem{JoyceQALE}
Dominic Joyce, \emph{Quasi-{ALE} metrics with holonomy {${\rm SU}(m)$} and
  {${\rm Sp}(m)$}}, Ann. Global Anal. Geom. \textbf{19} (2001), no.~2,
  103--132. \MR{1826397}

\bibitem{Joycegc}
\bysame, \emph{A generalization of manifolds with corners}, Adv. Math.
  \textbf{299} (2016), 760--862. \MR{3519481}

\bibitem{JoyceBook}
Dominic~D. Joyce, \emph{Compact manifolds with special holonomy}, Oxford
  Mathematical Monographs, Oxford University Press, Oxford, 2000. \MR{1787733}

\bibitem{Kgc}
Chris Kottke, \emph{Blow-up in manifolds with generalized corners}, Int. Math.
  Res. Not. IMRN (2018), no.~8, 2375--2415. \MR{3801487}

\bibitem{Kmb}
\bysame, \emph{Functorial compactification of linear spaces}, Proc. Amer. Math.
  Soc. \textbf{147} (2019), no.~9, 4067--4081. \MR{3993798}

\bibitem{KMgen}
Chris Kottke and Richard~B. Melrose, \emph{Generalized blow-up of corners and
  fiber products}, Trans. Amer. Math. Soc. \textbf{367} (2015), no.~1,
  651--705. \MR{3271273}

\bibitem{KR2}
Chris Kottke and Fr\'{e}d\'{e}ric Rochon, \emph{{$L^2$-cohomology of
  quasi-fibered boundary metrics}}, arXiv preprint arXiv:2103.16655 (2021).

\bibitem{KR1}
\bysame, \emph{Quasi-fibered boundary pseudodifferential operators}, arXiv
  preprint arXiv:2103.16650 (2021).

\bibitem{Mazzeo}
R.~Mazzeo, \emph{{Elliptic theory of differential edge operators I}},
  Communications in Partial Differential Equations \textbf{16} (1991), no.~10,
  1615--1664.

\bibitem{MM}
Rafe Mazzeo and Richard~B. Melrose, \emph{Pseudodifferential operators on
  manifolds with fibred boundaries}, vol.~2, 1998, Mikio Sato: a great Japanese
  mathematician of the twentieth century, pp.~833--866. \MR{1734130}

\bibitem{Msc}
R.~Melrose, \emph{Spectral and scattering theory for the {L}aplacian on
  asymptotically {E}uclidian spaces}, Spectral and scattering theory ({S}anda,
  1992), Lecture Notes in Pure and Appl. Math., vol. 161, Dekker, New York,
  1994, pp.~85--130. \MR{1291640}

\bibitem{MDAOMWC}
Richard Melrose, \emph{Differential analysis on manifolds with corners},
  \url{https://math.mit.edu/~rbm/book.html}, Unpublished book in progress.

\bibitem{MAPSIT}
Richard~B. Melrose, \emph{The {A}tiyah-{P}atodi-{S}inger index theorem},
  Research Notes in Mathematics, vol.~4, A K Peters, Ltd., Wellesley, MA, 1993.
  \MR{1348401}

\bibitem{Pflaum}
Markus~J. Pflaum, \emph{Analytic and geometric study of stratified spaces},
  Lecture Notes in Mathematics, vol. 1768, Springer-Verlag, Berlin, 2001.
  \MR{1869601}

\bibitem{Vasy}
A.~Vasy, \emph{Propagation of singularities in many-body scattering}, Ann. Sci.
  \'Ecole Norm. Sup. (4) \textbf{34} (2001), no.~3, 313--402. \MR{1839579}

\bibitem{Verona}
Andrei Verona, \emph{Stratified mappings---structure and triangulability},
  Lecture Notes in Mathematics, vol. 1102, Springer-Verlag, Berlin, 1984.
  \MR{771120}

\end{thebibliography}

\end{document}